\documentclass[11pt]{article}

\usepackage[margin=1.05in]{geometry}
\usepackage{amsmath,amssymb,amsthm,mathtools}
\usepackage[colorlinks=true,linkcolor=blue,citecolor=blue,urlcolor=blue]{hyperref}
\usepackage[capitalise]{cleveref}
\usepackage{enumerate}
\usepackage{comment}
\usepackage{todonotes}
\usepackage{tikz}
\usepackage{graphicx}
\usepackage[utf8]{inputenc}
\usetikzlibrary{arrows.meta,patterns,decorations.pathmorphing,calc}

\theoremstyle{plain}
\newtheorem{theorem}{Theorem}[section]
\newtheorem{proposition}[theorem]{Proposition}
\newtheorem{lemma}[theorem]{Lemma}
\newtheorem{corollary}[theorem]{Corollary}
\theoremstyle{definition}
\newtheorem{definition}[theorem]{Definition}
\theoremstyle{remark}
\newtheorem{remark}[theorem]{Remark}

\theoremstyle{plain}
\newtheorem{mainthm}{Theorem}
\newcommand{\thmletter}[1]{\renewcommand{\themainthm}{#1}}
\crefname{mainthm}{Theorem}{Theorems}
\Crefname{mainthm}{Theorem}{Theorems}

\newcommand{\R}{\mathbb{R}}
\newcommand{\N}{\mathbb{N}}
\newcommand{\Z}{\mathbb{Z}}
\newcommand{\Prob}{\mathbb{P}}
\newcommand{\E}{\mathbb{E}}
\newcommand{\Wien}{\mathbb{W}}
\newcommand{\dist}{\mathrm{dist}}
\newcommand{\rad}{\operatorname{rad}}
\newcommand{\Poi}{\mathrm{Poi}}
\newcommand{\Exp}{\mathrm{Exp}}

\newcommand{\Qin}{Q_{\mathrm{in}}}
\newcommand{\Qext}{Q_{\mathrm{ext}}}
\newcommand{\etaB}{\eta_c^{\mathrm{B}}}
\newcommand{\Vol}{\mathcal V}

\newcommand{\unc}{\mathrm{u}}
\newcommand{\iso}{\mathrm{i}}

\newcommand{\olap}{\mathfrak{o}}
\newcommand{\blk}{\mathfrak{b}}
\newcommand{\Qst}{Q^{*}}
\newcommand{\cell}[2]{(#1,#2)}
\newcommand{\scell}[2]{(#1,#2)^{*}}

\title{Motion helps the contact process: survival below\\ the percolation threshold on Poisson
Brownian motions}
\author{Peter Gracar\thanks{School of Mathematics, University of Leeds,
United Kingdom. ORCID:
\href{https://orcid.org/0000-0001-8340-8340}{0000-0001-8340-8340}.}}
\date{August 25, 2026}

\makeatletter
\let\extrasenglish\@empty
\let\noextrasenglish\@empty
\makeatother

\begin{document}

\maketitle

\begin{abstract}
We consider the SIS contact process on a Poisson system of independent Brownian motions in $\R^d$
of intensity $\eta$. Two particles are in contact when their distance is at most $2r$. An infected
particle transmits the infection to a susceptible particle in contact with it at rate
$\lambda\in(0,\infty]$, and recovers at rate $\mu\in(0,\infty]$, after which it is susceptible
again. When $\lambda=\infty$, recovery acts only while a particle is isolated. Let $\etaB$ be the critical intensity of the
static Boolean model, below which the contact graph has only finite components at every fixed
time. We show that for $d\ge2$, with motion, at every intensity strictly below $\etaB$, the infection survives
and spreads at positive speed once the recovery rate $\mu$ is small enough, depending on the
intensity; that is, started from a single infected particle, at time $t$ there is an infected
particle at distance of order $t$ from the origin. In particular there exists $\mu_\dagger>0$ such
that $\eta_c(\infty,\mu)<\etaB$ for every $\mu<\mu_\dagger$, and as
$\mu\downarrow0$ the set of intensities at which the infection survives extends to all of
$(0,\etaB)$. We also show that the
critical density is positive at every recovery rate, with a lower bound that does not depend on the
infection rate, and that for $d\ge2$, at every intensity and every finite infection rate, the
infection survives and spreads at positive speed once $\mu$ is small enough.
\end{abstract}
 
\tableofcontents
\bigskip

\section{Introduction and main results}
\label{sec:intro}

\subsection{Infections in a moving population}
\label{sec:intro-background}

We study an epidemic carried by a population of particles that move. This differs from a contact
process on a fixed graph in several ways. The graph along which the infection travels is random and
evolves in time and its evolution is correlated over arbitrarily long time spans through the
trajectories of the particles. Moreover, the motion is at the same time the transport mechanism of
the infection and the source of the spatial inhomogeneities to which the contact process is sensitive.
The standard tools for the contact process, i.e.\ a graphical construction on a fixed lattice, a
comparison with oriented percolation and a reproduction estimate, are either not available or have
to be constructed from scratch.

The systematic study of such models goes back to Kesten and Sidoravicius
\cite{KestenSidoravicius05,KestenSidoravicius06,KestenSidoravicius08}. They considered a Poisson
system of independent simple random walks on $\Z^d$ in which an infection is passed whenever an
infected and a healthy particle share a site. Without recovery, they first obtained linear bounds
on the propagation of the front \cite{KestenSidoravicius05} and later a shape theorem
\cite{KestenSidoravicius08}. With recovery, the model is an SIS contact process in a moving medium,
and in \cite{KestenSidoravicius06} they established, at every fixed particle density, a phase
transition in the recovery rate between extinction and survival.

Several variants of the medium and of the transport rule have been studied since. Dauvergne and
Sly \cite{DauvergneSly23} considered infected and susceptible particles diffusing at different
rates. Baldasso and Stauffer \cite{BaldassoStauffer20} considered biased random walks. Gracar and
Stauffer \cite{GracarStauffer19a} considered random walks among random conductances. Drewitz,
Gallo and Gracar \cite{DGG26} considered the Sierpi\'nski gasket, on which the motion is
subdiffusive. Baldasso and Teixeira \cite{BaldassoTeixeira20} replaced the independent walks by a
zero-range process and bounded the position of the front in one dimension. Variants in which
recovered particles are removed rather than returned to the susceptible state, i.e.\ SIR rather
than SIS, were studied by Dauvergne and Sly \cite{DauvergneSly22} on the lattice and by Grimmett
and Li \cite{GrimmettLi22} for Brownian particles in $\R^d$. In the continuum, the static object
underlying the model is the Boolean model. Its dynamic version, Poisson Brownian motions with a
fixed interaction radius, was studied by Peres, Sinclair, Sousi and Stauffer \cite{PSSS13}, whose
detection and percolation estimates for mobile geometric graphs are the starting point of the
local mixing estimates we prove.

All of these works share the premise that the infection spreads because the particles move, so
that a mobile population sustains an epidemic which a static population of the same density would
not. In the lattice models this premise cannot be directly tested, as the underlying graph is not
itself evolving and is assumed to be supercritical. Transmission there occurs only when two
particles share a site, so a frozen population never spreads the infection beyond its starting
site, at any density. Motion is the only mechanism available and there is nothing to compare it
against. In the continuum, the frozen medium has a percolation transition of its own, and in this
paper we study how the motion affects it.

\subsection{The effect of motion on the critical density}
\label{sec:intro-density}

The model has three parameters: the infection rate $\lambda$, the recovery rate $\mu$ and the
intensity $\eta$ of the medium. The transition in $\mu$ at fixed density and fixed $\lambda$ is the
one studied by Kesten and Sidoravicius. The complementary question is the transition in the
density at a fixed recovery rate, that is, extinction at low density and survival at high density.
This question has received attention more recently and is the subject of this paper. We write
$\eta_c(\lambda,\mu)$ for the critical density separating almost sure extinction from survival
with positive probability, and we ask where it lies and how it depends on the two rates.

In the continuum this question has a form which is not available on a lattice. At a fixed time the
contact graph of the medium is a Gilbert graph, and it has a percolation threshold. Below the
critical intensity $\etaB$ of the Boolean model, every component is finite almost surely, so an
infection in a frozen medium would be confined to the component of the particle it started on.
Above $\etaB$ there is an infinite component and no motion is needed at all. The threshold $\etaB$
therefore separates the intensities at which the static geometry can carry an infection to
infinity from those at which it cannot, and the question of whether the motion helps becomes the
question of whether $\eta_c$ is strictly below $\etaB$. On $\Z^d$ there is no such threshold, for
the opposite reason: with transmission only at coincidence, the frozen medium never percolates, at
any density, so the threshold to compare against is infinite.

The main result of this paper is that, for small recovery rates, $\eta_c$ is strictly below
$\etaB$. At every intensity below $\etaB$ the cluster of the seed is almost surely finite, so an
infection which survives there survives because the particles move. Above $\etaB$ this deduction
fails, and we explain in \Cref{sec:soft} why a bound of the form $\eta_c\le\etaB$ on its own
would not show that the motion contributes anything. We therefore prove strict inequality,
together with survival at positive speed from a single infected particle.

The question requires care when the transmission rule is instantaneous, that is, when an encounter
between an infected and a healthy particle infects the latter with certainty rather than at some
finite rate $\lambda$. Under an instantaneous rule, a recovery which occurs while the particle is
in contact with another infected particle is undone at once, and the model is degenerate unless
recovery is restricted to particles which are alone. Kesten and Sidoravicius avoided this by
allowing transmission only at the instant of a jump, and noted that their renormalisation could not
handle the alternative. The rule under which a particle heals only while isolated was introduced
and analysed by Baldasso and Stauffer \cite{BaldassoStauffer23}, who obtained local and global
survival results for independent random walks.

The density transition itself was settled for an interacting lattice medium, while the present
work was in preparation, by Baldasso, Hil\'ario and Ornelas \cite{BHO26}. For the infection with
instantaneous transmission and healing-when-isolated carried by a zero-range process on $\Z^d$ at
density $\rho$, they prove extinction at small $\rho$ and survival at large $\rho$. This gives a
critical density $\rho_c(\delta)\in(0,\infty)$ at every recovery rate $\delta$, and survival at
every positive density once $\delta$ is small. The renormalisation scheme of \Cref{sec:extinction}
is similar to theirs, differing primarily in the continuum objects of interest. The strict inequality above has a lattice counterpart, but a
degenerate one: since the static threshold on the lattice is infinite, it reads
$\rho_c(\delta)<\infty$, which is their theorem. In the continuum the statement to be improved upon
is finite, see \Cref{sec:method}.

\subsection{The particle system}
\label{sec:particles}

Fix a dimension $d\ge1$, an intensity $\eta>0$ and a radius $r>0$. Let $\mathcal P$ be a Poisson
point process on the path space $C([0,\infty);\R^d)$ with intensity measure
\begin{equation}
\label{eq:pathPPP}
  \nu_\eta(\mathrm dw)\;=\;\eta\int_{\R^d}\Wien_x(\mathrm dw)\,\mathrm dx ,
\end{equation}
where $\Wien_x$ denotes Wiener measure started at $x$. In other words, the initial positions form
a Poisson point process of intensity $\eta$ on $\R^d$, and given the initial positions each
particle performs an independent standard Brownian motion. We write $X_u(t)$ for the position at
time $t$ of a particle $u\in\mathcal P$, and
\[
  Q_K=[-K/2,K/2]^d,\qquad B(x,s)=\{y:|y-x|\le s\},
\]
\[
  \omega_d=\frac{\pi^{d/2}}{\Gamma(\tfrac d2+1)},\qquad
  \Vol=|B(0,2r)|=\omega_d(2r)^d ,
\]
so that $\Vol=2^d\omega_dr^d$. Since Wiener measure preserves Lebesgue measure, the system is
stationary in time. That is, for every fixed $t$ the point process $\{X_u(t)\}_{u\in\mathcal P}$
is again Poisson of intensity $\eta$ \cite{vdBMW97}. The system is also ergodic under spatial
translations.

We say that two particles $u\ne v$ are \emph{in contact} at time $t$ if $|X_u(t)-X_v(t)|\le2r$,
that is, if the balls of radius $r$ centred at them intersect. We say that a particle is
\emph{isolated} at time $t$ if it is in contact with no other particle. Note that
$t\mapsto\min_{v\ne u}|X_u(t)-X_v(t)|$ is continuous, so the set of times at which a fixed
particle is isolated is open.

Let $G_t$ be the \emph{Gilbert graph} at time $t$, that is, the graph on
$\{X_u(t)\}_{u\in\mathcal P}$ in which two points are joined when their distance is at most $2r$.
We write $\etaB(R,d)$ for the critical intensity of the Gilbert graph at connection radius $R$ in
$\R^d$, and abbreviate $\etaB=\etaB(2r,d)$. Recall that $\etaB$ is strictly positive in every
dimension, finite for $d\ge2$, and infinite for $d=1$ \cite{MeesterRoy96}.

\subsection{The graphical construction and the two models}
\label{sec:graphical}

The epidemic is driven by two independent families of marks and by a choice of recovery rule.
Fix $\lambda\in(0,\infty]$ and $\mu\in(0,\infty]$. Conditionally on $\mathcal P$, which is almost
surely countable, let
\begin{itemize}
\item $\mathcal R_u$, for $u\in\mathcal P$, be independent Poisson processes of rate $\mu$ on
  $[0,\infty)$, the \emph{recovery marks};
\item $\mathcal A_{u\to v}$, for ordered pairs of distinct particles $u,v\in\mathcal P$, be
  independent Poisson processes of rate $\lambda$ on $[0,\infty)$, the \emph{transmission marks},
  independent of $(\mathcal R_u)_u$.
\end{itemize}
We call $\omega=(\mathcal P,(\mathcal R_u)_u,(\mathcal A_{u\to v})_{u\ne v})$ the \emph{marked
configuration}. Formally, we fix a measurable enumeration of $\mathcal P$, for instance by
increasing distance of $X_u(0)$ from the origin, which is injective almost surely, and we read the
two families off an array $(\mathcal R^{(k)})_{k\in\N}$, $(\mathcal A^{(k,l)})_{k\ne l\in\N}$ of
independent Poisson processes of rates $\mu$ and $\lambda$, independent of $\mathcal P$.

We will use two features of this construction later. First, $\omega$ is not an independently
marked Poisson process, since the family $\mathcal A$ is indexed by ordered pairs of points rather
than by points. Hence there is no Mecke equation for $\omega$ itself, and we cannot use one. What
we use is that conditionally on $\mathcal P$ the marks are independent across particles and
across ordered pairs. Therefore every computation may be performed by conditioning on
$\mathcal P$, integrating out the marks, and applying the Mecke equation to the unmarked process
$\mathcal P$. Second, the pair $(\mathcal P,(\mathcal R_u)_u)$ alone is an independent marking of
$\mathcal P$, hence again a Poisson process on path space $\times$ mark space. This gives
the ergodicity in \Cref{prop:ergodic}.

The infection is of SIS type throughout. That is, a particle which recovers becomes susceptible
again, and may be infected any number of times. We specify the dynamics by declaring which marks
act.

\begin{itemize}
\item \emph{Transmission.} For $\lambda<\infty$, a mark of $\mathcal A_{u\to v}$ at time $t$ is
  \emph{active} if $u$ and $v$ are in contact at $t$. For $\lambda=\infty$ the family
  $(\mathcal A_{u\to v})$ is discarded and we declare instead that a transmission from $u$ to $v$
  occurs at every instant at which $u$ and $v$ are in contact.
\item \emph{Recovery.} Under the \emph{unconditional} rule, written $\mathfrak r=\unc$, every mark
  of $\mathcal R_u$ is \emph{effective}. Under the \emph{isolation} rule, written
  $\mathfrak r=\iso$, a mark of $\mathcal R_u$ at time $t$ is effective if and only if $u$ is
  isolated at $t$. The case $\mu=\infty$ is included with the convention that every instant of
  isolation is an effective recovery.
\end{itemize}

\begin{definition}[Open path]
\label{def:openpath}
Let $\lambda\in(0,\infty]$, $\mu\in(0,\infty]$ and $\mathfrak r\in\{\unc,\iso\}$. An \emph{open
path} from $(u_0,t_0)$ to $(u_n,t_{n+1})$ consists of particles $u_0,\dots,u_n$, such that
consecutive particles are distinct, and times $t_0<t_1<\dots<t_{n+1}$, such that
\begin{enumerate}[(i)]
\item for $1\le j\le n$ there is an active transmission mark from $u_{j-1}$ to $u_j$ at time
  $t_j$; when $\lambda=\infty$ this says that $u_{j-1}$ and $u_j$ are in contact at $t_j$;
\item for $0\le j\le n$ there is no effective recovery mark of $u_j$ in $(t_j,t_{j+1}]$.
\end{enumerate}
Its \emph{trace} is the c\`adl\`ag path $\gamma$ on $[t_0,t_{n+1}]$ given by $\gamma(t)=X_{u_j}(t)$
for $t\in[t_j,t_{j+1})$ and $0\le j\le n-1$, and by $\gamma(t)=X_{u_n}(t)$ for
$t\in[t_n,t_{n+1}]$. Its jumps have size at most $2r$. The path is \emph{contained} in a
space-time region $D$ if $(\gamma(t),t)\in D$ for every $t\in[t_0,t_{n+1}]$.
\end{definition}

An open path and its trace are drawn in \Cref{fig:openpath}.

Note that a path may re-enter a particle it has already used. This makes
\Cref{def:openpath} the SIS notion: a particle recovered at one time may carry the infection at a
later one. It is essential that clause (ii) refers to effective recovery marks and not to the
infected set. Had we required $u_j$ to be infected throughout $[t_j,t_{j+1}]$, the resulting
object would depend on the whole history of the process, and the locality used in
\Cref{sec:horiz,sec:vertical} would fail. As it stands, the notion is a function of the marked
configuration alone. Moreover, under the isolation rule, whether a given mark is effective depends
only on the configuration within distance $2r$ of the particle carrying it.

\begin{lemma}[Truncation]
\label{lem:truncate}
Let $u_0,\dots,u_n$ and $t_0<t_1<\dots<t_{n+1}$ be an open path, let $t\in(t_0,t_{n+1}]$, and let
$j$ be the unique index in $\{0,\dots,n\}$ with $t_j<t\le t_{j+1}$. Then $u_0,\dots,u_j$ with times
$t_0<\dots<t_j<t$ is an open path from $(u_0,t_0)$ to $(u_j,t)$.
\end{lemma}

\begin{proof}
The times are strictly increasing because $t>t_j$. Clause (i) for $1\le i\le j$ is unchanged.
Clause (ii) for $i<j$ is unchanged, and for $i=j$ it holds because
$(t_j,t]\subseteq(t_j,t_{j+1}]$.
\end{proof}

The lemma expresses the fact that extinction is absorbing. That is, an infection which is alive
at some time is alive at every earlier time.

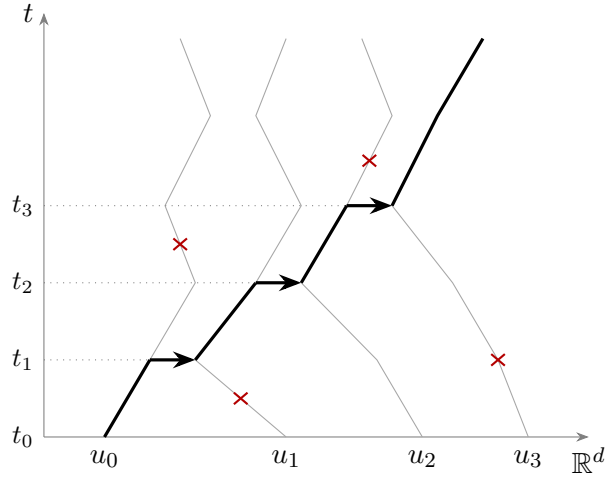
\begin{figure}[!h]
\centering
\begin{tikzpicture}[x=1cm,y=0.85cm,>=Stealth,scale=1]
  \draw[->,gray] (0.2,0) -- (0.2,6.6) node[left,black]{$t$};
  \draw[->,gray] (0.2,0) -- (7.4,0) node[below,black]{$\R^d$};
  \draw[gray!70] (1.0,0) -- (1.3,0.6) -- (1.6,1.2) -- (1.9,1.8) -- (2.2,2.4) -- (2.0,3.0)
                 -- (1.8,3.6) -- (2.1,4.3) -- (2.4,5.0) -- (2.2,5.6) -- (2.0,6.2);
  \draw[gray!70] (3.4,0) -- (2.8,0.6) -- (2.2,1.2) -- (2.6,1.8) -- (3.0,2.4) -- (3.3,3.0)
                 -- (3.6,3.6) -- (3.3,4.3) -- (3.0,5.0) -- (3.2,5.6) -- (3.4,6.2);
  \draw[gray!70] (5.2,0) -- (4.9,0.6) -- (4.6,1.2) -- (4.1,1.8) -- (3.6,2.4) -- (3.9,3.0)
                 -- (4.2,3.6) -- (4.5,4.3) -- (4.8,5.0) -- (4.6,5.6) -- (4.4,6.2);
  \draw[gray!70] (6.6,0) -- (6.4,0.6) -- (6.2,1.2) -- (5.9,1.8) -- (5.6,2.4) -- (5.2,3.0)
                 -- (4.8,3.6) -- (5.1,4.3) -- (5.4,5.0) -- (5.7,5.6) -- (6.0,6.2);
  \draw[very thick] (1.0,0) -- (1.3,0.6) -- (1.6,1.2);
  \draw[very thick] (2.2,1.2) -- (2.6,1.8) -- (3.0,2.4);
  \draw[very thick] (3.6,2.4) -- (3.9,3.0) -- (4.2,3.6);
  \draw[very thick] (4.8,3.6) -- (5.1,4.3) -- (5.4,5.0) -- (5.7,5.6) -- (6.0,6.2);
  \draw[->,very thick] (1.6,1.2) -- (2.2,1.2);
  \draw[->,very thick] (3.0,2.4) -- (3.6,2.4);
  \draw[->,very thick] (4.2,3.6) -- (4.8,3.6);
  \foreach \p in {(2.0,3.0),(2.8,0.6),(4.5,4.3),(6.2,1.2)}
     {\draw[thick,red!70!black] \p ++(-0.09,-0.09) -- ++(0.18,0.18);
      \draw[thick,red!70!black] \p ++(-0.09,0.09) -- ++(0.18,-0.18);}
  \node[below] at (1.0,-0.05) {$u_0$};
  \node[below] at (3.4,-0.05) {$u_1$};
  \node[below] at (5.2,-0.05) {$u_2$};
  \node[below] at (6.6,-0.05) {$u_3$};
  \draw[dotted,gray] (0.2,1.2) -- (1.6,1.2); \node[left] at (0.2,1.2) {\small$t_1$};
  \draw[dotted,gray] (0.2,2.4) -- (3.0,2.4); \node[left] at (0.2,2.4) {\small$t_2$};
  \draw[dotted,gray] (0.2,3.6) -- (4.2,3.6); \node[left] at (0.2,3.6) {\small$t_3$};
  \node[left] at (0.2,0) {\small$t_0$};
\end{tikzpicture}
\caption{An open path of length $3$ from $(u_0,t_0)$, drawn in bold; time runs upwards. The grey
curves are the trajectories of the particles, the bold horizontal arrows are the transmissions at
$t_1,t_2,t_3$, and the crosses are recovery marks. The path is open because no effective mark of
$u_{j-1}$ lies in $(t_{j-1},t_j]$; marks carried by a particle outside the interval during which
the path uses it, such as the mark of $u_0$ above $t_1$, are irrelevant. The trace $\gamma$ is the
bold curve together with the three jumps, each of size at most $2r$.}
\label{fig:openpath}
\end{figure}

Given a set $I_0$ of initially infected particles we set
\begin{equation}
\label{eq:Idef}
  I^{\lambda,\mu,\mathfrak r}_t=\{u\in\mathcal P:\ \text{there is an open path from }(v,0)
  \text{ to }(u,t)\text{ for some }v\in I_0\} ,
\end{equation}
and we abbreviate this to $I_t$ when the parameters are clear. At $t=0$ the right-hand side is
empty, since the times in \Cref{def:openpath} are strictly increasing, and we read it there as the
given set $I_0$. By \Cref{lem:truncate}, $I_s\ne\emptyset$ implies $I_t\ne\emptyset$ for every
$0\le t\le s$. Hence $t\mapsto\{I_t\ne\emptyset\}$ is decreasing and
$\{I_t\ne\emptyset\ \text{for all}\ t\ge0\}=\bigcap_{n\in\N}\{I_n\ne\emptyset\}$. We take
\eqref{eq:Idef} as the definition of the infected set. Nothing below needs a separate construction
of a Markov process. Moreover, at $\lambda=\infty$ above $\etaB$ the components of $G_t$ are
infinite, so a componentwise construction would not work in any case.

A \emph{translation-covariant factor} of a process $\mathcal X$ on which the translations
$(\vartheta_z)_{z\in\R^d}$ act is a point process of the form $\mathcal S=F(\mathcal X)$ with $F$
measurable and $F(\vartheta_z\mathcal X)=\vartheta_zF(\mathcal X)$ for every $z\in\R^d$. In the
next proposition $\mathcal X$ is the marked process $\widetilde{\mathcal P}$, and in \Cref{prop:soft} it
is $\mathcal P$ itself.

\begin{proposition}[Ergodicity and non-emptiness of covariant factors]
\label{prop:ergodic}
Attaching to each particle its recovery process as a mark makes
$\widetilde{\mathcal P}=\{(X_u,\mathcal R_u):u\in\mathcal P\}$ a Poisson process on
$C([0,\infty);\R^d)\times\mathbb M$, where $\mathbb M$ is the space of locally finite subsets of
$[0,\infty)$ and the intensity is $\nu_\eta\otimes\Pi_\mu$. Here $\Pi_\mu$ is the law of a
rate-$\mu$ Poisson process and $\Pi_\infty:=\delta_{[0,\infty)}$, since no marks are needed at
$\mu=\infty$. This process is ergodic under the spatial translations
$\vartheta_z:(w,\rho)\mapsto(w+z,\rho)$, $z\in\R^d$. Consequently, if $\mathcal S$ is a
translation-covariant factor of $\widetilde{\mathcal P}$ which is a point process of positive
intensity, then $\mathcal S\ne\emptyset$ almost surely.
\end{proposition}

\begin{proof}
The first assertion is the marking theorem. The recovery processes are independent of one another
and of $\mathcal P$, so marking the Poisson process $\{X_u:u\in\mathcal P\}$ of intensity
$\nu_\eta$ with independent marks of law $\Pi_\mu$ yields a Poisson process of intensity
$\nu_\eta\otimes\Pi_\mu$.

The intensity $\nu_\eta\otimes\Pi_\mu$ is invariant under $\vartheta_z$ because $\nu_\eta$ is, so
the translations act by measure-preserving transformations. The action is moreover mixing. Indeed,
the $\sigma$-field is generated by the counts of sets $\{w:w(0)\in R\}\cap C'$ with $R$ a bounded
rectangle. Any two such sets become disjoint once the translation is large enough, and counts of
disjoint sets under a Poisson process are independent. Hence the corresponding counts are
independent for all large $z$. A mixing measure-preserving action is ergodic, which gives the
second assertion.

For the third assertion, note that $\{\mathcal S=\emptyset\}$ is translation-invariant, since
$\mathcal S$ is a translation-covariant factor. Hence it has probability $0$ or $1$. Probability
$1$ would force the intensity of $\mathcal S$ to vanish, contrary to the hypothesis. Therefore
$\Prob(\mathcal S=\emptyset)=0$.
\end{proof}

This is the only ergodicity statement we use, and the pair-indexed transmission marks play no role
in it.

\label{sec:corners}
The pair $(\lambda,\mathfrak r)$ admits four combinations. Two of them are the object of this
paper, one is degenerate, and one is intermediate.

\paragraph{The finite-rate model $(\lambda<\infty,\unc)$.} This is the SIS contact process carried
by the mobile particle system. A susceptible particle is infected at rate $\lambda$ times the
number of infected particles in contact with it, and an infected particle recovers at rate $\mu$
irrespective of its surroundings. This is the model of \Cref{sec:smallmu}, and we write
$\Prob^{\lambda,\mu}_\eta$ for its law.

\paragraph{The instantaneous model $(\lambda=\infty,\iso)$.} Transmission occurs at every instant
of contact, and a particle recovers only while it is isolated. This is the model of
\Cref{sec:survupper,sec:extinction}, and we write $\Prob^{\infty,\mu}_\eta$ for its law.

\paragraph{The degenerate case $(\lambda=\infty,\unc)$.} If transmission were instantaneous and
recovery unconditional, an infected particle in contact with a second infected particle would be
reinfected immediately after every recovery, and no recovery would ever have any effect. The
infection would then be a deterministic function of $\mathcal P$ and would never die out. Making
recovery conditional on isolation is the standard way of keeping the model non-degenerate at
infinite infection rate, and it is the rule studied in \cite{BaldassoStauffer23}. For the same
reason we admit $\mu=\infty$ only under the isolation rule. Under $\unc$ it would make every
instant an effective recovery, and no infection could persist at all.

\paragraph{The intermediate case $(\lambda<\infty,\iso)$.} This is a possible model, and by
\Cref{prop:domination} below it lies between the two models we study. We do not treat it
separately. Every extinction statement we prove for the instantaneous model applies to it, and
every survival statement we prove for the finite-rate model applies to it.

The next proposition allows us to read a result proved at $\lambda=\infty$ as a statement about every
$\lambda$. In particular, an extinction result proved for the instantaneous model applies to the
finite-rate model at every infection rate. Note that the proposition changes both $\lambda<\infty$ to
$\lambda=\infty$ and $\unc$ to $\iso$ at once.

\begin{proposition}[Domination]
\label{prop:domination}
Let $\lambda\in(0,\infty)$ and $\mu\in(0,\infty)$, and realise both models on the same marked
configuration, with the same $\mathcal P$, the same recovery marks and the same initial infected
set. Then every open path of the finite-rate model is an open path of the instantaneous model. In
particular
\[
  I^{\lambda,\mu,\unc}_t\ \subseteq\ I^{\infty,\mu,\iso}_t\qquad\text{for every }t\ge0 ,
\]
almost surely, and the same inclusions hold with either intermediate case in between.
\end{proposition}

\begin{proof}
Let $(u_j)_{j\le n}$, $(t_j)_{j\le n+1}$ be an open path of the finite-rate model. A mark of
$\mathcal A_{u_{j-1}\to u_j}$ at $t_j$ is active only if $u_{j-1}$ and $u_j$ are in contact at
$t_j$, which is clause (i) of \Cref{def:openpath} at $\lambda=\infty$. For clause (ii), note that
a mark of $\mathcal R_{u_j}$ which is effective under the isolation rule is in particular a mark
of $\mathcal R_{u_j}$, hence effective under the unconditional rule. Therefore the set of
$\iso$-effective marks of a particle is contained in the set of its $\unc$-effective marks, and an
interval free of the latter is free of the former. Both clauses therefore hold in the
instantaneous model, and \eqref{eq:Idef} gives the inclusion. The two intermediate cases are
obtained by applying the same two observations one at a time.
\end{proof}

At $\lambda=\infty$ and $\mu=\infty$ no rates remain and the model is purely geometric. A particle
is infected at time $t$ if and only if it is joined to $I_0$ by a space-time path which uses
contacts as its spatial steps and, along each particle, only time intervals free of isolation. All
results below about the instantaneous model include this case.

\subsection{Survival and the critical density}
\label{sec:soft}

Unless stated otherwise, we infect a single particle at the origin at time $0$ and declare all
other particles susceptible. Formally, we work under the Palm measure of $\mathcal P$ at the
origin. We say that the infection \emph{dies out} if there is a finite time after which no
particle is infected, and that it \emph{survives} otherwise. Since the empty configuration is
absorbing, survival means that $I_t\ne\emptyset$ for every $t\ge0$. The main quantity of interest
in this paper is the \emph{critical density}
\begin{equation}
\label{eq:etac}
  \eta_c(\lambda,\mu)=\inf\big\{\eta>0:\ \Prob^{\lambda,\mu}_\eta(\text{survival})>0\big\} ,
\end{equation}
which is defined for every $\lambda\in(0,\infty]$ and $\mu\in(0,\infty]$. Here $\lambda=\infty$
is read as the instantaneous model, and $\mu=\infty$ is admitted only there. For the finite-rate
model at fixed $\eta$ and $\lambda$ we write in addition
\begin{equation}
\label{eq:muc}
  \mu_c(\lambda,\eta)=\sup\big\{\mu>0:\ \Prob^{\lambda,\mu}_\eta(\text{survival})>0\big\} ,
\end{equation}
with the convention $\sup\emptyset=0$. \Cref{prop:mono}(iii) shows that both quantities are well
defined, that is, that the sets in question are half-lines. The same proposition gives that
$\eta_c(\lambda,\mu)$ is non-decreasing in $\mu$ and non-increasing in $\lambda$ along the
finite-rate models. The comparison of a finite $\lambda$ with $\lambda=\infty$ is not of this
kind, since it changes the recovery rule as well. This comparison is \Cref{prop:domination}, and it
gives
\begin{equation}
\label{eq:etacinf}
  \eta_c(\infty,\mu)\ \le\ \eta_c(\lambda,\mu)
  \qquad\text{for every }\lambda,\mu\in(0,\infty) .
\end{equation}
We do not assert equality in the limit. The statement
$\eta_c(\infty,\mu)=\inf_{\lambda<\infty}\eta_c(\lambda,\mu)$ would be a continuity statement at
$\lambda=\infty$, and we do not prove it here. By \Cref{sec:open}(4) it would settle a question we
leave open, since it would transfer the finiteness of $\eta_c(\infty,\mu)$ to some finite
$\lambda$.
Finally, we write
\begin{equation}
\label{eq:rad}
  \rad(I_t)=\sup\{|X_u(t)|:\ u\in I_t\}\ \in[0,\infty] ,\qquad \sup\emptyset:=0 ,
\end{equation}
for the radius of the infected set. Positive speed will be stated in terms of this quantity.

The notion of survival above is the appropriate one when the initially infected set is finite.
When that set is infinite it is a weak notion, and we now quantify this.

\begin{proposition}[Soft survival]
\label{prop:soft}
Let $\lambda\in(0,\infty]$, $\mu\in(0,\infty]$ and $\mathfrak r\in\{\unc,\iso\}$, with
$\mu<\infty$ if $\mathfrak r=\unc$. Suppose the initially infected set $I_0$ is a
translation-covariant factor of $\mathcal P$ of positive intensity, and, if $\mu=\infty$, suppose
in addition that $I_0$ is \emph{increasing}: whenever $\mathcal P\subseteq\mathcal P'$ are locally
finite path configurations and $u\in\mathcal P$, membership $u\in I_0(\mathcal P)$ implies
$u\in I_0(\mathcal P')$. Then almost surely $I_t\ne\emptyset$ for every $t\ge0$. The extra
hypothesis at $\mu=\infty$ cannot be dropped; see \Cref{rem:softsharp}.
\end{proposition}

\begin{proof}
Write $\iota>0$ for the intensity of $I_0$. Since $I_0$ is itself a translation-covariant factor
of positive intensity, $I_0\ne\emptyset$ almost surely by \Cref{prop:ergodic}. This settles the
case $t=0$, so we fix $t>0$.

Assume first that $\mu<\infty$, and let $\mathcal S_t\subseteq I_0$ be the set of particles of
$I_0$ which carry no recovery mark at all in $[0,t]$. Take $n=0$, $t_0=0$ and $t_1=t$ in
\Cref{def:openpath}. Clause (i) is then empty of meaning, and an interval free of marks is free of
effective marks, so the length-zero path from $(u,0)$ to $(u,t)$ is open for every
$u\in\mathcal S_t$ under either recovery rule. This gives that $\mathcal S_t\subseteq I_t$. Since
$I_0$ is a factor of $\mathcal P$ alone and $(\mathcal R_u)_u$ is independent of $\mathcal P$, the
set $\mathcal S_t$ is an independent thinning of $I_0$ with retention probability $e^{-\mu t}$.
Hence it is a translation-covariant factor of $\widetilde{\mathcal P}$ of intensity
$\iota e^{-\mu t}>0$, and $\mathcal S_t\ne\emptyset$ almost surely by \Cref{prop:ergodic}.

Now let $\mu=\infty$, which forces $\mathfrak r=\iso$. In this case we replace $\mathcal S_t$ by
\[
  \mathcal S_t=\big\{u\in I_0:\ \exists\,v\in\mathcal P\setminus\{u\}\ \text{with}\
  |X_u(s)-X_v(s)|\le2r\ \text{for every }s\in[0,t]\big\} .
\]
Such a $u$ is never isolated on $[0,t]$, so it carries no effective recovery mark there, and the
same length-zero path is open. This gives that $\mathcal S_t\subseteq I_t$. The set $\mathcal S_t$
is a translation-covariant factor of $\mathcal P$, and we claim that its intensity is positive. By
the Mecke equation applied to $\mathcal P$, this intensity equals $|Q|^{-1}$ times
\[
  \int_{\{w(0)\in Q\}} \Prob\Big(w\in I_0(\mathcal P\cup\{w\})\ \text{ and }\
  \exists\,v\in\mathcal P:\ \sup_{s\le t}|w(s)-X_v(s)|\le2r\Big)\,\nu_\eta(\mathrm dw)
\]
for an arbitrary bounded $Q$. For fixed continuous $w$ both events are increasing in
$\mathcal P$, the first by hypothesis and the second because enlarging $\mathcal P$ enlarges the
set of candidates for $v$. Hence, by the Harris--FKG inequality for Poisson processes
\cite[Thm.~20.4]{LastPenrose18}, the integrand is at least $f(w)g(w)$, where
$f(w)=\Prob(w\in I_0(\mathcal P\cup\{w\}))$ and
\[
  g(w)=\Prob\big(\exists\,v\in\mathcal P:\ \sup_{s\le t}|w(s)-X_v(s)|\le2r\big)
  =1-\exp\big\{-\nu_\eta\big(\{v:\sup_{s\le t}|w-v|\le2r\}\big)\big\} .
\]
We now show that $g(w)>0$ for every continuous $w$. Bounding the $\nu_\eta$-measure from below by
the contribution of paths started in $B(w(0),r)$, we obtain
\[
  \nu_\eta\big(\{v:\sup_{s\le t}|w-v|\le2r\}\big)\ \ge\
  \eta\int_{B(w(0),r)}\Wien_x\Big(\sup_{s\le t}\big|W_s-x-(w(s)-w(0))\big|<2r-|x-w(0)|\Big)
  \,\mathrm dx\ >\ 0
\]
by the support theorem for Brownian motion, since the integrand is strictly positive for each
$x\in B(w(0),r)$. Finally, $\int_{\{w(0)\in Q\}}f\,\mathrm d\nu_\eta=\iota|Q|>0$ by the Mecke
equation again. Hence $f>0$ on a set of positive $\nu_\eta$-measure inside $\{w(0)\in Q\}$, and
on this set $fg>0$. Therefore the displayed intensity is positive, and $\mathcal S_t\ne\emptyset$
almost surely by \Cref{prop:ergodic}.

In both cases $I_t\ne\emptyset$ almost surely for each fixed $t>0$, hence almost surely
simultaneously for every rational $t>0$. Since $\{I_t\ne\emptyset\}$ decreases in $t$ by
\Cref{lem:truncate}, we obtain $I_t\ne\emptyset$ for every real $t>0$ as well.
\end{proof}

We now show that the extra hypothesis at $\mu=\infty$, that $I_0$ be increasing, cannot be
dropped.

\begin{remark}
\label{rem:softsharp}
Let $I_0=\{u\in\mathcal P:|X_u(0)-X_v(0)|>2r\ \text{for all }v\ne u\}$ be the set of particles
isolated at time $0$. This is a translation-covariant factor of intensity $\eta e^{-\eta\Vol}>0$
which is not increasing. Each of its particles is isolated throughout some interval $[0,\sigma)$,
hence carries an effective recovery mark in $(0,\varepsilon]$ for every $\varepsilon>0$. Then
\Cref{def:openpath}(ii) with $j=0$ forbids every open path issued from $(u,0)$, and
$I_t=\emptyset$ for every $t>0$.
\end{remark}

Note that the subcritical contact process on $\Z^d$ started from full occupancy also satisfies
the conclusion of \Cref{prop:soft}, since $\Prob(x\in I_t)\ge e^{-\mu t}$ for every site $x$. Hence
global survival in the sense of \eqref{eq:etac}, which is the sense used throughout this
literature, says nothing about whether the epidemic sustains itself once the initially infected
set is infinite. Accordingly, \Cref{prop:supercrit} below establishes above $\etaB$ that
the initially infected cluster is infinite, not that the epidemic sustains itself. Both
features of the instantaneous model are needed for this. At finite $\lambda$ the infinite cluster
is not infected at small times, and under unconditional recovery a pair of particles in permanent
contact is a two-site contact process, which dies out. The natural strengthenings are \emph{local
survival}, that is, that the set of times at which some particle within distance $1$ of the
origin is infected is unbounded, and \emph{non-degeneracy}, that is, that
$\liminf_{t\to\infty}\iota(t)>0$ where $\iota(t)$ is the intensity of $I_t$. We address neither
of these; see \Cref{sec:open}(3).

\subsection{Main results}
\label{sec:results}

We can now make the question of whether the motion helps precise. Below $\etaB$ the medium does
not percolate: at every fixed time the Gilbert graph has only finite components, so an infection
which survives there survives because the particles move. Above $\etaB$ the cluster of the seed
is infinite and, by \Cref{prop:soft}, survival in the sense of \eqref{eq:etac} is a property of
the initial condition rather than of the dynamics. A bound of the form
$\eta_c(\infty,\mu)\le\etaB$ is therefore not by itself evidence that anything happens. Taken
alone, it is consistent with the transition detected by \eqref{eq:etac} being only the
percolation transition of the medium, seen through a weak notion of survival.

The first theorem rules this out. It places the transition strictly below $\etaB$, in the range
where the soft mechanism is unavailable. Moreover, it gives the strongest form of survival the
method provides: from a single infected particle, the infection spreads to distance of order $t$
by time $t$.

\thmletter{A}
\begin{mainthm}[Survival below the static percolation threshold]
\label{thm:motion}
Let $d\ge2$ and $r>0$, and let $\eta<\etaB$, so that at every fixed time the Gilbert graph has
almost surely only finite components and the cluster of the seed is almost surely finite. For every
$\lambda\in(0,\infty)$ there is $\mu_*=\mu_*(\lambda;\eta,r,d)>0$ such that for every
$\mu\le\mu_*$ the infection started from a single particle at the origin survives with positive
probability and spreads at positive speed,
\[
  \Prob^{\lambda,\mu}_\eta\Big(\liminf_{t\to\infty}\frac{\rad(I_t)}{t}>0\Big)\ >\ 0 ,
\]
in the finite-rate model and, with $\Prob^{\infty,\mu}_\eta$ in place of
$\Prob^{\lambda,\mu}_\eta$, in the instantaneous model as well. Consequently there is
$\mu_\dagger=\mu_\dagger(r,d)\in(0,\infty]$ such that
\[
  \eta_c(\infty,\mu)\ <\ \etaB\quad\text{for every }\mu<\mu_\dagger ,
  \qquad
  \eta_c(\infty,\mu)\ =\ \etaB\quad\text{for every }\mu\in(\mu_\dagger,\infty) .
\]
\end{mainthm}

At these intensities the medium does not percolate at any fixed time, and yet the infection
spreads to infinity at linear speed. This form of survival differs from the one available above
$\etaB$. The initially infected set is a single particle, its cluster is almost surely finite,
\Cref{prop:soft} does not apply, and the conclusion is not a statement about the initial
condition.

\Cref{thm:motion} is proved in \Cref{sec:mainproof} by combining the next two theorems, which we
establish in the body of the paper and which are of independent interest. The first of them
locates the critical density \eqref{eq:etac}. It is positive for every recovery rate, with a bound
which does not depend on the infection rate. At infinite infection rate it is finite and, in
dimension two and above, at most the percolation threshold of the static medium.

\thmletter{B}
\begin{mainthm}[The critical density]
\label{thm:density}
Let $d\ge1$ and $r>0$.
\begin{enumerate}[(i)]
\item \emph{(Positivity, uniformly in $\lambda$.)} For every $\mu\in(0,\infty]$ there is
  $\bar\eta=\bar\eta(\mu,r,d)>0$ such that
  \[
    \eta_c(\lambda,\mu)\ \ge\ \bar\eta(\mu,r,d)\qquad\text{for every admissible }\lambda\in(0,\infty] ,
  \]
  that is, for every $\lambda\in(0,\infty]$ when $\mu<\infty$ and for $\lambda=\infty$ when
  $\mu=\infty$. Equivalently, the infection started from a single particle dies out almost surely
  at every intensity below $\bar\eta$, whatever the infection rate.
\item \emph{(Finiteness at $\lambda=\infty$.)} For every $\mu\in(0,\infty]$ and every $d\ge1$,
  $\eta_c(\infty,\mu)<\infty$.
\item \emph{(The Boolean bound, $d\ge2$.)} For $d\ge2$ and every $\mu\in(0,\infty]$,
  \[
    \eta_c(\infty,\mu)\ \le\ \eta_c(\infty,\infty)\ \le\ \etaB ,
  \]
  so the critical density at infinite infection rate is bounded above by the critical intensity of
  the static Boolean model, uniformly in the recovery rate.
\end{enumerate}
In particular $\eta_c(\infty,\mu)\in(0,\infty)$ for every $d\ge1$ and every $\mu\in(0,\infty]$.
That is, the instantaneous infection dies out almost surely for $\eta<\eta_c(\infty,\mu)$ and
survives with positive probability for $\eta>\eta_c(\infty,\mu)$.
\end{mainthm}

The remaining theorem is the complementary statement in the recovery rate. At every density and
every finite infection rate, however small, survival holds once recovery is slow enough. Read
through \eqref{eq:etacinf}, it says that the critical density of \Cref{thm:density} degenerates as
$\mu\downarrow0$ at every infection rate. This theorem gives the survival in \Cref{thm:motion},
and \Cref{thm:density}(iii) is used there only through the equality clause.

\thmletter{C}
\begin{mainthm}[Survival at small recovery rate and the degeneration of $\eta_c$]
\label{thm:smallmu}
Let $d\ge2$ and $\eta,r,\lambda>0$ with $\lambda<\infty$. There is
$\mu_*(\lambda)=\mu_*(\lambda;\eta,r,d)>0$ such that for every $\mu\le\mu_*(\lambda)$ the
finite-rate infection started from a single particle at the origin survives with positive
probability, and moreover
\[
  \Prob^{\lambda,\mu}_\eta\Big(\liminf_{t\to\infty}\frac{\rad(I_t)}{t}>0\Big)\ >\ 0 .
\]
In particular $\mu_c(\lambda,\eta)\ge\mu_*(\lambda)>0$ for every $\eta>0$ and every
$\lambda\in(0,\infty)$. Consequently
\[
  \lim_{\mu\downarrow0}\eta_c(\lambda,\mu)\ =\ 0\qquad\text{for every }\lambda\in(0,\infty] .
\]
\end{mainthm}

Together the two theorems say that $\eta_c(\lambda,\cdot)$ is a non-degenerate function of the
recovery rate. It is bounded away from $0$ at every fixed $\mu>0$ by a bound free of $\lambda$,
bounded above by $\etaB$ at $\lambda=\infty$, and tends to $0$ as $\mu\downarrow0$.
\Cref{thm:motion} applies at every intensity below $\etaB$, with a threshold $\mu_*$ depending on
the intensity, and by \Cref{thm:smallmu} the critical density $\eta_c(\infty,\mu)$ tends to $0$
as $\mu\downarrow0$. In other words, as recovery slows, the range of intensities at which the
motion sustains the infection extends to all of $(0,\etaB)$. The picture given by the three
theorems is drawn in \Cref{fig:phase} and described in \Cref{sec:phase} below.

\subsection{The phase diagram}
\label{sec:phase}

\Cref{fig:phase} collects the statements above into the $(\mu,\eta)$ plane at fixed $d\ge2$ and
$r>0$. Four features are marked.

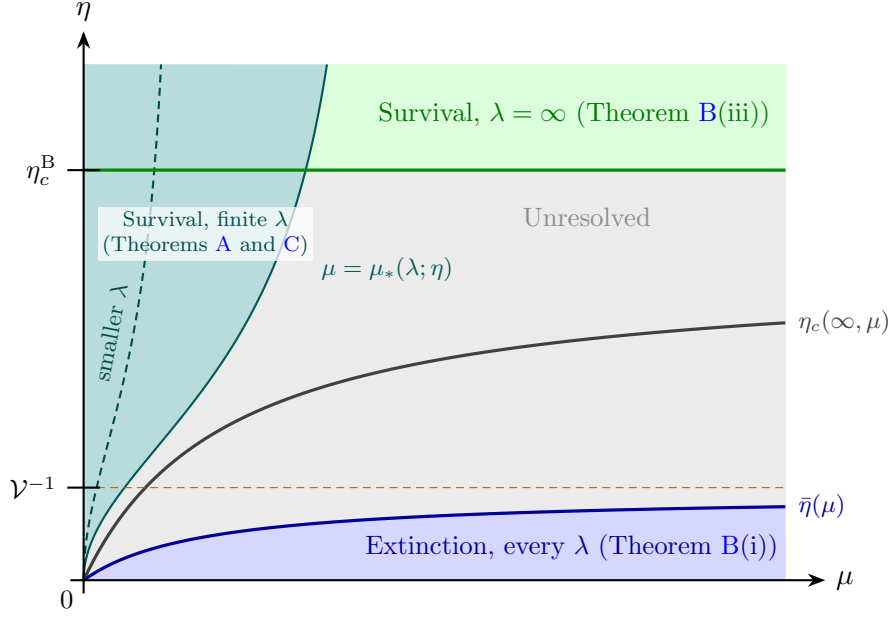
\begin{figure}[!h]
\centering
\begin{tikzpicture}[scale=1.75, >=Stealth]
  \fill[black!8] (0,0) rectangle (5.3, 3.9);
  \fill[green!14] (0, 3.1) rectangle (5.3, 3.9);
  \fill[blue!16]
    (0,0) -- plot[domain=0:5.3, samples=140, variable=\m] (\m, {0.65*\m/(\m+0.9)})
    -- (5.3, 0) -- cycle;
  \fill[teal!28]
    (0,0) -- plot[domain=0:3.9, samples=160, variable=\y]
      ({2.2*\y*\y/(\y*\y+3)}, \y) -- (0, 3.9) -- cycle;
  \draw[teal!70!black, thick]
    plot[domain=0:3.9, samples=160, variable=\y] ({2.2*\y*\y/(\y*\y+3)}, \y);
  \draw[teal!60!black, thick, densely dashed]
    plot[domain=0:3.9, samples=160, variable=\y] ({0.7*\y*\y/(\y*\y+3)}, \y);
  \draw[very thick, black!75]
    plot[domain=0:5.3, samples=140, variable=\m] (\m, {2.35*\m/(\m+1.1)});
  \draw[very thick, blue!60!black]
    plot[domain=0:5.3, samples=140, variable=\m] (\m, {0.65*\m/(\m+0.9)});
  \draw[very thick, green!55!black] (0, 3.1) -- (5.3, 3.1);
  \draw[densely dashed, orange!75!black] (0, 0.7) -- (5.3, 0.7);
  \draw[->, thick] (-0.1, 0) -- (5.6, 0) node[right] {$\mu$};
  \draw[->, thick] (0, -0.1) -- (0, 4.15) node[above] {$\eta$};
  \draw[thick] (3.5pt, 3.1) -- (-3.5pt, 3.1) node[left, font=\small] {$\etaB$};
  \draw[thick] (3.5pt, 0.7) -- (-3.5pt, 0.7) node[left, font=\small] {$\Vol^{-1}$};
  \node[below left, font=\small] at (0,0) {$0$};
  \node[green!45!black, font=\small, right] at (2.15, 3.52)
    {Survival, $\lambda=\infty$ (\Cref{thm:density}(iii))};
  \node[blue!60!black, font=\small, right] at (2.05, 0.25)
    {Extinction, every $\lambda$ (\Cref{thm:density}(i))};
  \node[blue!60!black, font=\footnotesize, right] at (5.32, 0.56) {$\bar\eta(\mu)$};
  \node[black!75, font=\footnotesize, right] at (5.32, 1.95) {$\eta_c(\infty,\mu)$};
  \node[teal!60!black, font=\scriptsize, align=center, inner sep=1.5pt,
        fill=white, fill opacity=0.72, text opacity=1] at (0.92, 2.62)
    {Survival, finite $\lambda$\\ (\Cref{thm:motion,thm:smallmu})};
  \node[teal!70!black, font=\footnotesize, right] at (1.72, 2.35)
    {$\mu=\mu_*(\lambda;\eta)$};
  \node[teal!60!black, font=\footnotesize, above, rotate=79] at (0.34, 1.85)
    {smaller $\lambda$};
  \node[black!45, font=\small] at (3.8, 2.75) {Unresolved};
\end{tikzpicture}
\caption{\emph{Teal}: survival with positive speed for the finite-rate
model to the left of $\mu=\mu_*(\lambda;\eta)$ (\Cref{thm:smallmu}), drawn for one value of
$\lambda$, the dashed curve being the same boundary at a smaller $\lambda$ (\Cref{rem:muStar}). \emph{Blue}: almost sure extinction
below $\bar\eta(\mu)$, simultaneously for every $\lambda\in(0,\infty]$ (\Cref{thm:density}(i)).
\emph{Light green}: survival at $\lambda=\infty$ above $\etaB$ (\Cref{thm:density}(iii)); by
\Cref{sec:soft} the content there is only that the cluster of the seed is infinite, so where the
teal region overlaps it the teal region says more.
\emph{Black}: the critical curve $\eta_c(\infty,\mu)$, non-decreasing and lying between the two
proved regions. \emph{Amber}: the mean degree one
line $\eta=\Vol^{-1}$. \emph{Grey}: neither proved.}
\label{fig:phase}
\end{figure}

\paragraph{Survival at small recovery rate and the strict inequality.} By \Cref{thm:smallmu} the
region $\mu\le\mu_*(\lambda;\eta)$ is a region of survival with positive speed for the finite-rate
model. It lies close to the $\eta$-axis and, by \Cref{rem:muStar}, shrinks towards it as
$\lambda\downarrow0$. The same remark shows that within the binding constraint $\lambda$ and
$\eta$ enter only through their product, which is why the region is drawn opening upwards. Its
part below the line $\eta=\etaB$ is \Cref{thm:motion}. There the medium has only finite
components at every fixed time, so the survival drawn is produced by the motion alone.

\paragraph{Extinction at low density for every $\lambda$.} By \Cref{thm:density}(i) the region
$\eta<\bar\eta(\mu)$ is a region of almost sure extinction, simultaneously for every
$\lambda\in(0,\infty]$. That is, the same curve bounds the extinction region of the instantaneous
model and of the finite-rate model at every infection rate. The function $\bar\eta$ may be taken
non-decreasing, since replacing it by $\mu\mapsto\sup_{\mu'\le\mu}\bar\eta(\mu')$ preserves the
conclusion, as $\eta_c(\lambda,\cdot)$ is non-decreasing by \Cref{prop:mono}(iii). It tends to $0$
as $\mu\downarrow0$. For $d\ge2$ this follows from \Cref{thm:smallmu}, and for every $d\ge1$ it
follows from the construction of \Cref{sec:extinction}. There $\bar\eta$ is a negative power of
the base scale of the renormalisation, the scale at which \Cref{lem:trigger} makes the starting
estimate of the recursion small.
As $\mu\downarrow0$ the local extinction probability from \Cref{lem:localext} tends to $0$, which
forces this base scale to grow. We claim no quantitative form of the degeneration; see
\Cref{sec:open}(2).

\paragraph{Survival above the Boolean threshold.} By \Cref{thm:density}(iii) the half-plane
$\eta>\etaB$ is a region of survival for the instantaneous model, uniformly in $\mu$. Its boundary
is a horizontal line and not a curve, since the argument uses no rates. By \Cref{sec:soft}, what
it asserts is that the cluster of the seed is infinite. Where the teal region overlaps it, the
teal region says more.

\paragraph{The critical curve at $\lambda=\infty$.} Between the two regions lies the curve
$\eta_c(\infty,\mu)$. It is non-decreasing in $\mu$ by \Cref{prop:mono}(iii), takes values in
$(0,\infty)$ by \Cref{thm:density}, and is bounded above by $\etaB$ even at $\mu=\infty$. By
\Cref{thm:motion} it is strictly below $\etaB$ for all small $\mu$, and indeed it tends to $0$
there. How far up in $\mu$ the strict inequality persists, that is, whether the threshold
$\mu_\dagger$ of that theorem is infinite and whether the strict inequality holds at $\mu=\infty$,
is \Cref{sec:open}(1). The curve is drawn below $\etaB$ throughout, which is just a conjecture at this point. For finite $\lambda$ the curve $\eta_c(\lambda,\mu)$ lies above it by
\eqref{eq:etacinf}, but we prove no upper bound on it at fixed $\mu$; see \Cref{sec:open}(4).

The marker on the $\eta$-axis is $\eta=\Vol^{-1}$, the point at which the mean degree of the
medium equals one. As explained in \Cref{sec:noR0}, it is the ceiling of every reproduction
criterion, and in
$d=2$ it lies below $\etaB$ by a factor of about $4.5$. The interval in which
$\eta_c(\infty,\mu)$ is known to lie therefore extends far beyond the range a first-moment
argument could cover.

\subsection{Methods and outline}
\label{sec:method}

We prove \Cref{thm:density}(ii)--(iii) in \Cref{sec:survupper}. \Cref{thm:smallmu}, and with it
\Cref{thm:motion}, is proved by a Lipschitz-surface argument in \Cref{sec:smallmu}.
\Cref{thm:density}(i) follows from \Cref{thm:extinction} together with \Cref{prop:domination}, and
is proved in \Cref{sec:extinction}.

\emph{A Lipschitz surface} for \Cref{thm:motion,thm:smallmu} and for the upper bounds on
$\eta_c$ in \Cref{thm:density}. We use the multi-scale percolation of Gracar and Stauffer
\cite{GracarStauffer19b}, in the form in which it was applied to infections with recovery in
\cite{GracarStauffer19a,BaldassoStauffer23}. It is collected in \Cref{sec:lipschitz} and used
twice, with a crowding event in \Cref{prop:crowding} and with the event of \Cref{sec:smallmu} for
\Cref{thm:smallmu}. Two points require additional work in the continuum. First, the cell event
must be an increasing function of the configuration, whereas the particle carrying the infection
into a cell is determined by the previous cell and, at finite $\lambda$, cannot simply be replaced
by a neighbour as it can on a lattice. Following \cite{BaldassoStauffer23}, we attach an
independent auxiliary path to each cell and let a designated infected particle follow it. This
lets the event quantify over points of space rather than over particles (\Cref{sec:auxpath}).
Second, the percolation is applied to particles conditioned to keep their displacement in a cube,
which is not a high-probability event. \Cref{lem:condkernel} shows that the kernel of a Brownian
particle conditioned in this way obeys the same estimates as the free one, and
\Cref{cor:mixingcond} transfers the local mixing estimates.

\emph{A multi-scale renormalisation} for \Cref{thm:density}(i), which we carry out in
\Cref{sec:extinction}. The reason for this approach is that no reproduction number is available
at $\lambda=\infty$. The lifetime of an infected particle is governed by the geometry rather than
by a rate, and it is positively correlated with the number of its offspring. Moreover, no
first-moment criterion is valid at the intensities in question in any case. Such criteria are
confined to mean degree below one, whereas $\etaB\Vol\approx4.5$ in $d=2$ and exceeds $1$ in
every finite dimension. We explain this in \Cref{sec:noR0}.

The structure of the remainder of this paper is as follows. In \Cref{sec:mono} we collect the
monotonicity statements. Every comparison we need is a pathwise statement about open paths. In
\Cref{sec:mixing} we collect the heat-kernel estimates and establish the two forms of the local
mixing estimate, together with their transport to particles conditioned to keep their
displacement in a cube. In \Cref{sec:lipschitz} we set up the space-time tessellation, record the
multi-scale Lipschitz percolation of Gracar and Stauffer \cite{GracarStauffer19b}, namely the
existence of the surface and the surrounding of the origin, and derive from it the propagation
and seeding lemmas which turn a surface into survival. These are used in \Cref{sec:survupper} and
again in \Cref{sec:smallmu}.

In \Cref{sec:survupper} we prove the upper bounds \Cref{thm:density}(ii)--(iii). For $d\ge2$ and
$\eta>\etaB$ we show that the infinite Gilbert cluster is infected at arbitrarily small times
(\Cref{prop:supercrit}). Finiteness in every dimension, which is needed for $d=1$ where
$\etaB=\infty$, follows from a crowding estimate fed into \Cref{thm:GS} (\Cref{prop:crowding}).
In \Cref{sec:smallmu} we prove \Cref{thm:smallmu} by defining a local increasing event which
carries the infection across a cell and has high probability once the scale is large and the
recovery rate small. In \Cref{sec:extinction} we prove \Cref{thm:extinction} by the
renormalisation mentioned above. Survival forces a half-crossing of a space-time box
(\Cref{lem:surv2H}), half-crossings cascade (\Cref{lem:cascade}), well-separated half-crossings
are nearly independent (\Cref{lem:horiz,prop:vertical}), and the resulting recursion contracts
(\Cref{prop:recursion,lem:contraction}) once the base-scale probability is small, which is the
triggering estimate of \Cref{sec:trigger}.

In \Cref{sec:mainproof} we assemble the three theorems, and in \Cref{sec:remarks} we explain why
no reproduction number is available, discuss the reduction to $\mu=\infty$, and list the open
problems.

\section{Monotonicity}
\label{sec:mono}

Throughout this section $\lambda\in(0,\infty]$, $\mu\in(0,\infty]$ and
$\mathfrak r\in\{\unc,\iso\}$ are arbitrary, subject only to the restriction of
\Cref{sec:corners} that $\lambda=\infty$ and $\mu=\infty$ are admitted only with
$\mathfrak r=\iso$. We collect in the next proposition the monotonicity properties of the process that
we will use.

\begin{proposition}
\label{prop:mono}
\begin{enumerate}[(i)]
\item \emph{(Attractiveness.)} On a fixed marked configuration, if $A\subseteq B$ are initial
  infected sets then $A_t\subseteq B_t$ for every $t\ge0$.
\item \emph{(Monotonicity in the configuration.)} Let $\omega\subseteq\omega'$ be marked
  configurations carrying the same trajectories, the same recovery marks and the same transmission
  marks on the particles, respectively the ordered pairs, that they have in common. Then every
  open path of $\omega$ is an open path of $\omega'$. Consequently $I_t(\omega)\subseteq
  I_t(\omega')$ for every $t$, and every event which is the existence of an open path with
  prescribed properties is increasing.
\item \emph{(Monotonicity of the survival probability.)} The map
  $\eta\mapsto\Prob^{\lambda,\mu}_\eta(\text{survival})$ is non-decreasing,
  $\lambda\mapsto\Prob^{\lambda,\mu}_\eta(\text{survival})$ is non-decreasing, and
  $\mu\mapsto\Prob^{\lambda,\mu}_\eta(\text{survival})$ is non-increasing. The three statements are
  at a fixed recovery rule $\mathfrak r$. Hence $\eta_c$ is well
  defined by \eqref{eq:etac} and is non-decreasing in $\mu$, and $\mu_c$ is well defined by
  \eqref{eq:muc} and is non-decreasing in $\eta$ and in $\lambda$.
\end{enumerate}
\end{proposition}

\begin{proof}
We start with (i). By \eqref{eq:Idef} the set $I_t$ is the set of endpoints of open paths issued
from $I_0$, and by \Cref{def:openpath} whether a given space-time path is open does not refer to
$I_0$. Enlarging $I_0$ therefore enlarges the family of admissible starting points and nothing
else.

For (ii) we check the two clauses of \Cref{def:openpath} separately. Clause (i) is a statement
about a transmission mark of a pair present in both configurations and about the contact of the
two particles concerned, that is, about their trajectories. Both are unchanged in passing from
$\omega$ to $\omega'$. For clause (ii) under $\mathfrak r=\unc$ there is nothing to prove, since
the effective marks of a particle are all of its marks. Under $\mathfrak r=\iso$, a recovery mark
of a particle $u$ at time $t$ is effective in $\omega'$ only if $u$ is isolated in $\omega'$ at
$t$. Since $\omega\subseteq\omega'$, such a $u$ is also isolated in $\omega$ at $t$, so the
effective marks of $u$ in $\omega'$ form a subset of those in $\omega$. Hence an interval free of
effective marks in $\omega$ is also free of effective marks in $\omega'$, and an open path of
$\omega$ remains open in $\omega'$.

We now prove (iii). Let $\eta\le\eta'$. By the superposition property applied to
\eqref{eq:pathPPP} we may realise $\mathcal P^{\eta}$ and $\mathcal P^{\eta'}$ on one probability
space with $\mathcal P^{\eta}\subseteq\mathcal P^{\eta'}$ and matched trajectories. We attach to
each particle of $\mathcal P^{\eta'}$ its own recovery marks and to each ordered pair of
$\mathcal P^{\eta'}$ its own transmission marks, and take the particle at the origin common to
both. Survival is the event that for every $t$ there is an open path from the origin particle
which extends to time $t$, and this event is increasing by (ii). Monotonicity in $\mu$ for
$\mu\le\mu'<\infty$ follows by thinning the recovery processes. A rate-$\mu'$ family contains a
rate-$\mu$ family, and removing marks can only remove effective ones, so it can only create open
paths. The case $\mu'=\infty$ is not a thinning and we treat it directly. Under the isolation
rule at $\mu'=\infty$ an interval is free of effective recovery marks if and only if it contains
no instant of isolation. At $\mu<\infty$ it suffices that it contain no mark at an instant of
isolation, which is a weaker requirement. Hence every open path at $\mu'=\infty$ is an open path
at every $\mu<\infty$, and
$\Prob^{\infty,\infty}_\eta(\text{survival})\le\Prob^{\infty,\mu}_\eta(\text{survival})$.
Monotonicity in $\lambda$ follows in the same way from the transmission processes. A
rate-$\lambda'$ family with $\lambda\le\lambda'$ contains a rate-$\lambda$ family, and adding
transmission marks can only create open paths. The case $\lambda'=\infty$ follows from the
observation, already made in \Cref{prop:domination}, that a mark is active only at an instant of
contact.

Note that the three statements hold at a fixed $\mathfrak r$. The comparison of the finite-rate
model with the instantaneous one changes the recovery rule as well, and is the content of
\Cref{prop:domination}. At $\mathfrak r=\unc$ the value $\lambda'=\infty$ is inadmissible, since it
is the degenerate combination of \Cref{sec:corners}, so the last clause is a statement about
$\mathfrak r=\iso$.
\end{proof}
\section{Heat kernels and local mixing}
\label{sec:mixing}

In this section we collect the estimates on the Brownian transition kernel that we use throughout
the paper, and we prove from them two local mixing statements, one for free and one for
conditioned particles. Local mixing is the property of the medium that the multi-scale
percolation of \Cref{sec:lipschitz} requires, and the vertical decoupling of \Cref{sec:vertical}
is based on it. We follow the arrangement of \cite{GracarStauffer19a}. We first state the kernel
estimates, then prove mixing for free particles, then study the kernel of a conditioned particle,
and finally adapt the free statement to conditioned particles.

\subsection{Kernel estimates}
\label{sec:HK}

Throughout, $B$ is a standard Brownian motion in $\R^d$, of variance $t$ per coordinate, with
transition density
\begin{equation}
\label{eq:heatkernel}
  q_\Delta(x,y)=(2\pi\Delta)^{-d/2}\exp\Big\{-\frac{|x-y|^2}{2\Delta}\Big\} .
\end{equation}
The \emph{relative motion} $X_u-X_v$ of two particles has variance $2t$ per coordinate, so its
transition density at time $t$ is $q_{2t}$. Every estimate below applies to it after replacing
$\Delta$ by $2\Delta$, which changes only the constants. We will need the following standard
estimates.

\begin{lemma}[Kernel estimates]
\label{lem:kernel}
There are constants $c,C$ depending only on $d$ such that, for all $\Delta>0$ and all
$x,x_1,x_2,y\in\R^d$,
\begin{enumerate}[(H1)]
\item\label{H1} $q_\Delta(x,y)\le(2\pi\Delta)^{-d/2}$;
\item\label{H2} $|q_\Delta(x_1,y)-q_\Delta(x_2,y)|\le c\,\Delta^{-(d+\kappa)/2}|x_1-x_2|^{\kappa}$
  with $\kappa=1$;
\item\label{H3} $\Prob_x\big[\sup_{s\le t}|B_s-x|\ge R\big]\le4d\,e^{-R^2/(2dt)}$ for $R,t>0$;
\item\label{H4} $\displaystyle(2\pi\Delta)^{-d/2}\int_{|z-y|\ge R}
  \exp\Big\{-\frac{(|z-y|-\varsigma)^2}{2\Delta}\Big\}\,\mathrm dz\ \le\ C\,e^{-R^2/(16\Delta)}$
  whenever $0\le2\varsigma\le R$.
\end{enumerate}
\end{lemma}

\begin{proof}
(H\ref{H1}) is the maximum of \eqref{eq:heatkernel}, and (H\ref{H3}) is the reflection bound applied
coordinatewise. For (H\ref{H2}), we have
$\nabla_xq_\Delta(x,y)=-q_\Delta(x,y)(x-y)/\Delta$, so
\[
  |\nabla_xq_\Delta(x,y)|\ \le\ (2\pi\Delta)^{-d/2}\Delta^{-1/2}\sup_{u\ge0}ue^{-u^2/2}
  \ \le\ c\Delta^{-(d+1)/2} ,
\]
and the claim follows from the mean value theorem. For (H\ref{H4}), we pass
to polar coordinates and use $|z-y|-\varsigma\ge|z-y|/2$ on the domain of integration, which
holds since $|z-y|\ge R\ge2\varsigma$. This gives the bound
$C_d\Delta^{-d/2}\int_R^\infty r^{d-1}e^{-r^2/(8\Delta)}\mathrm dr$. Splitting
$e^{-r^2/(8\Delta)}\le e^{-R^2/(16\Delta)}e^{-r^2/(16\Delta)}$ and using
$\Delta^{-d/2}\int_0^\infty r^{d-1}e^{-r^2/(16\Delta)}\mathrm dr=C_d'$ finishes the proof.
\end{proof}

\subsection{Soft local times}
\label{sec:softlocal}

We will couple families of independent random points with Poisson point processes by means of
the soft local time method of Popov and Teixeira. The statement we need is
the following.

\begin{proposition}[Soft local times]
\label{prop:softlocal}
Let $J$ be an at most countable index set and let $(Z_j)_{j\in J}$ be independent random points
of $\R^d$ with densities $g_j$. Let $(\xi_j)_{j\in J}$ be i.i.d.\ exponential random variables
of mean $1$ and define the soft local time
\[
  H_J(y)=\sum_{j\in J}\xi_j\,g_j(y),\qquad y\in\R^d .
\]
Let $\rho\colon\R^d\to[0,\infty)$ be measurable and deterministic. Then there is a coupling
$\mathbb Q$ of $(Z_j)_{j\in J}$ with a Poisson point process $\psi$ of intensity $\rho$
such that for every measurable $A\subseteq\R^d$
\begin{enumerate}[(i)]
\item\label{sl:up} $\mathbb Q\big(\{Z_j\}_{j\in J}\cap A\subseteq\psi\big)\ \ge\
  \mathbb Q\big(H_J(y)\le\rho(y)\ \text{ for all }y\in A\big)$;
\item\label{sl:low} $\mathbb Q\big(\psi\cap A\subseteq\{Z_j\}_{j\in J}\big)\ \ge\
  \mathbb Q\big(H_J(y)\ge\rho(y)\ \text{ for all }y\in A\big)$.
\end{enumerate}
\end{proposition}

\begin{proof}
The construction of \cite[Section 4]{PopovTeixeira15} produces a Poisson point process $\Sigma$ on
$\R^d\times[0,\infty)$ with intensity the product of Lebesgue measures such that
\[
  \big\{(Z_j,\zeta_j)\big\}_{j\in J}=\Sigma\cap\big\{(z,u):u\le H_J(z)\big\}
\]
for suitable heights $\zeta_j\le H_J(Z_j)$. The construction is introduced in
\cite[Section 4]{PopovTeixeira15} and proved in \cite[Cor.~4.4]{PopovTeixeira15}. A reformulation
for particles on a general space is given in \cite[Appendix~A]{Hilario2015}, and the statement we
use corresponds to \cite[Cor.~A.3]{Hilario2015}. Since $\rho$ is deterministic, the projection
$\psi$ of $\Sigma\cap\{(z,u):u\le\rho(z)\}$ is a Poisson point process of intensity $\rho$. It is
a function of $\Sigma$ alone, hence independent of the densities $g_j$. For (\ref{sl:up}), if
$H_J\le\rho$ on $A$ and $Z_j\in A$, then $\zeta_j\le H_J(Z_j)\le\rho(Z_j)$, so $Z_j\in\psi$. For
(\ref{sl:low}), if $H_J\ge\rho$ on $A$ and $z\in\psi\cap A$, then $z$ is the projection of a point
$(z,u)\in\Sigma$ with $u\le\rho(z)\le H_J(z)$. That point therefore lies in
$\Sigma\cap\{(w,v):v\le H_J(w)\}$, so it is one of the $(Z_j,\zeta_j)$ and $z\in\{Z_j\}_{j\in J}$.
Note that the two conclusions hold under the one coupling $\mathbb Q$, since both $\psi$ and
$(Z_j)_{j\in J}$ are read off the same $\Sigma$.
\end{proof}

\subsection{Local mixing}
\label{sec:vertmix}

We now state the two local mixing estimates, an upper and a lower form. Both forms carry the same
prefactor,
\begin{equation}
\label{eq:Gamma}
  \Gamma\ =\ \Gamma(K,\Delta,\epsilon,d)\ :=\ C\,\epsilon^{-d}
  \Big(\frac K{\Delta^{1/2}}\Big)^{d(d+1)}\ \ge\ 1 ,
\end{equation}
which is polynomial in the scale parameters. It comes from passing from a bound on the soft
local time $H_J(y)$ at a fixed $y$ to a bound holding simultaneously at every $y\in Q_{K'}$, which
\Cref{prop:softlocal} requires in either direction. This reduction is the first step of the proof
below. Only the exponential matters for what follows, since every factor in the recursion of
\Cref{sec:recursion} is polynomial in the scale parameters in any case. The lower form is
stated in \cite{GJLV26} with the sharper prefactor $|Q_{K'}|$. The proof below gives the
prefactor $\Gamma$, which suffices for our purposes. The setting is drawn in \Cref{fig:mixing}.

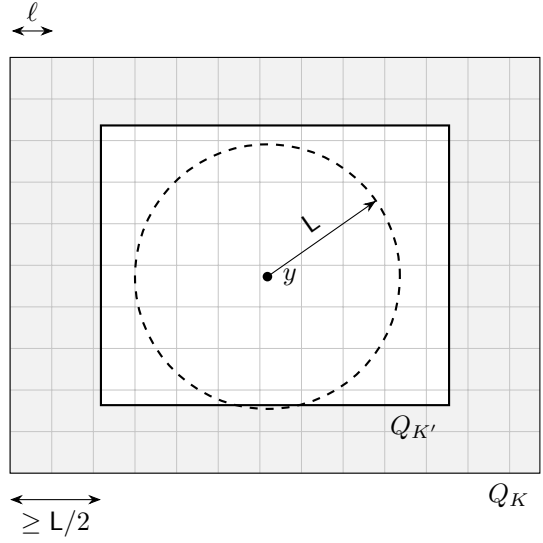
\begin{figure}[!h]
\centering
\begin{tikzpicture}[x=1cm,y=1cm,>=Stealth]
  \fill[black!5] (0,0) rectangle (7,5.5);
  \fill[white] (1.2,0.9) rectangle (5.8,4.6);
  \draw[step=0.55,gray!45,very thin] (0,0) grid (7,5.5);
  \draw (0,0) rectangle (7,5.5); \node[below left] at (7,0) {\small$Q_K$};
  \draw[thick] (1.2,0.9) rectangle (5.8,4.6); \node[below left] at (5.8,0.9) {\small$Q_{K'}$};
  \fill (3.4,2.6) circle (1.8pt) node[right=2pt]{\small$y$};
  \draw[thick,dashed] (3.4,2.6) circle (1.75);
  \draw[->] (3.4,2.6) -- ++(35:1.75) node[midway,above,sloped]{\small$\mathsf L$};
  \draw[<->] (0,-0.35) -- (1.2,-0.35) node[midway,below]{\small$\ge \mathsf L/2$};
  \draw[<->] (0.02,5.85) -- (0.57,5.85) node[midway,above]{\small$\ell$};
\end{tikzpicture}
\caption{The setting of the local mixing estimates. The box $Q_K$ is tessellated into subcubes
of side $\ell$, each carrying at most $\varrho\ell^d$ particles, and the conclusion is drawn
inside $Q_{K'}$. In the proof the sum $\sum_jq_\Delta(x_j,y)$ is split at the radius
$\mathsf L=c_1\Delta^{1/2}\epsilon^{-1/d}$: inside the dashed ball each particle is compared with the
minimiser of $q_\Delta(\cdot,y)$ over its subcube, and the shaded far field is estimated by a
Gaussian tail.}
\label{fig:mixing}
\end{figure}

\begin{theorem}[Local mixing, upper form]
\label{thm:mixingupper}
There are $c_0,c_1,C>0$ depending only on $d$ such that the following holds. Fix $K>\ell>0$
large enough and $\epsilon\in(0,1)$, and tessellate $Q_K$ into subcubes of side $\ell$. Let
$(x_j)_{j\in J}\subset Q_K$ be a configuration with at most $\varrho\ell^d$ points in each
subcube, where $\lfloor\varrho\ell^d\rfloor>0$. Let $\Delta\ge c_0\ell^2\epsilon^{-4/\kappa}$,
let $Y_j$ be the position at time $\Delta$ of a Brownian motion started at $x_j$, these being
independent, and let $K-K'\ge c_1\Delta^{1/2}\epsilon^{-1/d}$. Then $(Y_j)_{j\in J}$ can be
coupled with a Poisson process $\psi$ of intensity $\varrho(1+\epsilon)$ on $\R^d$, independent
of the initial configuration, so that
\begin{equation}
\label{eq:Mplus}
  \{Y_j\}_{j\in J}\cap Q_{K'}\subseteq\psi\quad\text{with probability at least}\quad
  1-\Gamma\exp\{-C\varrho\epsilon^2\Delta^{d/2}\} ,
\end{equation}
with $\Gamma$ as in \eqref{eq:Gamma}.
\end{theorem}

\begin{theorem}[Local mixing, lower form]
\label{thm:mixinglower}
There are $c_0,c_1,C>0$ depending only on $d$ such that the following holds. Fix $K>\ell>0$
large enough and $\epsilon\in(0,1)$, and tessellate $Q_K$ into subcubes of side $\ell$. Let
$(x_j)_{j\in J}\subset Q_K$ be a configuration with at least $\varrho\ell^d$ points in each
subcube, where $\lfloor\varrho\ell^d\rfloor>0$. Let $\Delta\ge c_0\ell^2\epsilon^{-4/\kappa}$,
let $Y_j$ be the position at time $\Delta$ of a Brownian motion started at $x_j$, these being
independent, and let $K-K'\ge c_1\Delta^{1/2}\epsilon^{-1/d}$. Then $(Y_j)_{j\in J}$ can be
coupled with a Poisson process $\psi$ of intensity $\varrho(1-\epsilon)$ on $\R^d$, independent
of the initial configuration, so that
\[
  \psi\cap Q_{K'}\subseteq\{Y_j\}_{j\in J}\quad\text{with probability at least}\quad
  1-\Gamma\exp\{-C\varrho\epsilon^2\Delta^{d/2}\} ,
\]
with $\Gamma$ as in \eqref{eq:Gamma}.
\end{theorem}

We prove the upper form in full below. The lower form is the local mixing statement of
\cite{GJLV26} at stability index $\alpha=2$, and its lattice counterpart is
\cite[Theorem~5]{GracarStauffer19a}. It is proved by the same argument with the inequalities
reversed. After the proof of the upper form we set out the preliminary thinning that the lower
form requires and the three points at which the two directions differ, and we refer to those
papers for the details.

\begin{proof}[Proof of \Cref{thm:mixingupper}]
Write $\mathsf L=c_1\Delta^{1/2}\epsilon^{-1/d}$, so that $K-K'\ge \mathsf L$ by hypothesis, and let $J$ be the
index set of the particles in $Q_K$. Then for $j\in J$, $Y_j$ has density $q_\Delta(x_j,\cdot)$ and these
are independent. We apply \Cref{prop:softlocal}(\ref{sl:up}) with $g_j=q_\Delta(x_j,\cdot)$,
$\rho\equiv\varrho(1+\epsilon)$ and $A=Q_{K'}$. This gives a coupling $\mathbb Q$ of
$(Y_j)_{j\in J}$ with a Poisson point process $\psi$ of intensity $\varrho(1+\epsilon)$,
independent of the initial configuration, under which $\{Y_j\}_{j\in J}\cap Q_{K'}\subseteq\psi$
holds with probability at least
\[
  \mathbb Q\big(H_J(y)\le\varrho(1+\epsilon)\ \text{ for all }y\in Q_{K'}\big),
  \qquad H_J(y)=\sum_{j\in J}\xi_j\,q_\Delta(x_j,y).
\]
Throughout the proof we fix
\begin{equation}
\label{eq:kappachoice}
  \varkappa:=\tilde C\epsilon\Delta^{d/2},\qquad \tilde C:=\tfrac14(2\pi)^{d/2},\qquad
  M:=\tfrac18\varkappa\varrho\epsilon=\tfrac{\tilde C}8\,\varrho\epsilon^2\Delta^{d/2},
\end{equation}
and note that by (H\ref{H1})
\begin{equation}
\label{eq:supkappa}
  \sup_{x,y\in\R^d}\varkappa\,q_\Delta(x,y)\ \le\ \varkappa(2\pi\Delta)^{-d/2}
  \ =\ \frac\epsilon4\ \le\ \frac14 .
\end{equation}
Note also that $\Delta^{1/2}\le K$, since $K\ge K'+c_1\Delta^{1/2}\epsilon^{-1/d}
>c_1\Delta^{1/2}$ and $c_1\ge1$. This is the only use made below of the hypothesis
$K-K'\ge c_1\Delta^{1/2}\epsilon^{-1/d}$.

We first explain how to move from a pointwise to a uniform bound.
For \Cref{prop:softlocal} we need to control
$\sup_{Q_{K'}}H_J$, and a pointwise bound is not sufficient for this. For an uncountable family
of events the quantity $\int_{Q_{K'}}\mathbb Q(H_J(y)>\varrho(1+\epsilon))\,\mathrm dy$ is the
expected Lebesgue measure of the exceedance set, and it bounds neither the probability that this
set is non-empty nor anything else that we could use. We therefore reduce the supremum to a
maximum over a finite net. This reduction is the source of the factor $\Gamma$.

Write $n=|J|$ and $\Xi=\sum_{j\in J}\xi_j$. The kernel $q_\Delta$ is symmetric, so (H\ref{H2}) applies
in its second variable and
\begin{equation}
\label{eq:HJlip}
  |H_J(y)-H_J(y')|\ \le\ \sum_{j\in J}\xi_j\big|q_\Delta(x_j,y)-q_\Delta(x_j,y')\big|
  \ \le\ c\,\Delta^{-(d+1)/2}\,\Xi\,|y-y'| ,\qquad y,y'\in\R^d .
\end{equation}
There are $(K/\ell)^d$ subcubes, each carrying at most $\varrho\ell^d$ points, so
$n\le\varrho K^d$. Moreover $M\le\tfrac{\tilde C}8\varrho K^d$, because $\epsilon\le1$ and
$\Delta^{1/2}\le K$. Since $\E e^{\Xi/2}=2^n$, Markov's inequality gives
$\mathbb Q(\Xi>4n+2M)\le2^ne^{-2n-M}\le e^{-M}$, and therefore
\begin{equation}
\label{eq:Xibound}
  \mathbb Q\big(\Xi>\Xi_\ast\big)\ \le\ e^{-M},\qquad
  \Xi_\ast:=4\varrho K^d+2M\ \le\ C_d\,\varrho K^d .
\end{equation}
We now tessellate $Q_{K'}$ into cubes of side $\varpi$ and let $\mathcal Y$ consist of one point
of each, where
\[
  \varpi:=\frac{\epsilon\,\Delta^{(d+1)/2}}{C_2\,K^d},\qquad
  C_2=C_2(d)\ \text{ chosen so that }\
  c\,\Delta^{-(d+1)/2}\,\Xi_\ast\,\sqrt d\,\varpi\ \le\ \frac{\varrho\epsilon}4 .
\]
Such a $C_2$ exists and depends on $d$ alone, since $\Xi_\ast\le C_d\varrho K^d$ makes the
factor $\varrho$ cancel. Every $y\in Q_{K'}$ lies within distance $\sqrt d\,\varpi$ of a point
of $\mathcal Y$, so on the event $\{\Xi\le\Xi_\ast\}$ the estimate \eqref{eq:HJlip} gives
$\sup_{Q_{K'}}H_J\le\max_{\mathcal Y}H_J+\varrho\epsilon/4$. Moreover $K'\le K$ and
$K^{d+1}\Delta^{-(d+1)/2}\ge1$, which gives
\begin{equation}
\label{eq:netsize}
  |\mathcal Y|\ \le\ \Big(1+\frac{K'}\varpi\Big)^d\ \le\
  \Big(C_3\,\frac{K^{d+1}}{\epsilon\,\Delta^{(d+1)/2}}\Big)^{d}
  \ =\ C_3^{\,d}\,\epsilon^{-d}\Big(\frac K{\Delta^{1/2}}\Big)^{d(d+1)} .
\end{equation}
Combining \eqref{eq:Xibound} and \eqref{eq:netsize} with Markov's inequality applied at each
point of the net, we obtain
\begin{equation}
\label{eq:markovstep}
  \mathbb Q\Big(\sup_{y\in Q_{K'}}H_J(y)>\varrho(1+\epsilon)\Big)
  \ \le\ e^{-M}+|\mathcal Y|\max_{y\in\mathcal Y}
  e^{-\varkappa\varrho(1+3\epsilon/4)}\,\E_{\mathbb Q}e^{\varkappa H_J(y)} .
\end{equation}

Next we bound the exponential moment. Since $J$ is finite and the $\xi_j$ are independent,
\begin{equation}
\label{eq:taylorApplicationUp}
  \E_{\mathbb Q}e^{\varkappa H_J(y)}=\prod_{j\in J}\E_{\mathbb Q}
  \exp\big(\varkappa\,q_\Delta(x_j,y)\,\xi_j\big)
  =\prod_{j\in J}\big(1-\varkappa\,q_\Delta(x_j,y)\big)^{-1},
\end{equation}
which is finite as soon as $\varkappa q_\Delta(x_j,y)<1$ for every $j$. The moment generating
function of an exponential variable is finite only under such a constraint. The choice of
$\varkappa$ in \eqref{eq:kappachoice} is made so that the constraint holds uniformly in $j$ and
$y$, and \eqref{eq:supkappa}, which follows from (H\ref{H1}), is the form in which we use it.
We apply the inequality $(1-x)^{-1}\le\exp(x+x^2)$, valid for $0\le x\le1/2$, to
\eqref{eq:taylorApplicationUp}, and then \eqref{eq:supkappa}, and obtain
\begin{align*}
  \prod_{j\in J}\big(1-\varkappa q_\Delta(x_j,y)\big)^{-1}
  &\le\prod_{j\in J}\exp\Big(\varkappa q_\Delta(x_j,y)
     \big(1+\varkappa q_\Delta(x_j,y)\big)\Big)\\
  &\le\exp\Big(\varkappa\sum_{j\in J}q_\Delta(x_j,y)
     \Big(1+\sup_{x\in\R^d}\varkappa q_\Delta(x,y)\Big)\Big)\\
  &\le\exp\Big(\varkappa\sum_{j\in J}q_\Delta(x_j,y)\,(1+\epsilon/4)\Big).
\end{align*}
It remains to show that
\begin{equation}
\label{eq:sumupper}
  \sum_{j\in J}q_\Delta(x_j,y)\ \le\ \varrho\,(1+\epsilon/4)
  \qquad\text{for every }y\in Q_{K'}.
\end{equation}
Once this is established the proof is complete, since then the maximum in \eqref{eq:markovstep}
is at most
\[
  \exp\Big(\varkappa\varrho\big[(1+\tfrac\epsilon4)(1+\tfrac\epsilon4)
  -(1+\tfrac{3\epsilon}4)\big]\Big)
  =\exp\Big(\varkappa\varrho\big[\tfrac{\epsilon^2}{16}-\tfrac\epsilon4\big]\Big)
  \le\exp\Big(-\frac{3\varkappa\varrho\epsilon}{16}\Big)\ \le\ e^{-M}
\]
for $\epsilon\le1$. Hence the right-hand side of \eqref{eq:markovstep} is at most
$(1+|\mathcal Y|)e^{-M}$, which by \eqref{eq:netsize} and \eqref{eq:kappachoice} is at most
$\Gamma\exp\{-C\varrho\epsilon^2\Delta^{d/2}\}$ with $C=\tilde C/8$ and $\Gamma$ as in
\eqref{eq:Gamma}. This is \eqref{eq:Mplus}.

Fix $y\in Q_{K'}$ and split $J$ into the near field
$J(y)=\{j\in J:|x_j-y|\le \mathsf L\}$ and its complement. We estimate both parts. The second
part contributes a Gaussian tail.

We start with the near field particles. We compare each particle with the point of its own
subcube at which $q_\Delta(\cdot,y)$ is smallest. For each subcube $T_i$ and each $x_j\in T_i$ set
\[
  x_j'=\operatorname*{argmin}_{w\in T_i}q_\Delta(w,y),
\]
choosing the point with the smallest first coordinate, then the smallest second coordinate, and
so on, if the minimiser is not unique. The triangle inequality gives
\[
  \sum_{j\in J(y)}q_\Delta(x_j,y)\ \le\ \sum_{j\in J(y)}q_\Delta(x_j',y)
  +\sum_{j\in J(y)}\big|q_\Delta(x_j,y)-q_\Delta(x_j',y)\big| .
\]
Let $I(y)$ be the set of subcubes containing at least one $x_j$ with $j\in J(y)$. By hypothesis
each subcube contains at most $\varrho|T_i|$ particles, and the minimum of a function over a set
is at most its average, so
\[
  \sum_{j\in J(y)}q_\Delta(x_j',y)
  =\sum_{i\in I(y)}\#\{j\in J(y):x_j\in T_i\}\min_{w\in T_i}q_\Delta(w,y)
  \le\sum_{i\in I(y)}\varrho\int_{T_i}q_\Delta(z,y)\,\mathrm dz
  \le\varrho ,
\]
where the last step holds because the subcubes are disjoint and
$\int_{\R^d}q_\Delta(z,y)\mathrm dz=1$.

We now turn to the Hölder term. Every subcube of $I(y)$ is contained in
$B(y,\mathsf L+\sqrt d\,\ell)$, so $\sum_{i\in I(y)}|T_i|\le\hat c\,\mathsf L^d$ for a constant
$\hat c=\hat c(d)$, using $\ell\le \mathsf L$. By (H\ref{H2}),
\begin{align*}
  \sum_{j\in J(y)}\big|q_\Delta(x_j,y)-q_\Delta(x_j',y)\big|
  &\le\sum_{i\in I(y)}\varrho|T_i|\,c\,\Delta^{-(d+\kappa)/2}(\sqrt d\,\ell)^{\kappa}\\
  &\le c'\varrho\,\mathsf L^d\ell^{\kappa}\Delta^{-(d+\kappa)/2}
  =c'c_1^d\,\varrho\,\epsilon^{-1}\Big(\frac{\ell^2}{\Delta}\Big)^{\kappa/2} ,
\end{align*}
where we used $\mathsf L^d=c_1^d\Delta^{d/2}\epsilon^{-1}$. Since
$\Delta\ge c_0\ell^2\epsilon^{-4/\kappa}$ we have $\ell^2/\Delta\le c_0^{-1}\epsilon^{4/\kappa}$,
so the right-hand side is at most $c'c_1^dc_0^{-\kappa/2}\varrho\,\epsilon$, which is at most
$\varrho\epsilon/8$ once $c_0$ is chosen large enough with respect to $c_1$ and $d$.

We now consider the far field particles.
Since $\Delta\ge c_0\ell^2$ we have $\ell\le\Delta^{1/2}$ for $c_0\ge1$,
and $\mathsf L\ge c_1\Delta^{1/2}$, so $\mathsf L':=\mathsf L-\sqrt d\,\ell\ge \mathsf L/2\ge2\sqrt d\,\ell$ provided
$c_1\ge4\sqrt d$. Every $j\notin J(y)$ lies in a subcube $T_i$ with
$\mathrm{dist}(y,T_i)\ge \mathsf L'$, and such subcubes are contained in $\{z:|z-y|\ge \mathsf L'\}$. For
$z\in T_i$ we have $|z-y|\le\mathrm{dist}(y,T_i)+\sqrt d\,\ell$, and so
\[
  |T_i|\,\sup_{x\in T_i}q_\Delta(x,y)
  \le\ell^d(2\pi\Delta)^{-d/2}e^{-\mathrm{dist}(y,T_i)^2/(2\Delta)}
  \le(2\pi\Delta)^{-d/2}\int_{T_i}\exp\Big\{-\frac{(|z-y|-\sqrt d\,\ell)^2}{2\Delta}\Big\}
  \mathrm dz .
\]
Summing over the far subcubes, which are disjoint, and using (H\ref{H4}) with $R=\mathsf L'$ and
$\varsigma=\sqrt d\,\ell$, we get
\begin{align*}
  \sum_{j\in J\setminus J(y)}q_\Delta(x_j,y)
  &\le\varrho\,(2\pi\Delta)^{-d/2}\int_{|z-y|\ge \mathsf L'}
  \exp\Big\{-\frac{(|z-y|-\sqrt d\,\ell)^2}{2\Delta}\Big\}\,\mathrm dz\\
  &\le C\varrho\,e^{-\mathsf L'^2/(16\Delta)}\ \le\ C\varrho\,e^{-c_1^2\epsilon^{-2/d}/64},
\end{align*}
using $\mathsf L'\ge \mathsf L/2$ and $\mathsf L^2/\Delta=c_1^2\epsilon^{-2/d}$ in the last step. Since
$\sup_{\epsilon\in(0,1)}\epsilon^{-1}e^{-c_1^2\epsilon^{-2/d}/64}\to0$ as $c_1\to\infty$, this
is at most $\varrho\epsilon/8$ once $c_1=c_1(d)$ is large enough, uniformly in
$\epsilon\in(0,1)$.

Adding the three contributions gives \eqref{eq:sumupper} and completes the proof.
\end{proof}

\begin{proof}[Proof of \Cref{thm:mixinglower}, sketch]
We follow the scheme of the proof just given, with $\rho\equiv\varrho(1-\epsilon)$ in place of
$\varrho(1+\epsilon)$ and \Cref{prop:softlocal}(\ref{sl:low}) in place of (\ref{sl:up}), so that a
lower bound on an infimum replaces an upper bound on a supremum. One preliminary reduction is
needed which has no counterpart above.

The hypothesis here bounds the occupancy of a subcube from below and not from
above, so $|J|$ is not controlled and neither \eqref{eq:HJlip} nor \eqref{eq:Xibound} is available
as it stands. Since counts are integers, a subcube carrying at least $\varrho\ell^d$ points
carries at least $\lceil\varrho\ell^d\rceil$ of them. Let $\tilde J\subseteq J$ consist of exactly
$\lceil\varrho\ell^d\rceil$ points of each subcube, chosen by any fixed rule. Then
\[
  n:=|\tilde J|=\lceil\varrho\ell^d\rceil\Big(\frac K\ell\Big)^d\ \le\ 2\varrho K^d ,
\]
since $\lceil v\rceil\le2v$ for $v\ge1$ and $\varrho\ell^d\ge1$ by the hypothesis
$\lfloor\varrho\ell^d\rfloor>0$. This is the only use made of that hypothesis. It suffices to work
with the thinned family. Indeed, applying \Cref{prop:softlocal}(\ref{sl:low}) to
$(Y_j)_{j\in\tilde J}$ with $\rho\equiv\varrho(1-\epsilon)$ and $A=Q_{K'}$ yields
$\psi\cap Q_{K'}\subseteq\{Y_j\}_{j\in\tilde J}\subseteq\{Y_j\}_{j\in J}$, which is the conclusion.
It remains to show that $\inf_{Q_{K'}}H_{\tilde J}\ge\varrho(1-\epsilon)$ with the stated
probability, where $H_{\tilde J}(y)=\sum_{j\in\tilde J}\xi_j\,q_\Delta(x_j,y)$.

With $\tilde J$ in place of $J$ the reduction of the infimum to a minimum over a
finite net is the same computation. The Lipschitz estimate \eqref{eq:HJlip} does not depend on the
direction, and the bound $n\le2\varrho K^d$ restores \eqref{eq:Xibound} with
$\Xi_\ast:=8\varrho K^d+2M\le C_d\varrho K^d$, so that $\varpi$ may be chosen as before and
\eqref{eq:netsize} holds unchanged. This produces the same factor $\Gamma$, with a larger constant
depending only on $d$. The choices \eqref{eq:kappachoice} of $\varkappa$ and $M$ are unchanged.
Three steps differ somewhat however. 
\begin{enumerate}[(i)]
\item The Chernoff step is applied with the exponent reversed:
  \[
    \Prob\big(H_{\tilde J}(y)<\varrho(1-\tfrac{3\epsilon}4)\big)
    \ \le\ e^{\varkappa\varrho(1-3\epsilon/4)}\,\E e^{-\varkappa H_{\tilde J}(y)} ,
  \]
  and the Laplace
  transform $\E e^{-\varkappa\xi q}=(1+\varkappa q)^{-1}$ replaces the moment generating function
  used in \eqref{eq:taylorApplicationUp}. It is finite for every $\varkappa>0$, so no constraint of
  the form \eqref{eq:supkappa} is needed for finiteness. That estimate is still used, to control the
  second-order term in $1+x\ge e^{x-x^2}$, which replaces $(1-x)^{-1}\le e^{x+x^2}$.
\item Only a lower bound on $\sum_jq_\Delta(x_j,y)$ is required, so the far field may be
  discarded and the sum restricted to a neighbourhood of $y$. In this direction the hypothesis
  $K-K'\ge c_1\Delta^{1/2}\epsilon^{-1/d}$ is used as follows. Since
  $Q_K=[-K/2,K/2]^d$, it places $Q_{K'}$ at distance at least $\mathsf L/2$ from the complement of
  $Q_K$, so that for $y\in Q_{K'}$ every subcube contained in $B(y,\mathsf L/2)$ is one of the
  subcubes of the tessellation and carries at least $\varrho\ell^d$ points. Running the near field
  at radius $\mathsf L/2$ rather than $\mathsf L$ changes only the value of $c_1$.
\item The association step takes the maximiser of $q_\Delta(\cdot,y)$ over each subcube rather
  than the minimiser. Put $\mathsf L':=\mathsf L/2-\sqrt d\,\ell$, so that $\mathsf L'\ge\mathsf L/4$ by the
  inequalities $\ell\le\Delta^{1/2}$ and $\mathsf L\ge c_1\Delta^{1/2}$ with $c_1\ge4\sqrt d$ used
  in the far field above. Let $I(y)$ be the set of subcubes contained in $B(y,\mathsf L/2)$, whose
  union contains $B(y,\mathsf L')$, and let $\tilde J(y)$ be the points of $\tilde J$ lying in them.
  The maximum of a function over a set is at least its average, so
  \begin{align*}
    \sum_{j\in\tilde J(y)}q_\Delta(x_j',y)
    &=\sum_{i\in I(y)}\lceil\varrho\ell^d\rceil\max_{w\in T_i}q_\Delta(w,y)
    \ \ge\ \frac{\lceil\varrho\ell^d\rceil}{\ell^d}\sum_{i\in I(y)}\int_{T_i}q_\Delta(z,y)\,\mathrm dz\\
    &\ \ge\ \varrho\int_{B(y,\mathsf L')}q_\Delta(z,y)\,\mathrm dz ,
  \end{align*}
  which is an integral over a ball rather than over all of $\R^d$. This step is the reason that
  the thinning has to be to $\lceil\varrho\ell^d\rceil$ and not to $\lfloor\varrho\ell^d\rfloor$.
  The step requires that the effective density $\lceil\varrho\ell^d\rceil\ell^{-d}$ be at least
  $\varrho$, whereas $\lfloor\varrho\ell^d\rfloor\ell^{-d}$ can be as small as $\varrho/2$. The
  missing mass is estimated as the far field was, by (H\ref{H4}) with $\varsigma=0$ and
  $R=\mathsf L'$:
  \[
    \varrho\int_{|z-y|\ge\mathsf L'}q_\Delta(z,y)\,\mathrm dz\ \le\
    C\varrho\,e^{-\mathsf L'^2/(16\Delta)}\ \le\ C\varrho\,e^{-c_1^2\epsilon^{-2/d}/256}
    \ \le\ \frac{\varrho\epsilon}8
  \]
  for $c_1=c_1(d)$ large, uniformly in $\epsilon\in(0,1)$. The H\"older error term is as before
  with $\lceil\varrho\ell^d\rceil\le2\varrho|T_i|$ in place of $\varrho|T_i|$, and the factor $2$
  is absorbed by $c_0$. It is again at most $\varrho\epsilon/8$. Together these give
  $\sum_{j\in\tilde J(y)}q_\Delta(x_j,y)\ge\varrho(1-\epsilon/4)$, which is \eqref{eq:sumupper} with
  both inequalities reversed.
\end{enumerate}
The final arithmetic then reads
$\exp\big(\varkappa\varrho[(1-\tfrac{3\epsilon}4)-(1-\tfrac\epsilon4)^2]\big)
=\exp\big(-\varkappa\varrho[\tfrac\epsilon4+\tfrac{\epsilon^2}{16}]\big)\le e^{-2M}\le e^{-M}$, and
on the event $\{\Xi\le\Xi_\ast\}$ the net gives
$\inf_{Q_{K'}}H_{\tilde J}\ge\min_{\mathcal Y}H_{\tilde J}-\varrho\epsilon/4
\ge\varrho(1-\tfrac{3\epsilon}4)-\varrho\epsilon/4=\varrho(1-\epsilon)$, so that the bound
$(1+|\mathcal Y|)e^{-M}\le\Gamma\exp\{-C\varrho\epsilon^2\Delta^{d/2}\}$ follows as before.
The details are in \cite{GJLV26}, where the statement is proved for $\alpha$-stable motion and
specialises to the above at $\alpha=2$, and in \cite[Theorem~5]{GracarStauffer19a} for random walks
on a uniformly elliptic conductance graph over $\Z^d$.
\end{proof}

\subsection{The kernel of a conditioned particle}
\label{sec:condkernel}

The multi-scale percolation of \Cref{sec:lipschitz} evaluates its local events not for the free
system but for one in which every particle is conditioned to keep its displacement inside a cube
throughout a time interval (\Cref{def:restricted}). We show in this section that for Brownian
motion the transition kernel of a particle so conditioned obeys the same estimates as the free
one, with constants depending only on $d$. On a lattice with random
conductances the corresponding statement is \cite[Lemma 1]{GracarStauffer19a}, where no product
structure is available. Here the cube is a product and the coordinates of a Brownian motion are
independent, which makes everything explicit.

Fix $\mathsf a,\Delta>0$, put $Q_\mathsf a=[-\mathsf a/2,\mathsf a/2]^d$, and let
\[
  \mathcal D=\mathcal D(x,\mathsf a,\Delta)
  =\big\{B_s-x\in Q_\mathsf a\ \text{for every }s\in[0,\Delta]\big\},
  \qquad h=h(\mathsf a,\Delta)=\Prob_x[\mathcal D] ,
\]
which does not depend on $x$. Write $g_\Delta(x,\cdot)$ for the density of $B_\Delta$ under
$\Prob_x[\,\cdot\mid\mathcal D]$. We now prove the estimates for the conditioned kernel.

\begin{lemma}[The conditioned kernel]
\label{lem:condkernel}
There is $C_0=C_0(d)$ such that if $\mathsf a\ge C_0\sqrt\Delta$ then, with constants depending only
on $d$:
\begin{enumerate}[(i)]
\item\label{ck:h} $h\ge\tfrac12$;
\item\label{ck:up} $g_\Delta(x,y)\le2\,q_\Delta(x,y)\le2(2\pi\Delta)^{-d/2}$ for all $x,y$;
\item\label{ck:lip} $|g_\Delta(x_1,y)-g_\Delta(x_2,y)|\le C_d\,\Delta^{-(d+1)/2}|x_1-x_2|$;
\item\label{ck:exit} for all $R>0$ and all $t>0$,
  $\Prob_x\big[\sup_{s\le t}|B_s-x|\ge R\ \big|\ \mathcal D\big]\le 8d\,e^{-R^2/(2dt)}$;
\item\label{ck:mass} $\int_{\R^d}g_\Delta(x,y)\,\mathrm dx=1$ for every $y$.
\end{enumerate}
\end{lemma}

\begin{proof}
Since the conditioning event is a condition on the increment $B-x$, the conditioned kernel is
translation invariant: $g_\Delta(x,y)=k_\Delta(y-x)/h$, where $k_\Delta$ is the density at time
$\Delta$ of a Brownian motion started at $0$ and killed on leaving $Q_\mathsf a$. The coordinates of
$B$ are independent and $Q_\mathsf a$ is a product of intervals, so the killed motion is a product of
$d$ independent one-dimensional killed motions and
\begin{equation}
\label{eq:kprod}
  k_\Delta(z)=\prod_{i=1}^dk^{(1)}_\Delta(z_i) ,
\end{equation}
where $k^{(1)}_\Delta$ is the density at time $\Delta$ of a one-dimensional Brownian motion
started at $0$ and killed on leaving $[-\mathsf a/2,\mathsf a/2]$.

We start with (\ref{ck:h}). By the reflection principle, a one-dimensional Brownian motion leaves
$[-\mathsf a/2,\mathsf a/2]$ before time $\Delta$ with probability at most
$4e^{-\mathsf a^2/(8\Delta)}$, so $h\ge(1-4e^{-\mathsf a^2/(8\Delta)})^d$, which exceeds $\tfrac12$
once $\mathsf a\ge C_0\sqrt\Delta$ with $C_0=C_0(d)$ large enough.

For (\ref{ck:up}), note that killing only decreases a transition density, so $k_\Delta\le q_\Delta$
pointwise. Dividing by $h\ge\tfrac12$ gives the claim.

For (\ref{ck:lip}) we use the method of images on the interval $[-\mathsf a/2,\mathsf a/2]$, which
gives
\[
  k^{(1)}_\Delta(z)=\sum_{n\in\Z}
  \Big[q^{(1)}_\Delta(z-2n\mathsf a)-q^{(1)}_\Delta(z+\mathsf a-2n\mathsf a)\Big] ,
\]
where $q^{(1)}$ is the one-dimensional Gaussian kernel. Now
\[
  \big|\partial_uq^{(1)}_\Delta(u)\big|=\frac{|u|}{\Delta}\,q^{(1)}_\Delta(u)
  \ \le\ C\Delta^{-1}e^{-u^2/(4\Delta)} ,
\]
and for $|z|\le\mathsf a/2$ the $n$-th pair of arguments is at least $(2|n|-1)\mathsf a/2$ in absolute
value, so differentiating the series term by term gives
\[
  \big|\partial_zk^{(1)}_\Delta(z)\big|\ \le\ C\Delta^{-1}
  \sum_{n\in\Z}e^{-c\,n^2\mathsf a^2/\Delta}\ \le\ C\Delta^{-1} ,
\]
where the sum is bounded by a constant because $\mathsf a\ge C_0\sqrt\Delta$. Differentiating
\eqref{eq:kprod} in the $i$-th coordinate and bounding the remaining $d-1$ factors by
$(2\pi\Delta)^{-1/2}$ gives $\|\nabla k_\Delta\|_\infty\le C_d\Delta^{-(d+1)/2}$, and the claim
follows from the mean value theorem and $h\ge\tfrac12$.

For (\ref{ck:exit}), note that conditioning on an event of probability at least $\tfrac12$ at most
doubles probabilities, and the free bound is
$\Prob_x[\sup_{s\le t}|B_s-x|\ge R]\le4d\,e^{-R^2/(2dt)}$, which is (H\ref{H3}). Note that this holds
for every $t>0$ and not only for $t=\Delta$, since the conditioning event $\mathcal D$ is fixed
and only its probability is used. The intermediate times are needed, since \Cref{sec:smallmu}
applies the estimate at $t=T_L\ll\beta$.

Finally, (\ref{ck:mass}) follows from translation invariance and $\int k_\Delta=h$:
$\int g_\Delta(x,y)\,\mathrm dx=h^{-1}\int k_\Delta(y-x)\,\mathrm dx=1$.
\end{proof}

In other words, (\ref{ck:up}) and (\ref{ck:lip}) are (H\ref{H1}) and (H\ref{H2}) with $\kappa=1$ up to
constants, and (\ref{ck:exit}) is (H\ref{H3}), for every $t>0$ and not only at $t=\Delta$; the
transfer of (H\ref{H4}) is part of \Cref{cor:mixingcond} below. Moreover, (\ref{ck:mass}) says that a Poisson system whose
particles are independently conditioned in this way retains its intensity at later times away from
the boundary.
Every upper estimate of \Cref{sec:survupper,sec:smallmu,sec:vertmix} therefore holds under
the conditioned law with constants depending only on $d$.

\subsection{Local mixing for conditioned particles}
\label{sec:condmix}

We are now ready to transfer the local mixing estimates to conditioned particles.

\begin{corollary}[Local mixing for conditioned particles]
\label{cor:mixingcond}
Let $\mathsf a\ge C_0\sqrt\Delta$ with $C_0=C_0(d)$ as in \Cref{lem:condkernel}, and let each $Y_j$ be
the position at time $\Delta$ of a Brownian motion started at $x_j$ and conditioned to keep its
displacement inside $Q_\mathsf a$ throughout $[0,\Delta]$, where the conditionings are performed
independently over $j$. Then the conclusions of \Cref{thm:mixinglower,thm:mixingupper} hold
verbatim, with the constants $c_0,c_1,C$ replaced by others depending only on $d$. Nothing below
depends on the values of these constants or on which of the two laws they belong to, and we write
$c_0,c_1,C$ for either, as generic constants of the local mixing estimate.
\end{corollary}

\begin{proof}
Both statements concern the law of the family $(Y_j)_{j\in J}$, and both proofs use that law only
through the following four properties. The $Y_j$ are independent. The point $Y_j$ has density
$q_\Delta(x_j,\cdot)$, which enters via the estimates (H\ref{H1}), (H\ref{H2}) and (H\ref{H4}) of
\Cref{lem:kernel}. The kernel $q_\Delta$ is symmetric, which is how (H\ref{H2}), stated in the first
variable, is applied in the second. Finally, $\int q_\Delta(x,y)\,\mathrm dx=1$. No use is made of
the Markov property of the underlying paths, nor of anything about the trajectories between $0$ and
$\Delta$. In particular (H\ref{H3}) is not used by either proof, and we list it below only because the
transported kernel is used elsewhere in the paper. The transported kernel is symmetric as required,
since $g_\Delta(x,y)=k_\Delta(y-x)/h$ with $k_\Delta$ even.

Independence survives the conditioning, since each particle is conditioned on an event depending
only on its own path. The density of $Y_j$ becomes $g_\Delta(x_j,\cdot)$, which by
\Cref{lem:condkernel}(\ref{ck:up}) satisfies (H\ref{H1}) with the constant doubled, by
(\ref{ck:lip}) satisfies (H\ref{H2}) with $\kappa=1$, by (\ref{ck:exit}) satisfies (H\ref{H3}), and by
(\ref{ck:mass}) integrates to $1$ in the starting variable. For (H\ref{H4}), which bounds an integral of
a Gaussian tail, the pointwise domination
$g_\Delta\le2q_\Delta$ of (\ref{ck:up}) suffices. Substituting $g_\Delta$ for $q_\Delta$ throughout
therefore leaves both proofs intact.
\end{proof}
\section{Multi-scale Lipschitz percolation}
\label{sec:lipschitz}

Both survival results, \Cref{prop:crowding} and \Cref{thm:smallmu}, are obtained in the same way.
We define a local increasing event of high probability and apply to it the multi-scale Lipschitz
percolation of Gracar and Stauffer \cite{GracarStauffer19b}. The two applications use the same
framework with different local events.

\subsection{Tessellation and local events}
\label{sec:tessellation}

Fix a spatial scale $L>0$ and a constant $\blk>0$, and set $\beta=\blk L^2$. The ratio
$\beta/L^2=\blk$ will always be a small fixed constant, and $L$ will be taken large. Fix also
an integer \emph{overlap parameter} $\olap\ge\max\{d,2\}$. For $i\in\Z^d$ and $\tau\in\Z$ we define
\[
  C_i=\prod_{j=1}^d\big[i_jL,(i_j+1)L\big],\qquad
  \Qst(i)=\prod_{j=1}^d\big[(i_j-\olap)L,(i_j+\olap+1)L\big] ,
\]
together with the \emph{reduced super-cube}
\[
  \Qst_-(i)=\prod_{j=1}^d\big[(i_j-\olap+1)L,(i_j+\olap)L\big]\ \subset\ \Qst(i) ,
\]
whose boundary is at distance $L$ from that of $\Qst(i)$. We call $C_i$ the \emph{central cube} of
$i$ and $\Qst(i)$ the \emph{super-cube}. The
\emph{cell} $\cell{i}{\tau}$ is $C_i\times[\tau\beta,(\tau+1)\beta]$ and the \emph{super-cell}
$\scell{i}{\tau}$ is $\Qst(i)\times[\tau\beta,(\tau+\olap)\beta]$. We write $Q_z=[-z/2,z/2]^d$.

\begin{definition}[Restricted, increasing, and the associated probability]
\label{def:restricted}
An event $A$ is \emph{restricted} to a region $X\subseteq\R^d$ and a time interval $[t_0,t_1]$ if
it is measurable with respect to the $\sigma$-field generated by the particles lying in $X$ at
time $t_0$, their trajectories on $[t_0,t_1]$ and the marks those particles carry on $[t_0,t_1]$.
When the isolation rule is in force, the $\sigma$-field also includes the trajectories on
$[t_0,t_1]$ of the particles lying within distance $2r$ of $X$ at time $t_0$, since these decide
whether those marks are effective.
A particle has \emph{displacement inside} $X'$ during $[t_0,t_1]$ if $X_u(t)-X_u(t_0)\in X'$ for
every $t\in[t_0,t_1]$. An event is \emph{increasing} if it is preserved by adding particles, in
the sense of \Cref{prop:mono}(ii). For an increasing event $E$ restricted to $X$ and $[0,t]$, we
write $\nu_E(\zeta,X,X',t)$ for the probability that $E$ holds when the particles in $X$ at time
$0$ form a Poisson process of intensity $\zeta$ and are conditioned to have displacement inside
$X'$ during $[0,t]$.
\end{definition}

The clause about particles within distance $2r$ of $X$ is needed only under $\mathfrak r=\iso$.
In that case, by \Cref{def:openpath}, whether a recovery mark is effective depends on the
configuration near the particle carrying it. Both events to which \Cref{thm:GS} is applied below
are in any case insensitive to this clause. The crowding event of \Cref{prop:crowding} involves
no marks at all, and the acceptable-cell event of \Cref{sec:cellevent} is defined under
$\mathfrak r=\unc$.

For each $(i,\tau)\in\Z^{d+1}$ let $E(i,\tau)$ be an increasing event restricted to the super-cell
$\scell{i}{\tau}$, invariant under space-time translations of the tessellation. We say that the
cell $(i,\tau)$ is \emph{good} if $E(i,\tau)$ holds.

\subsection{The base-height index and the two-sided Lipschitz surface}
\label{sec:baseheight}

Following \cite{GracarStauffer19b}, we pick one of the $d$ spatial dimensions and call it the
\emph{height}, indexed by $h\in\Z$. The remaining $d-1$ spatial dimensions together with the time
dimension form the \emph{base}, indexed by $b\in\Z^d$. This is a fixed permutation of the
coordinates of $\Z^{d+1}$, so every space-time cell $(i,\tau)$ has a base-height index $(b,h)$ and
conversely. Note that the base contains the time coordinate, and that for $d\ge2$ it contains at
least one spatial coordinate as well. This is the origin of the restriction $d\ge2$ in
\Cref{thm:smallmu}. We refer to a cell either by $(i,\tau)$ or by its base-height index $(b,h)$.

A function $F:\Z^d\to\Z$ is \emph{Lipschitz} if $|F(x)-F(y)|\le1$ whenever $\|x-y\|_1=1$.

\begin{definition}[Two-sided Lipschitz surface]
\label{def:surface}
A \emph{two-sided Lipschitz surface} $F$ is a set of base-height cells $(b,h)\in\Z^{d+1}$ such
that for every $b\in\Z^d$ there are exactly two, possibly equal, integers $F_+(b)\ge0$ and
$F_-(b)\le0$ with $(b,F_+(b)),(b,F_-(b))\in F$, and such that $F_+$ and $F_-$ are Lipschitz
functions. We say that $F$ \emph{exists} if $F_+(b)<\infty$ and $F_-(b)>-\infty$ for every $b$.
For $D\in\N$ we say that $F$ \emph{surrounds} a cell $(b',h')$ at distance $D$ if every path
$(b',h')=(b_0,h_0),(b_1,h_1),\dots,(b_n,h_n)$ with $\|(b_k,h_k)-(b_{k-1},h_{k-1})\|_1=1$ for all
$k$ and $\|(b_n,h_n)-(b_0,h_0)\|_1>D$ intersects $F$. We say that a space-time cell belongs to
$F$ if its base-height index does.
\end{definition}

A realisation is shown in \Cref{fig:surface}.

\begin{figure}[!h]
\centering
\includegraphics[width=0.85\textwidth]{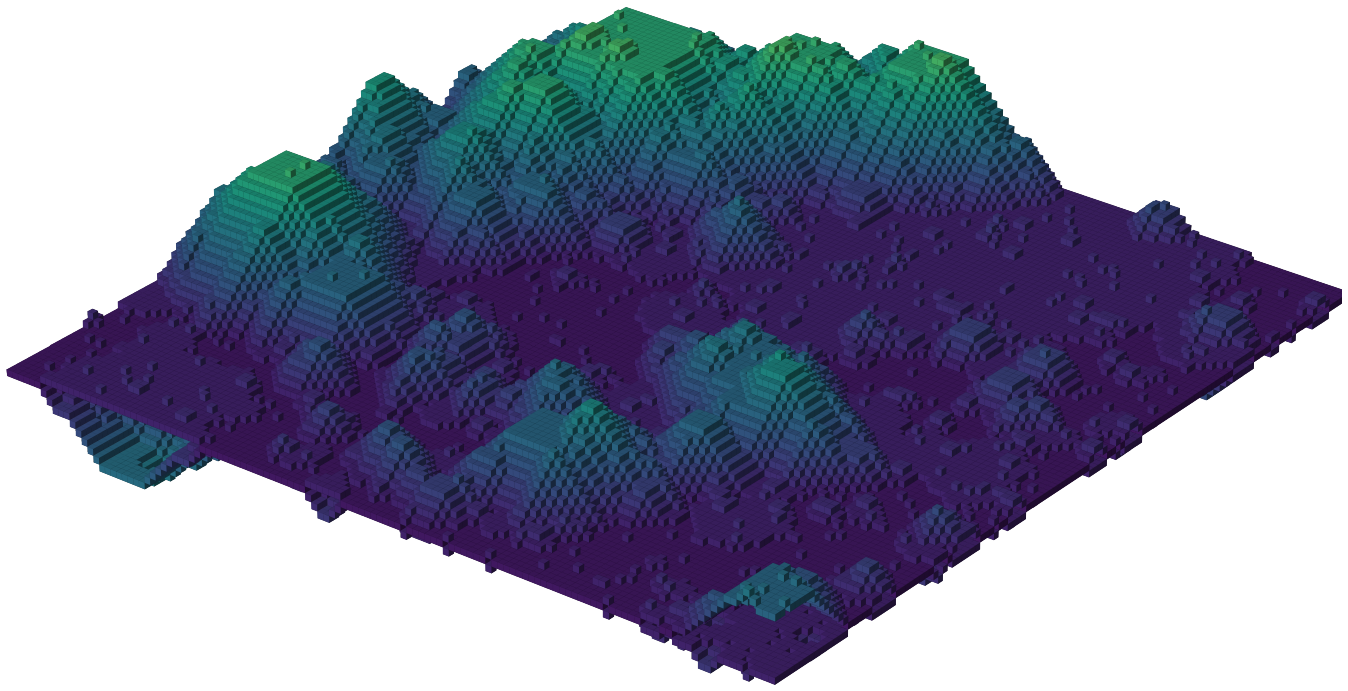}
\caption{A realisation of a two-sided Lipschitz surface over a two-dimensional base. Over every
base index $b$ the surface contains the two cells at heights $F_+(b)\ge0$ and $F_-(b)\le0$,
coloured by their distance from the plane $h=0$. Where all nearby cells are good both functions
vanish; the cells which are not good force the deviations, which enclose them from above and
from below.}
\label{fig:surface}
\end{figure}

\subsection{The existence of the surface}
\label{sec:GSthm}

The two theorems of this subsection are Theorems 2.1 and 2.2 of \cite{GracarStauffer19b},
proved there for Poisson systems
of continuous-time random walks on a uniformly elliptic conductance graph over $\Z^d$, and we
quote them in the corresponding form for Poisson Brownian motions in $\R^d$, where the
ellipticity constant of \cite{GracarStauffer19b} equals $1$. This form
is available because the framework uses the motion as a black box. The proofs in
\cite{GracarStauffer19b} depend on the underlying space only through applications of a local
mixing estimate of the form of \Cref{thm:mixinglower}, applied to the particles of a region
conditioned as in \Cref{def:restricted}. We prove this estimate for Brownian
motions in \Cref{sec:vertmix,sec:condmix}, as \Cref{thm:mixinglower} for free particles and
\Cref{cor:mixingcond} for conditioned ones. With the mixing estimate in place, the proofs apply unchanged apart from the values of
the constants.

\begin{theorem}[Existence of the surface]
\label{thm:GS}
There are positive constants $\mathsf c_0,\mathsf c_1,\mathsf c_2$, depending only on $d$, such that the following holds.
Tessellate as in \Cref{sec:tessellation} with $\beta/L^2<\mathsf c_0$, and let $E(i,\tau)$ be an
increasing event restricted to the super-cell $\scell{i}{\tau}$. Fix $\varepsilon\in(0,1)$ and fix
$w$ such that
\[
  w\ \ge\ \sqrt{\frac{\olap\beta}{\mathsf c_2L^2}\,\log\Big(\frac{8\mathsf c_1}{\varepsilon}\Big)} .
\]
Then there is $\alpha_0>0$, depending on $\varepsilon$, $\olap$, $w$ and the ratio $\beta/L^2$ but
not on $L$ nor on the intensity, such that if
\begin{equation}
\label{eq:GScond}
  \min\Big\{\varepsilon^2\eta L^d,\
  \log\frac1{1-\nu_E\big((1-\varepsilon)\eta,\,Q_{(2\olap+1)L},\,Q_{wL},\,\beta\big)}\Big\}
  \ \ge\ \alpha_0 ,
\end{equation}
then almost surely a two-sided Lipschitz surface $F$ exists on which $E(i,\tau)$ holds for every
$(i,\tau)\in F$.
\end{theorem}

\begin{theorem}[Surrounding the origin]
\label{thm:GSsurround}
Assume the hypotheses of \Cref{thm:GS}. There are $c,C>0$ and $D_1$, depending only on $d$, such
that for all $D\ge D_1$,
\[
  \Prob\big[F\text{ does not surround the origin at distance }D\big]\ \le\
  \begin{cases}
    \displaystyle\sum_{s\ge D}s\,\exp\Big\{-\frac{C\eta L\sqrt s}{(\log Ls)^{c}}\Big\}, & d=1,\\[2ex]
    \displaystyle\sum_{s\ge D}s^d\exp\Big\{-\frac{C\eta Ls}{(\log Ls)^{c}}\Big\}, & d=2,\\[2ex]
    \displaystyle\sum_{s\ge D}s^d\exp\{-C\eta Ls\}, & d\ge3 .
  \end{cases}
\]
In particular the right-hand side is summable in $D$ in every dimension, so by Borel--Cantelli $F$
surrounds the origin at an almost surely finite distance.
\end{theorem}

The right-hand side of \Cref{thm:GSsurround} is an explicit
function of $D,\eta,L$ and $d$ alone: the constant $\alpha_0$ is fixed once
$\varepsilon,\olap,w$ and $\blk$ are, and once $\nu_E$ satisfies \eqref{eq:GScond} the remaining
dependence is absorbed into $c$, $C$ and $D_1$. It is non-increasing in $\eta$ and, once
$L\ge L_1(d)$, non-increasing in $L$. These properties are used in the proof of
\Cref{thm:smallmu}. The $d=1$
line takes its rate from \cite[Lemma 5.6]{GracarStauffer19b}; see \Cref{rem:GSd1}.

\begin{remark}[The dimension $d=1$]
\label{rem:GSd1}
\Cref{thm:GS,thm:GSsurround} are stated in \cite{GracarStauffer19b} for $d\ge2$. Baldasso and
Stauffer \cite[Remark 2.4]{BaldassoStauffer23} observed that the proofs also apply at $d=1$. The
argument uses \cite[Proposition 6.1]{GracarStauffer19b}, which holds in dimension one as well,
and the decay stated there follows from \cite[Lemma 5.6]{GracarStauffer19b}. The $d=1$ bound
$C\sqrt t/(\log t)^{c}$ of that lemma is the source of the $d=1$ line of \Cref{thm:GSsurround}.
That line is a stretched exponential rather than an exponential. This makes no difference here,
since all that is used is summability against $\sum_s s$. This allows \Cref{prop:crowding} to be
applied at $d=1$, the only dimension in which it is needed. \Cref{thm:smallmu} is unaffected; it
requires $d\ge2$ for the independent reason given in \Cref{sec:baseheight}.
\end{remark}

Throughout we take $w$ such that
\begin{equation}
\label{eq:wlarge}
  w\ \ge\ C_0\sqrt\blk\qquad\text{and}\qquad w\ \ge\ 2(2\olap+1) ,
\end{equation}
which \Cref{thm:GS} permits, since its own requirement on $w$ is also a lower bound. In
\eqref{eq:GScond} the probability $\nu_E$ is evaluated under the conditioned law of
\Cref{def:restricted}, in which every particle of the region keeps its displacement inside
$Q_{wL}$ during a block. The framework is designed for events of this kind, and the
conditioning requires no argument beyond the following three observations.

First, the first condition of \eqref{eq:wlarge} is the hypothesis $wL\ge C_0\sqrt\beta$ of
\Cref{lem:condkernel}, so by \Cref{cor:mixingcond} the local mixing, and with it every upper
estimate derived from (H\ref{H1})--(H\ref{H4}), holds under the conditioned law with constants
depending only on $d$. The conditioning is imposed on each particle separately, so it preserves
independence between particles, and by \Cref{lem:condkernel}(\ref{ck:mass}) the conditioned
system retains its intensity at later times.

Second, the second condition of \eqref{eq:wlarge} gives $(2\olap+1)L\le wL/2$, so a trajectory
which remains in a super-cube throughout a block has displacement inside $Q_{wL}$. An event
which confines its particle in this way implies the event on which the law is conditioned, and
conditioning on an event only increases the probability of events which imply it. This holds
conditionally on the past of the block as well, on the event that the trajectory has so far
remained in the super-cube.

Third, if an event is determined by the trajectory on an initial part of the block together
with the marks, and keeps the particle at distance at least $L$ from the boundary of the
super-cube, then its conditioned probability is at least a constant multiple of its free one,
with a constant depending only on $\blk$ and $d$. Indeed, on such an event the particle remains
in the super-cube for the rest of the block with probability bounded below, by the Markov
property and Brownian scaling, and on the intersection the whole block is spent in the
super-cube, so the second observation applies.

Every event evaluated by \eqref{eq:GScond} below is of one of these forms, and we pass between
the free and the conditioned law accordingly.

\subsection{From a surface to survival}
\label{sec:propagation}

We now isolate the two steps which turn \Cref{thm:GS,thm:GSsurround} into a survival statement.
Both applications below produce an event with the following property.

\begin{definition}[Carrying the infection]
\label{def:carries}
An increasing event $E(i,\tau)$ restricted to the super-cell $\scell{i}{\tau}$ \emph{carries the
infection} if, on $E(i,\tau)$, whenever some particle lying in the central cube $C_i$ at time
$\tau\beta$ is infected then
\begin{enumerate}[(P1)]
\item\label{P1} for every $i'$ with $\|i'-i\|_\infty\le2$ there is, at time $(\tau+1)\beta$, an
  infected particle in the central cube $C_{i'}$; and
\item\label{P2} at every time of $[\tau\beta,(\tau+1)\beta]$ there is an infected particle in
  $\Qst(i)$.
\end{enumerate}
\end{definition}

\begin{lemma}[Propagation along the surface]
\label{lem:propagate}
Let $d\ge1$, let $E$ carry the infection in the sense of \Cref{def:carries}, and let $F$ be a
two-sided Lipschitz surface on which $E$ holds. Suppose that for some $(i,\tau)\in F$ there is an
infected particle in $C_i$ at time $\tau\beta$. Then the infection survives. If moreover $d\ge2$,
then
\[
  \liminf_{t\to\infty}\frac{\rad(I_t)}{t}\ \ge\ \frac L\beta\ =\ \frac1{\blk L}\ >\ 0 .
\]
\end{lemma}

\begin{proof}
Let $(b,h)$ be the base-height index of $(i,\tau)$. Recall that the base always contains the time
coordinate. Suppose first only that $d\ge1$. Let $b'$ be obtained from $b$ by increasing the time
coordinate by one, so that $\|b'-b\|_1=1$. Since $F_\pm$ are Lipschitz, there is $h'$ with
$(b',h')\in F$ and $|h'-h|\le1$, where we choose whichever of the two sides of $F$ contains
$(b,h)$. In space-time terms $(b',h')$ is a cell $(i',\tau+1)$ with $\|i'-i\|_\infty\le1\le\olap$.
Iterating this step and using (P\ref{P1}) and (P\ref{P2}) as below shows that the infection survives.

For the speed we assume $d\ge2$, so that the base contains at least one spatial coordinate besides
the time coordinate. We fix one such coordinate and call it $b_1$. Let $b'$ now be obtained from
$b$ by increasing both the time coordinate and the coordinate $b_1$ by one, so that
$\|b'-b\|_1=2$. The Lipschitz property gives $h'$ with $(b',h')\in F$ and $|h'-h|\le2$. In
space-time terms $(b',h')$ is a cell $(i',\tau+1)$ with $\|i'-i\|_\infty\le2\le\olap$, whose $b_1$
coordinate is that of $i$ increased by one.

By (P\ref{P1}) applied to $(i,\tau)$ there is an infected particle in $C_{i'}$ at time $(\tau+1)\beta$.
This is the hypothesis of (P\ref{P1}) for the cell $(i',\tau+1)\in F$. Iterating, we obtain cells
$(i^{(n)},\tau+n)\in F$, $n\ge0$, with $i^{(0)}=i$, with an infected particle in $C_{i^{(n)}}$ at
time $t_n:=(\tau+n)\beta$, and with the $b_1$ coordinate of $i^{(n)}$ equal to that of $i$ plus
$n$. In particular $I_{t_n}\ne\emptyset$ for every $n$, and by (P\ref{P2}) the infection is non-empty
at every intermediate time as well. Hence the infection survives.

For the speed, (P\ref{P2}) gives for every $t\in[t_n,t_{n+1}]$ an infected particle in
$\Qst(i^{(n)})$, whose $b_1$ coordinate is therefore at least $(i_{b_1}+n-\olap)L$. Hence
$\rad(I_t)\ge(n-\olap-|i_{b_1}|)L$ for such $t$, while $t\le t_{n+1}=(\tau+n+1)\beta$. Therefore
\[
  \liminf_{t\to\infty}\frac{\rad(I_t)}{t}
  \ \ge\ \lim_{n\to\infty}\frac{(n-\olap-|i_{b_1}|)L}{(\tau+n+1)\beta}\ =\ \frac L\beta . \qedhere
\]
\end{proof}

\Cref{fig:propagation} illustrates \Cref{lem:propagate} on the surface of \Cref{fig:surface}.

\begin{figure}[!h]
\centering
\includegraphics[width=0.85\textwidth]{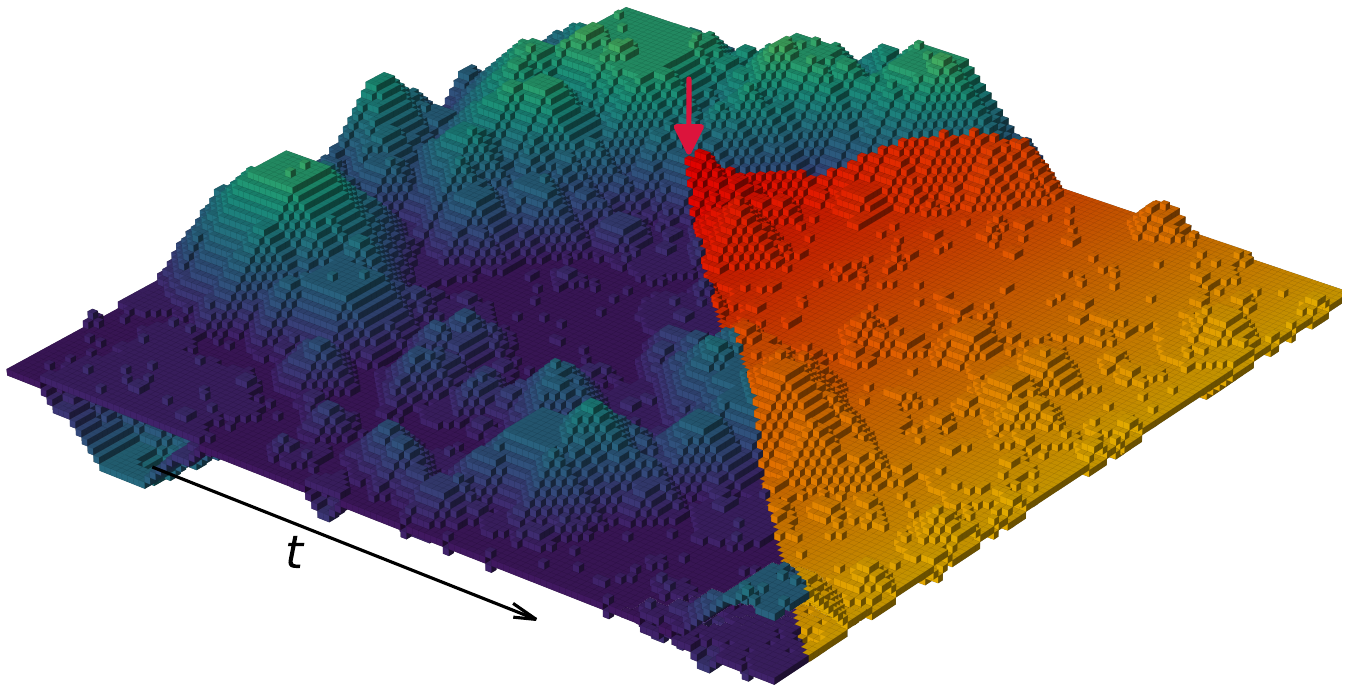}
\caption{The surface of \Cref{fig:surface}, in which one base coordinate is the time
coordinate $t$. The infection enters the surface at the marked cell. Iterating (P\ref{P1})
along the surface infects, $n$ blocks later, surface cells whose base has spatial displacement
at most $n$ from the entry column; these cells are coloured by the time of infection, from red
at the entry cell to orange. The proof of \Cref{lem:propagate} follows a single chain of such
cells, one per block, along which a fixed spatial coordinate increases by one per block, giving
the speed bound.}
\label{fig:propagation}
\end{figure}

\begin{lemma}[Seeding]
\label{lem:seed}
Let $d\ge1$, let $E$ carry the infection in the sense of \Cref{def:carries}, and assume the
hypotheses of \Cref{thm:GS}, so that a two-sided Lipschitz surface $F$ on which $E$ holds exists
almost surely. Let $D$ be the distance at which $F$ surrounds the origin, which is almost surely
finite by \Cref{thm:GSsurround}, and choose $D_0$ with $\Prob[D\le D_0]\ge\tfrac34$. Put
\[
  \tau_0=D_0+1,\qquad
  \mathcal B=\big\{(i,\tau):\ 0\le\tau\le\tau_0,\ \|i\|_\infty\le2\tau_0\big\} .
\]
If $\Prob[E(i,\tau)^c]\le(4|\mathcal B|)^{-1}$ for every cell, then with probability at least
$\tfrac12$ the following holds: $E$ holds on every cell of $\mathcal B$, $F$ surrounds the origin
at distance $D_0$, and consequently, if some particle lying in the central cube $C_0$ is infected
at time $0$, then there is a cell of $F$ whose central cube contains an infected particle at the
initial time of that cell's block.
\end{lemma}

\begin{proof}
By a union bound, the event $\mathcal E$ that $E$ holds on every cell of $\mathcal B$ and that
$D\le D_0$ has probability at least $1-\tfrac14-\tfrac14=\tfrac12$.

We work on $\mathcal E$ and suppose that a particle of $C_0$ is infected at time $0$. By (P\ref{P1})
applied to the cell $(0,0)$, which lies in $\mathcal B$, the central cube of every $(i',1)$ with
$\|i'\|_\infty\le2$ contains an infected particle at time $\beta$. We iterate this step. Every
cell $(i,\tau)$ with $\tau\le\tau_0$ and $\|i\|_\infty\le2\tau$ lies in $\mathcal B$ and so
satisfies $E$. Hence, for every $\tau\le\tau_0$, at time $\tau\beta$ the central cube of every
$(i,\tau)$ with $\|i\|_\infty\le2\tau$ contains an infected particle. Each of these times is the
initial time of a block.

It remains to find a cell of $F$ of this form. Let $b_0$ be the base index of the cell $(0,0)$,
and let $b_\tau$ be $b_0$ with its time coordinate raised to $\tau$, so that
$\|b_\tau-b_0\|_1=\tau$. The path $(b_0,0),(b_0,1),\dots,(b_0,D_0+1)$ moves only in the height
direction and has $\ell^1$-displacement $D_0+1>D_0$, so on $\mathcal E$ it intersects $F$. The
only cells of $F$ over the base $b_0$ are $(b_0,F_+(b_0))$ and $(b_0,F_-(b_0))$, with
$F_-(b_0)\le0\le F_+(b_0)$. Hence either $F_+(b_0)\le D_0+1$, or $F_-(b_0)=0$. In the second case
$(b_0,0)\in F$, so the cell $(0,0)$ itself belongs to $F$ and there is nothing to prove. We
therefore assume $F_+(b_0)\le D_0+1$. Since $F_+$ is Lipschitz,
$F_+(b_{\tau_0})\le F_+(b_0)+\tau_0\le D_0+1+\tau_0$. The base-height cell
$(b_{\tau_0},F_+(b_{\tau_0}))$ belongs to $F$. In space-time coordinates it is a cell
$(i,\tau_0)$ whose height coordinate is $F_+(b_{\tau_0})$ and whose other spatial coordinates
vanish, so
\[
  \|i\|_\infty\ \le\ D_0+1+\tau_0\ \le\ 2\tau_0 ,
\]
where the last inequality uses $\tau_0=D_0+1$. Hence this cell of $F$ has an infected particle in
its central cube at time $\tau_0\beta$, which is the initial time of its block.
\end{proof}

\begin{remark}
\label{rem:seedaccounting}
The parameters must be fixed in the following order. First we fix $\varepsilon,\olap,w$ and
$\blk$, and hence $\alpha_0$. Then we fix a scale $L_1$. The exponent in \Cref{thm:GSsurround}
contains $\eta Ls$, so the tail of $D$ decreases in $L$ and a single $D_0$ serves for every
$L\ge L_1$. This fixes $\tau_0$ and $|\mathcal B|=(\tau_0+1)(4\tau_0+1)^d$ independently of $L$.
Only then do we take $L$ large, using $\Prob[E^c]\to0$. If $D_0$ were allowed to depend on $L$,
then so would $|\mathcal B|$, and the requirement $\Prob[E^c]\le(4|\mathcal B|)^{-1}$ could not
be met. This is why we quote \Cref{thm:GSsurround} in quantitative form rather than as almost
sure finiteness.

Note that two different quantities are involved. \Cref{thm:GS} requires $\nu_E$, at the reduced
intensity $(1-\varepsilon)\eta$ and under the conditioned law of \Cref{def:restricted}, whereas
\Cref{lem:seed} requires $\Prob[E^c]$ for the process itself. Both inputs of \Cref{lem:seed} are
evaluated under the Palm measure at the origin, and both transfer from the stationary law. By
Slivnyak's theorem the Palm law is the law of $\mathcal P$ with one further particle added, and
both $E$ and $\{D\le D_0\}$ are increasing. Indeed, the surface is built from the field of good
cells, which by \Cref{prop:mono}(ii) only grows when particles are added, and the construction of
\Cref{thm:GS} produces a surface at least as close to the origin when more cells are good.
\end{remark}
\section{Upper bounds on the critical density}
\label{sec:survupper}

In this section we consider the instantaneous model $(\lambda=\infty,\iso)$ and prove parts (ii) and
(iii) of \Cref{thm:density}, that is, the two upper bounds on the critical density.

\subsection{Above the Boolean threshold}
\label{sec:supercrit}

\begin{proposition}
\label{prop:supercrit}
Let $d\ge2$ and $\mu\in(0,\infty]$. For every $\eta>\etaB$ the instantaneous infection started
from a single particle at the origin survives with positive probability. Consequently
$\eta_c(\infty,\mu)\le\etaB$ for every $\mu\in(0,\infty]$.
\end{proposition}

\begin{proof}
Since $\eta>\etaB$, the Gilbert graph $G_0$ has almost surely a unique infinite cluster
$\mathcal C_\infty$ \cite[Ch.~3]{MeesterRoy96}. Let $A$ be the event that the particle $u_0$ at
the origin belongs to $\mathcal C_\infty$, so that $\Prob(A)=\theta(\eta)>0$. Throughout the proof
we write $I_t$ for the infected set of the process started from $I_0=\{u_0\}$.

We first show that on $A$, every particle of $\mathcal C_\infty$ is infected at arbitrarily small
times. Let $v\in\mathcal C_\infty$ and let $u_0=v_0,v_1,\dots,v_k=v$ be a chain of particles of
$\mathcal C_\infty$ with consecutive ones in contact at time $0$. Almost surely no two particles
are at distance exactly $2r$ at time $0$, so each consecutive pair is at distance strictly less
than $2r$ at time $0$. By continuity of the trajectories there is $\varsigma>0$ such that all $k$
pairs are in contact throughout $[0,\varsigma]$. Note that $\varsigma$ depends on $v$ and on the
chain chosen. Given $t>0$, put $\varsigma'=\min\{t/2,\varsigma\}$ and $s_j=j\varsigma'/(k+1)$ for
$1\le j\le k+1$. Then $v_{j-1}$ and $v_j$ are in contact at $s_j$ for $1\le j\le k$, which is
\Cref{def:openpath}(i) at $\lambda=\infty$. Clause (ii) holds as well. Indeed, a recovery mark is
effective only at an instant of isolation, and $v_j$ is in contact with $v_{j+1}$ throughout
$[0,\varsigma]$, hence is never isolated there. Taking $t_0=0$ and $t_j=s_j$ for $1\le j\le k+1$
gives an open path of length $k$ from $(u_0,0)$ to $(v,s_{k+1})$, so $v\in I_{s_{k+1}}$ with
$s_{k+1}=\varsigma'\le t/2<t$. Since $t>0$ was arbitrary, we obtain that for every
$v\in\mathcal C_\infty$ and every $t>0$ there is $s<t$ with $v\in I_s$.

Next we find a persistent particle. Suppose first $\mu<\infty$ and, for $t>0$, let
\[
  \mathcal S_t=\{v\in\mathcal C_\infty:\ \mathcal R_v\cap[0,t]=\emptyset\} .
\]
This is a translation-covariant factor of $\widetilde{\mathcal P}$ of intensity
$\theta(\eta)\eta e^{-\mu t}>0$, hence almost surely non-empty by \Cref{prop:ergodic}. Note that
that proposition is a statement about the stationary law, whereas we work under the Palm measure
at the origin. However, the event $\{\mathcal S_t\ne\emptyset\}$ is increasing in $\mathcal P$.
Indeed, membership of an infinite Gilbert cluster is preserved by adding particles, and the absence
of recovery marks is a property of the particle itself. By Slivnyak's theorem the Palm law is that
of $\mathcal P$ with one further particle, so the event retains probability one. Let
$v\in\mathcal S_t$. By the first part of the proof there is an open path from $(u_0,0)$ to $(v,s)$
with $s<t$. Since $v$ carries no recovery mark at all in $[0,t]$, clause (ii) allows this path to
be prolonged to $(v,t)$. Hence $v\in I_t$ on $A$, and in particular $I_t\ne\emptyset$.

For $\mu=\infty$ we replace $\mathcal S_t$ by the set of $v\in\mathcal C_\infty$ having a particle
within distance $2r$ throughout $[0,t]$. Such a $v$ is never isolated on $[0,t]$, hence carries no
effective recovery mark there, and the same prolongation applies. The intensity of this set is
positive by the computation in the proof of \Cref{prop:soft}, applied with $\mathcal C_\infty$ in
the role of that proposition's initial set. The hypothesis of that proposition is met because
$\mathcal C_\infty$ is increasing: adding particles to $\mathcal P$ can only enlarge the
Gilbert components, so a particle in an infinite component remains in one. Note that the infected
set of the present proposition is still the one started from $I_0=\{u_0\}$.

The two parts of the proof give $\Prob\big(A\cap\{I_t\ne\emptyset\}\big)=\Prob(A)$ for every
fixed $t>0$. Intersecting over rational $t$ and using \Cref{lem:truncate}, we obtain that
$I_t\ne\emptyset$ for every $t\ge0$ almost surely on $A$. Since $\Prob(A)=\theta(\eta)>0$, the
infection survives with positive probability.
\end{proof}

The proof above requires both features of the instantaneous model. This is why
\Cref{prop:supercrit} has no counterpart at finite $\lambda$ and why the bound it gives is a
statement about $\eta_c(\infty,\cdot)$ alone; see \Cref{sec:soft}. The same two features make the
bound hold uniformly in $\mu$, up to and including $\mu=\infty$, and they make it soft. Above
$\etaB$ the proof produces an infinite initially infected cluster, not a propagating epidemic. The non-soft counterpart is \Cref{thm:motion}, which lies
strictly below $\etaB$ and is proved by a different route.

\subsection{Finiteness of the threshold in every dimension}
\label{sec:crowding}

For $d=1$ we have $\etaB=\infty$, so \Cref{prop:supercrit} gives no information. In order to
obtain finiteness of $\eta_c$ we use instead the multi-scale Lipschitz percolation of
\Cref{sec:lipschitz}. The event to which we apply it is purely geometric, and the following
proposition is correspondingly short. At infinite infection rate the spread within a cell is a
deterministic consequence of the geometry, so the recovery marks never enter the event.

\begin{proposition}
\label{prop:crowding}
Let $d\ge1$ and $\mu\in(0,\infty]$, the model being the instantaneous one of this section. Then
$\eta_c(\infty,\mu)<\infty$.
\end{proposition}

\begin{proof}
Put $m=\lceil L\sqrt d/r\rceil$ and $\ell=L/m$, so that $\ell\le r/\sqrt d$. The tiling of
$\R^d$ by the half-open \emph{small cubes} $\prod_j[k_j\ell,(k_j+1)\ell)$, $k\in\Z^d$, refines the
tessellation of \Cref{sec:tessellation}. More precisely, the $m^d$ small cubes contained in a central
cube $C_i$ cover all of it but its upper faces, and likewise the $((2\olap+1)m)^d$ small cubes
contained in a super-cube $\Qst(i)$ cover all of it but its upper faces. The discarded faces are
Lebesgue-null and play no role below. Since a small cube has diameter
$\ell\sqrt d\le r$, any two points lying in the same small cube or in two adjacent ones are within
distance $2r$ of each other, hence in contact. Fix the tessellation of \Cref{sec:tessellation} with
parameters $L$, $\beta=\blk L^2$ and $\olap\ge\max\{d,2\}$, fix $\varepsilon=\tfrac12$ and
$\blk<\mathsf c_0$, take $w$ satisfying both the lower bound of \Cref{thm:GS} and
\eqref{eq:wlarge}, and let $\alpha_0$ be the constant given by \Cref{thm:GS} for these
parameters. Let $\widehat E(i,\tau)$ be the event that every small cube contained in the
super-cube $\Qst(i)$ contains at least one particle which remains inside that cube throughout
$[\tau\beta,(\tau+1)\beta]$.

The event $\widehat E(i,\tau)$ is increasing in $\mathcal P$, since it is a finite conjunction,
over a deterministic family of cubes, of statements asserting the existence of a particle with a
prescribed property. It is also restricted to the super-cell $\scell{i}{\tau}$ in the sense of
\Cref{def:restricted}, since by the choice of the cubes every witness lies in $\Qst(i)$ at the
initial time of the block. Note that both statements need the cubes to be contained in $\Qst(i)$
rather than merely to meet it. A cube meeting $\Qst(i)$ may hold its witness outside the
super-cube, and then the event is not restricted to the super-cell at all.

The condition \eqref{eq:GScond} is imposed on $\nu_{\widehat E}$, and we estimate that quantity.
A particle confined to a
small cube has displacement inside $Q_{wL}$, since $\ell\le r/\sqrt d\le wL/2$, so by the second
observation of \Cref{sec:GSthm} the conditioning only increases the probability that a given
particle is so confined. Hence, by
thinning, the confined starters of a given small cube are
Poisson of mean at least $(1-\varepsilon)\eta\,\ell^dq_\beta$, where $q_\beta>0$ is the probability
that a Brownian motion started uniformly in a cube of side $\ell$ remains in it up to time $\beta$.
Since there are $((2\olap+1)m)^d$ small cubes in $\Qst(i)$, we obtain
\[
  1-\nu_{\widehat E}\big((1-\varepsilon)\eta,\,Q_{(2\olap+1)L},\,Q_{wL},\,\beta\big)
  \ \le\ \big((2\olap+1)m\big)^d\exp\big\{-(1-\varepsilon)\eta\,\ell^dq_\beta\big\}
  \ \xrightarrow[\eta\to\infty]{}\ 0 ,
\]
with $L,\olap,\blk,w,\varepsilon$ held fixed. So $\log\frac1{1-\nu_{\widehat E}}\to\infty$, and we
may choose $\eta$ so large that both conditions of \eqref{eq:GScond} hold. Note that the left-hand
condition $\varepsilon^2\eta L^d\ge\alpha_0$ is also improved by increasing $\eta$.

On $\widehat E(i,\tau)$, each cube carries a confined particle, and any two confined particles of
the same or of adjacent cubes are in contact throughout the time interval. Hence the confined
particles of the super-cube form a connected subgraph of $G_t$ for every $t$ in that interval, and
none of them is ever isolated during it. So none of their recovery marks is effective, and this
holds for every $\mu\in(0,\infty]$. Consequently, if at the beginning of the interval some confined
particle of the super-cube is infected, then by \Cref{def:openpath} every confined particle of the
super-cube is infected at every later time of the interval. Indeed, the chain of contacts gives
clause (i) at arbitrary times, and the absence of effective marks gives clause (ii). This spread is
deterministic on $\widehat E$, so it involves no further event and in particular imposes no
further monotonicity requirement. It remains to pass from an arbitrary infected
particle to a confined one. This must be done at every cell, not only at the origin, since the
particles provided by (P\ref{P1}) at time $(\tau+1)\beta$ were confined during the previous
block and need not be confined during the next. Let then $u$ be any particle of $C_i$ infected at
time $\tau\beta$, and let $\mathfrak c$ be the small cube containing $X_u(\tau\beta)$. On
$\widehat E(i,\tau)$ that cube carries a particle $v$ confined to it throughout the block. If
$v=u$ then $u$ is itself confined and the conclusion below applies to it directly, by the
length-zero path from $(u,\tau\beta)$. We therefore assume $v\ne u$. Then $u$ and
$v$ lie in a common small cube at time $\tau\beta$, hence are in contact then, and by continuity
they remain in contact throughout $[\tau\beta,\tau\beta+t_1]$ for some $t_1>0$. So $u$ is not
isolated on $(\tau\beta,\tau\beta+t_1]$, no mark of $u$ is effective there, and the path of length
one from $(u,\tau\beta)$ to $(v,\tau\beta+t_1)$ is open. Moreover $u$ itself is infected throughout
$[\tau\beta,\tau\beta+t_1]$ and lies in $\Qst(i)$ there. Indeed, being in contact with $v$ it stays
within $2r$ of $\mathfrak c$, which meets $C_i$, hence within $3r$ of $C_i$, and
$\dist(C_i,\partial\Qst(i))=\olap L\ge3r$ for $L$ large. Thus a confined particle of $\Qst(i)$ is
infected, and by the previous paragraph all of them are infected at every later time of the block.
In the language of \Cref{def:carries}, $\widehat E$ therefore carries the infection. Property
(P\ref{P1}) holds because $C_{i'}\subseteq\Qst(i)$ whenever $\|i'-i\|_\infty\le\olap$, in particular
whenever $\|i'-i\|_\infty\le2$, and $C_{i'}$ is a union of small cubes, each of which carries a
confined particle, hence one infected and still in $C_{i'}$ at time $(\tau+1)\beta$. Property
(P\ref{P2}) holds because $u$ is infected and in $\Qst(i)$ on $[\tau\beta,\tau\beta+t_1]$, while the
confined particles, carrying no effective recovery mark during the block, are infected and in
$\Qst(i)$ from $\tau\beta+t_1$ on.

By \Cref{thm:GS} there is almost surely a two-sided Lipschitz surface $F$ on which $\widehat E$
holds. Fix $D_0$ as in \Cref{lem:seed} and enlarge $\eta$ once more so that
$\Prob[\widehat E^c]\le(4|\mathcal B|)^{-1}$. This is possible because $|\mathcal B|$ depends only
on $D_0$ and $d$, and because $D_0$ may be chosen uniformly in $\eta$, the tail in
\Cref{thm:GSsurround} being decreasing in $\eta$. \Cref{lem:seed} then gives, with probability at
least $\tfrac12$, an infected particle in the central cube of a cell of $F$ at the initial time
of its block, and \Cref{lem:propagate} propagates the infection for ever. Hence the infection
survives with positive probability and $\eta_c(\infty,\mu)<\infty$.
\end{proof}

For $d\ge2$ the conclusion of \Cref{prop:crowding} is superseded by \Cref{prop:supercrit}, which
gives the better bound $\etaB$ and uses nothing beyond continuum percolation.
\Cref{prop:crowding} is needed only for $d=1$. This is why the validity of \Cref{thm:GS} in that
dimension, recorded in \Cref{rem:GSd1}, is used here.
\section{Survival and positive speed at small recovery rate}
\label{sec:smallmu}

In this section we prove \Cref{thm:smallmu}. The model is the finite-rate one, that is,
$\lambda\in(0,\infty)$ and $\mathfrak r=\unc$, and $d\ge2$. The construction follows
\cite[Section 5]{GracarStauffer19a} and \cite[Section 3]{BaldassoStauffer23}, with two changes
required by the continuum. First, a contact is an event of positive duration, so transmission at
rate $\lambda$ needs a proximity window. Second, the seed has to be given a trajectory before it
is known which particle the seed is.

\Cref{thm:GS} requires a cell event which is a function of the configuration and increasing in it.
However, the particle carrying the infection into a cell is determined by the previous cell's
randomness, so the event cannot nominate it, and quantifying over all particles of the central
cube would destroy monotonicity. Following \cite{BaldassoStauffer23} we therefore attach to each
cell an independent auxiliary path and let a designated infected particle follow it. In this way
the cell event quantifies over points of space rather than over particles. We present this in
\Cref{sec:auxpath}.

\subsection{Parameters}
\label{sec:smallmuparams}

We use the tessellation of \Cref{sec:tessellation} with parameters $L$, $\beta=\blk L^2$ and
$\olap\ge\max\{d,3\}$. We take the value $3$ rather than $2$ because $(F_{\ref{F3}})$ below provides
witnesses in neighbours at $\ell^\infty$-distance $\olap-1$, and this range must cover the
distance $2$ that \Cref{lem:propagate} uses. The sprinkling parameter of \eqref{eq:GScond} is $\varepsilon=\tfrac12$
throughout this section. Fix a \emph{proximity window} $\mathfrak d=\mathfrak d(r,d)>0$ small
enough that, for a standard Brownian motion $B$,
\begin{equation}
\label{eq:p0}
  q_0:=\Prob\Big[\sup_{s\le\mathfrak d}|B_s|\le r/4\Big]\ \ge\ \tfrac34 ,
\end{equation}
which is possible by (H\ref{H3}). This is the continuum substitute for the requirement, on the
lattice, that two particles share a site. Put $p_{\mathrm{inf}}=1-e^{-\lambda\mathfrak d}>0$.
Set
\begin{equation}
\label{eq:smallmuparams}
  T_L=L^{5/3},\qquad W_L=L^{4/3},\qquad \mathcal J=\{1,\dots,\lfloor L^{1/3}\rfloor-1\},\qquad
  t_j=\tau\beta+jW_L ,
\end{equation}
so that $t_{|\mathcal J|}+\mathfrak d-\tau\beta\le T_L\le\beta/2$ for $L$ large. Note that the
final sampling window must close before $\tau\beta+T_L$, since \Cref{def:qualify} reads the
configuration only up to that time. Dropping the last index secures this when $L^{1/3}$ is an
integer.
Finally fix a grid
\begin{equation}
\label{eq:grid}
  \mathcal Z_i\subseteq C_i\quad\text{of mesh } \tfrac{r}{8\sqrt d},
  \qquad |\mathcal Z_i|\le C_d\big(L/r\big)^d ,
\end{equation}
so that every point of $C_i$ lies within distance $r/8$ of a point of $\mathcal Z_i$.

\subsection{Auxiliary paths and the modified construction}
\label{sec:auxpath}

To each cell $(i,\tau)$ we attach, independently of everything else and of the other cells,
\begin{itemize}
\item a Brownian path $\gamma^{(i,\tau)}=(\gamma^{(i,\tau)}_s)_{s\in[0,\beta]}$ with
  $\gamma^{(i,\tau)}_0=0$;
\item a Poisson process $\mathcal R^{(i,\tau)}$ of rate $\mu$ on $[0,\beta]$; and
\item a family $(\mathcal A^{(i,\tau)}_{\to v})_{v\in\mathcal P}$ of independent Poisson processes
  of rate $\lambda$ on $[0,\beta]$.
\end{itemize}
We modify the construction of \Cref{sec:graphical} as follows. At time $\tau\beta$, if the central
cube $C_i$ contains an infected particle, we designate one of them, say the one whose position at
$\tau\beta$ is smallest in the lexicographic order. This rule is measurable with respect to the
state at time $\tau\beta$. We let the designated particle follow the trajectory
\[
  X(\tau\beta)+\gamma^{(i,\tau)}_{\,\cdot\,-\tau\beta}
  \qquad\text{on }[\tau\beta,(\tau+1)\beta] ,
\]
carrying $\mathcal R^{(i,\tau)}$ as its recovery marks and $(\mathcal A^{(i,\tau)}_{\to v})_v$ as
its outgoing transmission marks during that interval. If $C_i$ contains no infected particle at
$\tau\beta$, the path $\gamma^{(i,\tau)}$ is simply not used. No particle is ever re-assigned in
the middle of a block. Every substitution is decided at a block-initial time, and a particle is
substituted at most once per block, since the central cubes are disjoint. In \Cref{lem:rerand} we
show that this re-randomisation preserves the law.

In \cite[Section 3]{BaldassoStauffer23} the auxiliary path is assigned instead on lateral entry,
to the first infected particle to enter the cell from a neighbour. Here we must not do this. An
offspring $v\in\Upsilon_z$, in the notation of \Cref{def:qualify} below, travelling towards a
neighbouring central cube would be liable to such a
reassignment, and \Cref{lem:spread}, which estimates its displacement by the killed heat kernel of
its own motion, would not apply. The route to the surface used here makes the lateral branch
unnecessary, since every entry into a cell of $F$ occurs at a block-initial time by
\Cref{lem:seed,lem:propagate}.

\begin{lemma}
\label{lem:rerand}
The modified construction has the same law as the original one.
\end{lemma}

\begin{proof}
The designation is measurable with respect to the state at time $\tau\beta$. The objects
substituted for the designated particle on $(\tau\beta,(\tau+1)\beta]$, namely its trajectory
increment, its recovery marks and its outgoing transmission marks, are independent copies of the
objects they replace, independent of the state at time $\tau\beta$. Replacing a conditionally
independent family by an independent copy of itself leaves the joint law unchanged. The
substitutions made in distinct cells of a common time block concern distinct particles, since the
central cubes $C_i$ are disjoint, and those made in distinct blocks concern disjoint time
intervals. So the substitutions may be performed one block after another, each preserving the law
of the whole marked configuration.
\end{proof}

By \Cref{lem:rerand} the seed's trajectory and marks are external randomness attached to the
cell, and the cell event may refer to them without quantifying over particles.

\begin{remark}[The cost of the substitution]
\label{rem:designatedoverlap}
The substitution has an effect beyond the block in which it is made. What is replaced is a
trajectory increment, so a particle designated in some block $\tau'$ continues after
$(\tau'+1)\beta$ from a displaced position, and its trajectory in the process actually run differs
from its trajectory in $\omega$ at every later time. We resolve this in two separate points.

For the probability of the cell event nothing needs counting. \Cref{thm:GS} requires only that
the local event be increasing, be restricted to the super-cell, and have marginal probability
satisfying \eqref{eq:GScond} at the sprinkled intensity. It requires no independence, nor even any
quantitative decorrelation, between the events of different cells, since the correlations are
absorbed entirely by the loss of intensity from $\eta$ to $(1-\varepsilon)\eta$. By
\Cref{def:qualify} below, the set $\Upsilon_z$ is a
functional of $\omega$ and of the auxiliary objects of $(i,\tau)$ alone. This gives
monotonicity (\Cref{lem:accincreasing} below) and locality, and the marginal law is then given by
\Cref{lem:rerand}.

For the pathwise conclusion of \Cref{lem:acccarries} below, which needs a witness that is in
$C_{i'}$ in the process actually run and not merely in $\omega$, the count is finite and explicit.
One particle is designated per central cube at each block-initial time, and $\Qst(i)$ contains
$N=(2\olap+1)^d$ central cubes, so at most $N$ members of $\Upsilon_z$ have a run trajectory
differing from their $\omega$ trajectory during the block. Demanding $N+1$ witnesses in $(F_{\ref{F3}})$
therefore leaves at least one uncompromised, which is all that lemma uses.
\end{remark}

Note that the acceptable-cell event of \Cref{sec:cellevent} below is not a function of the
configuration alone. It also reads
the auxiliary objects $\gamma^{(i,\tau)}$, $\mathcal R^{(i,\tau)}$ and
$(\mathcal A^{(i,\tau)}_{\to v})_v$ of its own cell, which \Cref{def:restricted} does not mention.
Nothing need be added to \Cref{thm:GS} on that account. Every mark in the construction is Poisson,
and the auxiliary paths and mark families are attached to the cells independently of $\mathcal P$
and of each other. So the whole marked configuration is again a Poisson process on an enlarged
mark space, as in \Cref{prop:ergodic}, and the theorem may be read on that process, with distinct
cells carrying disjoint auxiliary objects. In particular the particles added by the sprinkling
arrive carrying their own marks, which the indexing of
$(\mathcal A^{(i,\tau)}_{\to v})_{v\in\mathcal P}$ by the random particle set provides.

\subsection{Qualifying proximity relative to a path}
\label{sec:proximity}

Fix a cell $(i,\tau)$, write $\gamma=\gamma^{(i,\tau)}$, and let
\begin{equation}
\label{eq:Jprime}
  \mathcal J'=\Big\{j\in\mathcal J:\ \sup_{s\in[t_j,t_j+\mathfrak d]}
  \big|\gamma_{s-\tau\beta}-\gamma_{t_j-\tau\beta}\big|\le r/4\Big\} .
\end{equation}
This is a subset of $\mathcal J$ determined by $\gamma$ alone. Since $W_L>\mathfrak d$ for $L$
large, the windows $[t_j,t_j+\mathfrak d]$ are disjoint, so the events $\{j\in\mathcal J'\}$ are
independent, each of probability $q_0\ge\frac34$.

\begin{definition}[Qualifying proximity to a reference path]
\label{def:qualify}
Let $z\in\mathcal Z_i$ and write $\pi_z(s)=z+\gamma_{s-\tau\beta}$ for the reference path issued
from $z$. Throughout this definition $X_v$ and $\mathcal R_v$ denote the trajectory and the
recovery marks of $v$ in the configuration $\omega$ of \Cref{sec:graphical}, before the
substitution of \Cref{sec:auxpath}. Only $\pi_z$ and $\mathcal A^{(i,\tau)}_{\to v}$ are auxiliary
objects of the cell. The set $\Upsilon_z$ defined below is thus a functional of $\omega$ and of
the auxiliary objects of $(i,\tau)$ alone, and the two are independent. A particle $v$
\emph{qualifies at $t_j$}, $j\in\mathcal J'$, if
\begin{enumerate}[(i)]
\item $|\pi_z(t_j)-X_v(t_j)|\le r/2$;
\item $v$ displaces by at most $r/4$ during $[t_j,t_j+\mathfrak d]$; and
\item $\mathcal A^{(i,\tau)}_{\to v}$ has a mark in $[t_j,t_j+\mathfrak d]$.
\end{enumerate}
It \emph{qualifies freshly at $t_j$} if in addition $|\pi_z(t_k)-X_v(t_k)|>r/2$ for every
$k\in\mathcal J'$ with $k>j$. We let $\Upsilon_z$ be the set of particles which qualify freshly at
some $j\in\mathcal J'$, carry no recovery mark during $[\tau\beta,\tau\beta+T_L]$, and remain in
the reduced super-cube $\Qst_-(i)$ throughout $[\tau\beta,\tau\beta+T_L]$.
\end{definition}

Freshness is the device used inside the proof of \cite[Lemma 2]{GracarStauffer19a}. Its effect
here is that the events at distinct sampling times are disjoint, so that $\Upsilon_z$ is a
thinning of $\mathcal P$ and $|\Upsilon_z|$ is exactly Poisson. Counting instead all particles
ever meeting the path would give only a stochastic lower bound, since collisions away from the
sampling grid would not be counted. Note that all three clauses, and both further requirements,
refer only to the window $[\tau\beta,\tau\beta+T_L]$. So the further thinning of
\Cref{lem:spread} below, which concerns $[\tau\beta+T_L,(\tau+1)\beta]$, is conditionally
independent of $\Upsilon_z$.

\begin{lemma}[Transmission to qualifying particles]
\label{lem:slack}
Let $x\in C_i$ and let $z\in\mathcal Z_i$ be within $r/8$ of $x$. If $v$ qualifies at $t_j$ with
$j\in\mathcal J'$ relative to $\pi_z$, then a particle following $x+\gamma$ is in contact with $v$
throughout $[t_j,t_j+\mathfrak d]$, and therefore transmits to $v$ at the mark of clause (iii).
\end{lemma}

\begin{proof}
For $s\in[t_j,t_j+\mathfrak d]$,
\begin{align*}
  \big|X_v(s)-(x+\gamma_{s-\tau\beta})\big|
  &\ \le\ \underbrace{|X_v(s)-X_v(t_j)|}_{\le\,r/4}
  +\underbrace{|X_v(t_j)-\pi_z(t_j)|}_{\le\,r/2}
  +\underbrace{|\pi_z(t_j)-\pi_z(s)|}_{\le\,r/4}
  +\underbrace{|z-x|}_{\le\,r/8}\\
  &\ \le\ \tfrac98 r\ \le\ 2r ,
\end{align*}
by clauses (ii) and (i), by $j\in\mathcal J'$, and by the choice of $z$.
\end{proof}

\begin{lemma}[Offspring count]
\label{lem:proximity}
There is $C_{\mathrm p}=C_{\mathrm p}(r,d,\lambda)>0$ such that the following holds for all large
$L$, under the free and under the conditioned law of \eqref{eq:GScond}. Condition on $\gamma$ and
on $\mathcal R^{(i,\tau)}$, on the event that
$|\mathcal J'|\ge\frac12|\mathcal J|$ and that $\sup_{s\le T_L}|\gamma_s|\le\tfrac12\olap L$, which
is the condition $(F_{\ref{F1}})$ of \Cref{sec:cellevent} below. Then, for each $z\in\mathcal Z_i$,
the set $\Upsilon_z$ is Poisson and
\[
  \E\big|\Upsilon_z\big|\ \ge\ C_{\mathrm p}\,\eta\,e^{-\mu T_L}\,L^{1/3} ,
\]
with $C_{\mathrm p}=\Theta(p_{\mathrm{inf}})=\Theta(\lambda)$ as $\lambda\to0$.
\end{lemma}

\begin{proof}
We condition in addition on the state at time $\tau\beta$. Whether $v\in\Upsilon_z$ is then a
measurable function of $v$'s increment on $[\tau\beta,\tau\beta+T_L]$, of its recovery marks there
and of the process $\mathcal A^{(i,\tau)}_{\to v}$, all read, by \Cref{def:qualify}, in the
configuration $\omega$ before any substitution. These are independent across $v$ and independent of
the point process of positions at time $\tau\beta$, which is Poisson of intensity $\eta$. So
$\Upsilon_z$ is a thinning of that process and hence Poisson.

Fix $j\in\mathcal J'$. By stationarity the particles at time $t_j$ form a Poisson process of
intensity $\eta$. The conditioning has made the point $\pi_z(t_j)$ deterministic, so the number of
particles within distance $r/2$ of it is Poisson of mean $\eta\omega_d(r/2)^d$. For such a particle
$v$, let $Q_j$ be the event that clauses (ii) and (iii) hold, that $v$ carries no recovery mark on
$[\tau\beta,\tau\beta+T_L]$ and that $v$ stays in $\Qst_-(i)$ there. The recovery marks and the
transmission process are independent of the trajectory, so those two factors split off. Clause
(ii) and the confinement are both properties of the trajectory, and we combine them by
$\Prob[A\cap B]\ge\Prob[A]-\Prob[B^c]$. On $(F_{\ref{F1}})$ the point $\pi_z(t_j)$ lies at distance at
least $(\tfrac{\olap}2-1)L\ge L/2$ from $\partial\Qst_-(i)$, since $z\in C_i$ is at distance
$(\olap-1)L$ from it and $|\gamma|\le\tfrac12\olap L$. As $v$ is within $r/2$ of $\pi_z(t_j)$,
(H\ref{H3}) with radius $L/2-r$ and time $T_L=L^{5/3}$ bounds the probability that $v$ leaves
$\Qst_-(i)$ by $C\exp\{-cL^{1/3}\}$. Hence, by \eqref{eq:p0} and the definition of
$p_{\mathrm{inf}}$,
\[
  \Prob[Q_j]\ \ge\ p_{\mathrm{inf}}\,e^{-\mu T_L}\big(q_0-Ce^{-cL^{1/3}}\big)
  \ \ge\ \tfrac12\,p_{\mathrm{inf}}\,e^{-\mu T_L}
\]
for $L$ large, using $q_0\ge\frac34$. Conditionally on $Q_j$, the probability that $v$ fails to
qualify freshly is bounded, by the Markov property at $t_j+\mathfrak d$ and (H\ref{H1}), by
\[
  \sum_{k>j}C_d\big(W_L/2\big)^{-d/2}(k-j)^{-d/2}\omega_d(r/2)^d\ \le\ C\,L^{1/3-2d/3}
  \ \le\ \tfrac12
\]
for $L$ large, in every dimension $d\ge1$. Note that conditioning avoids subtracting one bound
from another. Hence the expected number of particles qualifying freshly at $t_j$ and lying in
$\Upsilon_z$ is at least $\frac14\eta\omega_d(r/2)^dp_{\mathrm{inf}}e^{-\mu T_L}$. The events at
distinct $j$ are disjoint, and $|\mathcal J'|\ge\frac12|\mathcal J|\ge\frac13L^{1/3}$ for $L$
large, so summing gives the claim with $C_{\mathrm p}=\frac1{12}\omega_d(r/2)^dp_{\mathrm{inf}}$.

Under the conditioned law, the event that the particle started at $y$ lies in $\Upsilon_z$ is
determined by its trajectory and marks on $[\tau\beta,\tau\beta+T_L]$ and keeps it in
$\Qst_-(i)$, at distance at least $L$ from $\partial\Qst(i)$, so by the third observation of
\Cref{sec:GSthm} its probability decreases by at most a constant factor. The claim follows
after decreasing $C_{\mathrm p}$ by that factor.
\end{proof}

\begin{lemma}[Survival and spread over the rest of the block]
\label{lem:spread}
Fix the ratio $\blk=\beta/L^2$. There is $c_p=c_p(\olap,\blk,d)>0$ such that the following holds
for all large $L$. Let $\mathcal F$ be generated by the auxiliary objects of the cell together
with the marked configuration on $[\tau\beta,\tau\beta+T_L]$, so that each $\Upsilon_z$ is
$\mathcal F$-measurable and its particles lie in the reduced super-cube $\Qst_-(i)$ at time
$\tau\beta+T_L$. Then,
conditionally on $\mathcal F$, the events
\begin{align*}
  \mathcal S_v&=\big\{v\text{ carries no recovery mark in }(\tau\beta+T_L,(\tau+1)\beta]\big\},\\
  \mathcal T_v^{i'}&=\big\{v\text{ stays in }\Qst(i)\text{ on }[\tau\beta+T_L,(\tau+1)\beta]
  \text{ and }v\in C_{i'}\text{ at time }(\tau+1)\beta\big\}
\end{align*}
are independent over $v\in\Upsilon_z$, and $\Prob[\mathcal S_v\cap\mathcal T_v^{i'}\mid\mathcal F]
\ge c_p\,e^{-\mu\beta}$ for every $v\in\Upsilon_z$ and every $i'$ with
$\|i'-i\|_\infty\le\olap-1$. Consequently, if $|\Upsilon_z|\ge n$ then for every $\nu$ with $1\le\nu\le n+1$
\[
  \Prob\big[\text{fewer than }\nu\text{ particles }v\in\Upsilon_z\text{ satisfy }
  \mathcal S_v\cap\mathcal T_v^{i'}\ \big|\ \mathcal F\big]
  \ \le\ \binom{n}{\nu-1}\big(1-c_pe^{-\mu\beta}\big)^{n-\nu+1} .
\]
\end{lemma}

\begin{proof}
Conditionally on $\mathcal F$ the trajectories of distinct particles after $\tau\beta+T_L$ are
independent Brownian motions started at their current positions, and their recovery marks after
that time are independent fresh Poisson processes. Both are independent of $\mathcal F$. This is
where we use that \Cref{def:qualify} only reads the window $[\tau\beta,\tau\beta+T_L]$.
This gives the conditional independence, and
$\Prob[\mathcal S_v\mid\mathcal F]=e^{-\mu(\beta-T_L)}\ge e^{-\mu\beta}$. By
\Cref{def:qualify} the trajectories and marks meant here are those of $\omega$, before any
substitution. The discrepancy between $\omega$ and the process actually run is confined to the at
most $N$ compromised members identified in \Cref{rem:designatedoverlap}, and is paid for by the
surplus of witnesses that $(F_{\ref{F3}})$ requires, not by an exception here.

For $\mathcal T_v^{i'}$ the particle is at some $x\in\Qst_-(i)$ at time $\tau\beta+T_L$ and the
time remaining is $\theta:=\beta-T_L\in[\beta/2,\beta]$ for $L$ large. Writing
$p^{\Qst(i)}_\theta$ for the transition density of Brownian motion killed on leaving $\Qst(i)$,
\[
  \Prob\big[\mathcal T_v^{i'}\ \big|\ \mathcal F\big]
  \ =\ \int_{C_{i'}}p^{\Qst(i)}_{\theta}(x,y)\,\mathrm dy\ \ge\ c_p .
\]
Indeed, after the scaling $x\mapsto x/L$ the domain becomes the fixed cube $[-\olap,\olap+1]^d$,
the starting point ranges over the fixed cube $[-\olap+1,\olap]^d$, the target over a fixed unit
sub-cube of $[-\olap+1,\olap]^d$, and the elapsed time over $\theta/L^2\in[\blk/2,\blk]$. Both the
starting point and the target are thus confined to a compact subset of the interior of the domain,
at distance at least $1$ from its boundary in the rescaled picture. The killed heat kernel of a
bounded Lipschitz domain is bounded below on such compacts by a positive constant depending only on
the domain, the target and the time range. In particular $c_p$ does not depend on $L$. This is
where we use the reduction from $\Qst(i)$ to $\Qst_-(i)$ and the restriction to
$\|i'-i\|_\infty\le\olap-1$. The two events concern
disjoint sources of randomness, so their conditional probabilities multiply. Conditionally on
$\mathcal F$ the number of $v\in\Upsilon_z$ satisfying $\mathcal S_v\cap\mathcal T^{i'}_v$
therefore dominates a binomial variable with $n$ trials and success probability
$c_pe^{-\mu\beta}$, and the display is the corresponding lower tail bound.
\end{proof}

\subsection{The cell event}
\label{sec:cellevent}

\begin{definition}[Acceptable cell]
\label{def:acceptable}
Put $n_0:=\lceil\tfrac18C_{\mathrm p}\eta e^{-\mu T_L}L^{1/3}\rceil$. This is an eighth of the
mean given by \Cref{lem:proximity} at intensity $\eta$, and hence a quarter of it at the sprinkled
intensity $(1-\varepsilon)\eta=\eta/2$ at which \eqref{eq:GScond} evaluates the event. The cell
$(i,\tau)$ is \emph{acceptable} if the following three conditions hold.
\begin{enumerate}[$(F_1)$]
\item\label{F1} $\mathcal R^{(i,\tau)}$ has no mark in $[0,T_L]$; the path $\gamma^{(i,\tau)}$
  satisfies $\sup_{s\le T_L}|\gamma^{(i,\tau)}_s|\le\tfrac12\olap L$, so that the reference paths
  $\pi_z$, $z\in\mathcal Z_i$, stay at distance at least $(\tfrac{\olap}2-1)L$ from
  $\partial\Qst_-(i)$; and $|\mathcal J'|\ge\frac12|\mathcal J|$.
\item\label{F2} for every $z\in\mathcal Z_i$, the set $\Upsilon_z$ satisfies
  $|\Upsilon_z|\ge n_0$.
\item\label{F3} for every $z\in\mathcal Z_i$ and every $i'$ with $\|i'-i\|_\infty\le\olap-1$,
  there are at least $N+1=(2\olap+1)^d+1$ distinct particles $v\in\Upsilon_z$ which carry no
  recovery mark in $(\tau\beta+T_L,(\tau+1)\beta]$, remain in $\Qst(i)$ throughout
  $[\tau\beta+T_L,(\tau+1)\beta]$, and lie in $C_{i'}$ at time $(\tau+1)\beta$.
\end{enumerate}
\end{definition}

\begin{lemma}
\label{lem:accincreasing}
The event that $(i,\tau)$ is acceptable is increasing in the particle configuration, and is
restricted to the super-cell $\scell{i}{\tau}$ in the sense of \Cref{def:restricted}, together with
the auxiliary randomness of the cell.
\end{lemma}

\begin{proof}
Condition $(F_{\ref{F1}})$ does not involve the particles at all. For $(F_{\ref{F2}})$ and $(F_{\ref{F3}})$ it suffices to
show that $v\in\Upsilon_z$ is a predicate of the single particle $v$. By \eqref{eq:Jprime} the set
$\mathcal J'$ and the reference path $\pi_z(s)=z+\gamma_{s-\tau\beta}$ are functions of
$\gamma^{(i,\tau)}$ alone, and the grid $\mathcal Z_i$ and the threshold $n_0$ are deterministic.
Each of clauses (i)--(iii) of \Cref{def:qualify}, the freshness clause, the absence of recovery
marks and the confinement to $\Qst_-(i)$ constrains only $v$'s own trajectory, its own recovery
marks and the process $\mathcal A^{(i,\tau)}_{\to v}$. Note that clause (iii) requires only that
this process have a mark; that the mark is active is shown separately in \Cref{lem:slack}. In
particular the freshness clause, though negative, is negative in $v$'s own position only, so
adding particles cannot expel a particle from $\Upsilon_z$. Hence
$\Upsilon_z(\omega)\subseteq\Upsilon_z(\omega')$ whenever $\omega\subseteq\omega'$ in the sense of
\Cref{prop:mono}(ii), and the same holds for the set of witnesses of $(F_{\ref{F3}})$. Both conditions are
lower bounds on the cardinality of such a set, and the index sets $\mathcal Z_i$,
$\{i':\|i'-i\|_\infty\le\olap-1\}$ and $\mathcal J'$ do not grow with the configuration, so no new
conjunct is created. The event is therefore increasing.

For locality, \Cref{def:qualify} requires every member of $\Upsilon_z$ to remain in $\Qst_-(i)$
throughout $[\tau\beta,\tau\beta+T_L]$, hence to lie in $\Qst_-(i)\subseteq\Qst(i)$ already at time
$\tau\beta$. This is the time at which \Cref{def:restricted} evaluates membership of the region.
The witnesses of $(F_{\ref{F3}})$ are members of $\Upsilon_z$, and their further requirements concern
only their own trajectories and marks on $[\tau\beta+T_L,(\tau+1)\beta]$. No clause refers to a
particle entering $\Qst(i)$ after $\tau\beta$, so the event is restricted to
$\Qst(i)\times[\tau\beta,(\tau+1)\beta]$, and therefore to the super-cell, together with
$\gamma^{(i,\tau)}$, $\mathcal R^{(i,\tau)}$ and $(\mathcal A^{(i,\tau)}_{\to v})_v$.
\end{proof}

\begin{lemma}[The carrying property of acceptable cells]
\label{lem:acccarries}
The event $E(i,\tau)=\{(i,\tau)\text{ is acceptable}\}$ carries the infection in the sense of
\Cref{def:carries}.
\end{lemma}

\begin{proof}
Suppose $(i,\tau)$ is acceptable and some particle of $C_i$ is infected at time $\tau\beta$. The
designated particle of \Cref{sec:auxpath} is then one of them. Write $x\in C_i$ for its position
at $\tau\beta$ and let $z\in\mathcal Z_i$ be within $r/8$ of $x$. By $(F_{\ref{F1}})$ it carries no
recovery mark on $[\tau\beta,\tau\beta+T_L]$ and so is infectious throughout, and it stays in
$\Qst(i)$ there. By $(F_{\ref{F2}})$ and \Cref{lem:slack} it transmits to every $v\in\Upsilon_z$, of which
there are at least $n_0$, and each of these is then infected and free of recovery marks up to
$\tau\beta+T_L$.

Fix $i'$ with $\|i'-i\|_\infty\le\olap-1$. By $(F_{\ref{F3}})$ at least $N+1$ members $v$ of
$\Upsilon_z$ carry no recovery mark on
$(\tau\beta+T_L,(\tau+1)\beta]$, remain in $\Qst(i)$ there, and lie in $C_{i'}$ at time
$(\tau+1)\beta$. All of this is read in $\omega$, and by \Cref{rem:designatedoverlap} at most $N$
of them are compromised by the substitution. We fix one of the others, for which the trajectory
just described is also its trajectory in the process actually run. Since $\olap\ge3$ the range
$\|i'-i\|_\infty\le\olap-1$ covers $\|i'-i\|_\infty\le2$, which is (P\ref{P1}) of \Cref{def:carries}.
For (P\ref{P2}), the designated particle is infected and in $\Qst(i)$ on $[\tau\beta,\tau\beta+T_L]$,
and the witness $v$ just selected is infected and in $\Qst(i)$ on
$[\tau\beta+T_L,(\tau+1)\beta]$.
\end{proof}

\begin{lemma}[Probability of an acceptable cell]
\label{lem:mainSmallMu}
Fix $\blk$ and $\olap$. There are $C,c>0$, depending only on $\olap,\blk,r,d,\lambda$, such that
for all large $L$ and all $\mu>0$ with $\mu\beta\le\log2$,
\begin{equation}
\label{eq:cellbound}
  \Prob\big[(i,\tau)\text{ is acceptable}\big]\ \ge\
  1-\mu T_L-Ce^{-cL^{1/3}}-C L^{d}\exp\big\{-c\,\eta L^{1/3}\big\} .
\end{equation}
Consequently, for every $\alpha_0>0$ there is $L_0=L_0(\alpha_0,\lambda,\eta,r,d,\blk,\olap)$ such
that the right-hand side of \eqref{eq:cellbound} exceeds $1-e^{-\alpha_0}$ whenever $L\ge L_0$ and
\begin{equation}
\label{eq:mustar}
  \mu\ \le\ \mu_*(L)\ :=\ \frac{\log 2}{\blk L^{2}} .
\end{equation}
\end{lemma}

\begin{proof}
We first bound the probability that $(F_{\ref{F1}})$ fails. The process $\mathcal R^{(i,\tau)}$ has a
mark in $[0,T_L]$ with probability $1-e^{-\mu T_L}\le\mu T_L$. By (H\ref{H3}) with radius
$\tfrac12\olap L$ and time $T_L=L^{5/3}$, the path $\gamma^{(i,\tau)}$ leaves the ball of radius
$\tfrac12\olap L$ with probability at most $C\exp\{-c\olap^2L^{1/3}\}$. The indicators
$\{j\in\mathcal J'\}$, $j\in\mathcal J$, are independent Bernoulli variables of parameter
$q_0\ge\frac34$, so by Hoeffding's inequality
$\Prob[|\mathcal J'|<\frac12|\mathcal J|]\le e^{-|\mathcal J|/8}\le e^{-cL^{1/3}}$. Note that this
is the only place where the seed is paid for, and it costs a single $\mu T_L$. There is no union
bound, because $(F_{\ref{F1}})$ concerns the auxiliary objects of the cell and not any particle.

Next we bound the probability that $(F_{\ref{F2}})$ fails. Fix $z\in\mathcal Z_i$. Recall that the bound
must hold at the sprinkled intensity $(1-\varepsilon)\eta=\eta/2$, which is where
\eqref{eq:GScond} evaluates it. On $(F_{\ref{F1}})$, \Cref{lem:proximity} applied at that intensity makes
$|\Upsilon_z|$ Poisson of mean at least $\tfrac12C_{\mathrm p}\eta e^{-\mu T_L}L^{1/3}\ge2n_0$, so
the Chernoff bound $\Prob[\Poi(m)\le m/2]\le e^{-m/8}$ gives
$\Prob[|\Upsilon_z|<n_0]\le\exp\{-\frac1{16}C_{\mathrm p}\eta e^{-\mu T_L}L^{1/3}\}
\le\exp\{-\frac1{32}C_{\mathrm p}\eta L^{1/3}\}$, using $\mu T_L\le\mu\beta\le\log2$. A union
bound over the $|\mathcal Z_i|\le C_d(L/r)^d$ grid points gives the third term of
\eqref{eq:cellbound}.

Finally we bound the probability that $(F_{\ref{F3}})$ fails given $(F_{\ref{F2}})$. \Cref{lem:spread} applied
with $\nu=N+1$ and $n=n_0$, together with a union bound over the $|\mathcal Z_i|(2\olap-1)^d$
pairs $(z,i')$, gives a failure probability at most
\[
  CL^d\binom{n_0}{N}\big(1-c_pe^{-\mu\beta}\big)^{n_0-N}
  \ \le\ CL^d\,n_0^{N}\exp\big\{-(n_0-N)c_pe^{-\mu\beta}\big\}
  \ \le\ CL^{d}\exp\big\{-\tfrac1{64}C_{\mathrm p}c_p\eta L^{1/3}\big\}
\]
for $L$ large, again using $e^{-\mu\beta}\ge\frac12$ and $e^{-\mu T_L}\ge\frac12$. Here the number
$N$ of witnesses lost to \Cref{rem:designatedoverlap} is a constant, so $n_0-N\ge n_0/2$, and the
polynomial factor $n_0^N\le CL^{N/3}$ is absorbed by the exponential.

Summing the three bounds gives \eqref{eq:cellbound}. For the last assertion, impose
\eqref{eq:mustar}. Then $\mu\beta\le\log2$, so the first part applies, and
\[
  \mu T_L\ \le\ \frac{\log2}{\blk L^2}\,L^{5/3}\ =\ \frac{\log2}{\blk}\,L^{-1/3}
  \ \longrightarrow\ 0 ,
\]
while the remaining two terms of \eqref{eq:cellbound} carry a factor $e^{-cL^{1/3}}$ against a
polynomial and also tend to $0$. Hence for $L$ large the right-hand side of \eqref{eq:cellbound}
exceeds $1-e^{-\alpha_0}$.
\end{proof}

\begin{remark}[The threshold]
\label{rem:muStar}
The binding constraint in \eqref{eq:mustar} is $\mu\beta\le\log2$, not $\mu T_L\lesssim1$. Tracing
it through the proof in \Cref{sec:goodcell} below, the third term of \eqref{eq:cellbound} requires
$C_{\mathrm p}(\lambda)\eta L_0^{1/3}\gtrsim d\log L_0$, and $C_{\mathrm p}(\lambda)=\Theta(\lambda)$,
so the smallest admissible scale and the resulting threshold are
\[
  L_0(\lambda)\ \asymp\ \Big(\frac{\log(1/\lambda)}{\lambda\eta}\Big)^{3} ,
  \qquad
  \mu_*(\lambda)\ =\ \frac{\log2}{\blk L_0(\lambda)^2}
  \ \asymp\ \frac{(\lambda\eta)^{6}}{\big(\log(1/\lambda)\big)^{6}}
  \qquad(\lambda\to0) .
\]
A lower bound on the admissible $L_0$ gives an upper bound on $\mu_*$, so this is a threshold of
that order and not merely at least that order. In particular $\mu_*(\lambda)\to0$ as
$\lambda\to0$, which determines the shape of the finite-$\lambda$ survival region drawn in
\Cref{fig:phase}. The constraint involves $\lambda$ and $\eta$ only through their product, so the
display describes equally the regime of small $\eta$ at fixed $\lambda$, once the logarithm is
read as $\log\big(1/(\lambda\eta)\big)$. It does not describe the regime of large $\eta$. The
scale $L_0$ is subject to further requirements which do not involve the intensity at all, namely
the dimensional threshold $L_1(d)$ of \Cref{thm:GSsurround} and the several ``for $L$ large''
conditions of \Cref{sec:smallmuparams}, and these bound $\mu_*=\log2/(\blk L_0^2)$ from above
uniformly in $\eta$. We make no claim that the exponent is sharp; see \Cref{sec:open}(2).
\end{remark}

\subsection{Proof of \texorpdfstring{\Cref{thm:smallmu}}{Theorem C}}
\label{sec:goodcell}

\Cref{def:acceptable} propagates the infection provided it enters a cell through the base, that is,
provided some particle of the central cube is infected at the initial time of the block. Every
entry is of that form here. Indeed, \Cref{lem:seed} gives an infected particle in the central
cube of a cell of $F$ at the initial time of that cell's block, and by \Cref{lem:propagate}
every subsequent entry again occurs through a central cube $C_{i'}$ at a time $(\tau+1)\beta$. We
may therefore take
\[
  E(i,\tau)\ =\ \big\{(i,\tau)\text{ is acceptable}\big\} .
\]
This event is increasing by \Cref{lem:accincreasing} and restricted to the super-cell
$\scell{i}{\tau}$, and \Cref{lem:mainSmallMu} bounds its probability directly.

\begin{proof}[Proof of \Cref{thm:smallmu}]
Let $d\ge2$ and $\eta,r,\lambda>0$. Fix $\varepsilon=\frac12$ and $\olap=\max\{d,3\}$, consistent
with \Cref{sec:smallmuparams}, then $\blk<\mathsf c_0$ so that the ratio condition of \Cref{thm:GS}
holds, and then $w$ large enough to satisfy both the lower bound of \Cref{thm:GS} and the two
requirements \eqref{eq:wlarge}. Note that all four choices are free of
$L$. Let $\alpha_0$ be the constant given by \Cref{thm:GS} for these parameters. It too is
free of $L$ and of the intensity.

We begin by fixing a first scale. By \Cref{lem:mainSmallMu}, applied at the sprinkled intensity
$(1-\varepsilon)\eta=\eta/2$ and under the conditioned law, there is
$L_1=L_1(\alpha_0,\lambda,\eta,r,d,\blk,\olap)$ such that the second condition of
\eqref{eq:GScond} holds for every $L\ge L_1$ and every $\mu\le\mu_*(L)$. Enlarging $L_1$ we may
also assume $\varepsilon^2\eta L_1^d\ge\alpha_0$, which then holds for every $L\ge L_1$. So
\eqref{eq:GScond} holds at every scale $L\ge L_1$ with $\mu\le\mu_*(L)$.

Next we fix the seeding constants. The right-hand side of the bound of \Cref{thm:GSsurround} is an
explicit function of $D$, $\eta$, $L$ and $d$ alone. It is non-increasing in $L$ once $L$ exceeds
the dimensional threshold $L_1(d)$ of \Cref{thm:GSsurround}, and it is summable in $D$. We enlarge
$L_1$ so that $L_1\ge L_1(d)$, and choose an integer $D_0\ge D_1(d)$, with $D_1$ the threshold of
\Cref{thm:GSsurround}, for which that function, evaluated at $L=L_1$,
is at most $\frac14$. This fixes $\tau_0$ and $\mathcal B$, hence $|\mathcal B|$, once and for all
and independently of $L$. Note that no probabilistic statement has been made at the scale $L_1$,
so the choice is not circular. Only the deterministic bound has been read off, and by
\Cref{thm:GSsurround} that bound carries no dependence on $\nu_E$ or on $\alpha_0$ beyond the
threshold condition \eqref{eq:GScond}. This is why the bound may be reused at $L_0\ge L_1$ even
though $\nu_E$ varies with $L$.

We now fix the working scale. By \Cref{lem:mainSmallMu} again, we choose $L_0\ge L_1$ so large
that in addition $\Prob[E^c]\le(4|\mathcal B|)^{-1}$ at intensity $\eta$, and we put
$\mu_*(\lambda):=\mu_*(L_0)=\log2/(\blk L_0^2)$. Fix $\mu\le\mu_*(\lambda)$ and work at $L=L_0$.
Then \eqref{eq:GScond} holds, so \Cref{thm:GS} and \Cref{thm:GSsurround} both apply, and
$\Prob[D>D_0]$ is at most the bound at $(D_0,\eta,L_0)$, which is itself at most the bound at
$(D_0,\eta,L_1)\le\frac14$. The event $E(i,\tau)=\{(i,\tau)\text{ acceptable}\}$ is increasing and
restricted to the super-cell by \Cref{lem:accincreasing}, as \Cref{thm:GS} requires, and carries
the infection by \Cref{lem:acccarries}, as \Cref{lem:seed,lem:propagate} require.

\Cref{thm:GS} now provides, almost surely, a two-sided Lipschitz surface $F$ on which $E$ holds.
The particle $u_0$ at the origin lies in the central cube $C_0$ at time $0$ and is infected, so
\Cref{lem:seed} applies. That is, with probability at least $\tfrac12$ some cell of $F$ has an infected particle in its
central cube at the initial time of its block. On that event
\Cref{lem:propagate} propagates the infection for ever, since $E$ carries the infection in the sense of
\Cref{def:carries}, and gives
\[
  \liminf_{t\to\infty}\frac{\rad(I_t)}{t}\ \ge\ \frac{L_0}{\blk L_0^2}\ =\ \frac1{\blk L_0}\ >\ 0 .
\]
Hence the infection survives, and spreads at positive speed, with probability at least
$\tfrac12>0$. By \Cref{lem:rerand} the same holds for the original construction.
\end{proof}

\section{Extinction at small intensity}
\label{sec:extinction}

In this section we consider the instantaneous model, that is $\lambda=\infty$ and
$\mathfrak r=\iso$, with $\mu\in(0,\infty]$. Our goal is to prove part (i) of \Cref{thm:density}.
We do not use the tessellation of \Cref{sec:tessellation} here. The scales $\ell_k$, $w_k$ and the
boxes $B_k$ introduced below belong to the renormalisation and are unrelated to the cells
$\cell{i}{\tau}$. Since $\lambda$, $\mathfrak r$ and $\mu$ are fixed throughout, we abbreviate the
law $\Prob^{\infty,\mu}_\eta$ of \Cref{sec:corners} to $\Prob_\eta$, and we write $\Prob_\varrho$
for the same law at intensity $\varrho$.

We follow the renormalisation scheme introduced for the zero-range process by Baldasso, Hil\'ario
and Ornelas \cite{BHO26}. The scheme has three steps. First, survival forces a half-crossing of a
space-time box. Second, a half-crossing at one scale forces two well-separated half-crossings at
the scale below. Third, the resulting recursion contracts once the base-scale probability is small.
Baldasso, Hil\'ario and Ornelas in turn adapt the constructions of \cite{FMUV23,HUVV22} for
renewal contact processes. The three analytic inputs of the scheme are different objects in the
continuum, and we prove them in \Cref{sec:horiz,sec:vertical,sec:trigger}. 

\begin{theorem}
\label{thm:extinction}
Let $d\ge1$, $r>0$ and $\mu\in(0,\infty]$. There is $\bar\eta=\bar\eta(\mu,r,d)>0$ such that for
every $\eta<\bar\eta$ the instantaneous infection started from a single particle dies out almost
surely.
\end{theorem}

\subsection{Half-crossings and cascading}
\label{sec:cascading}

We fix once and for all the scale multiplier and the offset parameter
\begin{equation}
\label{eq:atheta}
  a=8,\qquad \vartheta=\tfrac18 ,
\end{equation}
and a base scale $\ell_0\in\N$. We define
\begin{equation}
\label{eq:scales}
  \ell_k=a^k\ell_0,\qquad w_k=\ell_0\ell_k,\qquad
  B_k=[-w_k,w_k]^d\times[0,\ell_k],
\end{equation}
so that $w_0=\ell_0^2$, $w_{k+1}=aw_k$, and every scale-$k$ box has aspect ratio $2\ell_0$. A
\emph{scale-$k$ box} is any translate $B=Q\times[s,s+\ell_k]$ of $B_k$, with
$Q=y+[-w_k,w_k]^d$. For such a box we write
\[
  Q^+=y+[-w_k-2r,\ w_k+2r]^d,\qquad B^+=Q^+\times[s,s+\ell_k] .
\]
We need the collar of width $2r$ because the trace of an open path is càdlàg with jumps of
size at most $2r$, so it can leave a region by jumping over its boundary.

\begin{definition}[Half-crossings]
\label{def:halfcross}
Let $B=Q\times[s,s+h]$ with $Q=y+[-w,w]^d$.
\begin{itemize}
\item $T(B)$ is the event that there is an open path contained in $B$ whose trace has duration
  at least $h/2$.
\item For $1\le i\le d$, $S^+_i(B)$ is the event that there is an open path contained in $B^+$
  whose trace starts at a point with $x_i\le y_i+\vartheta w$ and later reaches a point with
  $x_i\ge y_i+w$; and $S^-_i(B)$ is its mirror image, obtained by replacing $x_i-y_i$ by
  $y_i-x_i$.
\end{itemize}
The \emph{half-crossing event} of $B$ is
$H(B)=T(B)\cup\bigcup_{i=1}^d\big(S^+_i(B)\cup S^-_i(B)\big)$.
\end{definition}

The two kinds of half-crossing are drawn in \Cref{fig:halfcross}.

\begin{figure}[!h]
\centering
\begin{tikzpicture}[x=1cm,y=0.8cm,>=Stealth,baseline={(0,0)}]
  \draw[dashed,gray] (-0.25,0) rectangle (5.25,4.5);
  \draw (0,0) rectangle (5,4.5);
  \draw[dotted] (-0.25,2.25) -- (5.25,2.25) node[right]{\small$s+h/2$};
  \node[below] at (2.5,0) {\small$y$}; \draw[dotted] (2.5,0) -- (2.5,4.5);
  \draw[very thick] plot[smooth] coordinates {(1.6,0) (2.9,0.9) (1.9,1.8) (3.1,2.6)};
  \node at (2.5,5.1) {$T(B)$};
  \node[below left] at (5.25,0) {\small$Q^+$};
\end{tikzpicture}
\hfill
\begin{tikzpicture}[x=1cm,y=0.8cm,>=Stealth,baseline={(0,0)}]
  \draw[dashed,gray] (-0.25,0) rectangle (5.25,4.5);
  \draw (0,0) rectangle (5,4.5);
  \draw[dotted] (2.5,0) -- (2.5,4.5); \node[below] at (2.5,-0.02) {\small$y_i$};
  \draw[dotted] (3.125,-0.72) -- (3.125,4.5);
  \node[below] at (3.125,-0.72) {\small$y_i+\vartheta w$};
  \node[below] at (5,-0.02) {\small$y_i+w$};
  \draw[very thick] plot[smooth] coordinates {(2.9,0.6) (2.2,1.4) (3.6,2.2) (4.2,3.1) (5.05,3.9)};
  \fill (2.9,0.6) circle (1.6pt); \fill (5.05,3.9) circle (1.6pt);
  \node at (2.5,5.1) {$S^+_i(B)$};
\end{tikzpicture}
\caption{The two kinds of half-crossing of a box $B=Q\times[s,s+h]$, $Q=y+[-w,w]^d$, with time
running upwards. The dashed rectangle is the collar $Q^+$, which absorbs the jumps of the trace.
Left: $T(B)$, a path contained in $B$ of duration at least $h/2$. Right: $S^+_i(B)$, a path
contained in $B^+$ whose trace starts to the left of $y_i+\vartheta w$ and reaches
$y_i+w$. We use the offset $\vartheta w$, rather than the centre $y_i$ itself, so that a
half-crossing can be caught by the finite grid of windows of \Cref{lem:localisation}.}
\label{fig:halfcross}
\end{figure}
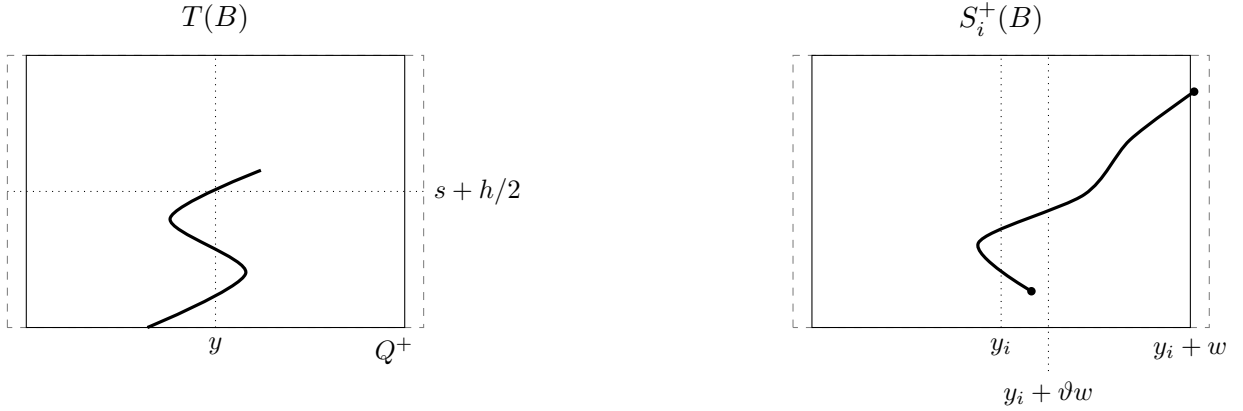

\begin{lemma}
\label{lem:Hprops}
For every space-time box $B=Q\times[s,s+h]$ the event $H(B)$ is increasing in the marked
configuration, and it is determined by the restriction of the marked configuration to the
particles which visit the $6r$-neighbourhood of $Q$ during $[s,s+h]$.
\end{lemma}

\begin{proof}
We first show monotonicity. Each of the events in \Cref{def:halfcross} is the existence of an
open path with a property of its trace alone, so monotonicity follows from \Cref{prop:mono}(ii)
applied to the witnessing open path.

We now show locality. A witnessing path has its trace in $Q^+$, the $2r$-neighbourhood of $Q$.
Clause (ii) of \Cref{def:openpath} is imposed on the particle $u_j$ over the interval
$(t_j,t_{j+1}]$, which is closed at its right endpoint, and at the right endpoint $u_j$ need no
longer carry the trace. More precisely, for $t<t_{j+1}$ we have $X_{u_j}(t)=\gamma(t)\in Q^+$,
while at $t=t_{j+1}$ the particle $u_j$ is in contact with $u_{j+1}$ and so lies within $2r$ of
$\gamma(t_{j+1})\in Q^+$. Hence $X_{u_j}(t)$ lies within $2r$ of $Q^+$ throughout
$[t_j,t_{j+1}]$. By \Cref{def:openpath}, whether a mark of $u_j$ carried there is effective
depends only on the particles within distance $2r$ of $u_j$ at that instant. It therefore depends
only on the particles visiting the $4r$-neighbourhood of $Q^+$, that is the $6r$-neighbourhood
of $Q$.
\end{proof}

\begin{lemma}[Survival and half-crossings]
\label{lem:surv2H}
Let $u$ be a particle with $X_u(0)\in[-\vartheta w_k,\vartheta w_k]^d$. If the infection started
from $u$ alone survives past time $\ell_k/2$, then $H(B_k)$ occurs. In particular this applies to
a single particle at the origin.
\end{lemma}

\begin{proof}
By hypothesis there is an open path from $(u,0)$ whose trace $\gamma$ is defined on
$[0,\ell_k/2]$ and satisfies $\gamma(0)=X_u(0)$, so that $|\gamma_i(0)|\le\vartheta w_k$ for every
$i$. If $\gamma$ remains in $[-w_k,w_k]^d$ on $[0,\ell_k/2]$, then the path is contained in
$B_k$ and has duration $\ell_k/2$, so $T(B_k)$ occurs.

Otherwise let $t_1\le\ell_k/2$ be the first time at which $|\gamma_i(t)|\ge w_k$ for some
$i$. Before $t_1$ all coordinates of $\gamma$ are smaller than $w_k$ in absolute value. At
$t_1$ they are at most $w_k+2r$ in absolute value, because the jump has size at most $2r$.
Hence the restriction of the path to $[0,t_1]$ is contained in $B_k^+$. If
$\gamma_i(t_1)\ge w_k$, then its trace starts at $\gamma_i(0)\le\vartheta w_k$ and reaches
$\gamma_i(t_1)\ge w_k$, so $S^+_i(B_k)$ occurs. If instead $\gamma_i(t_1)\le-w_k$, then
$\gamma_i(0)\ge-\vartheta w_k$ and $S^-_i(B_k)$ occurs. Note that the restriction of an open
path to a subinterval of its time interval is again an open path, since
\Cref{def:openpath}(ii) is inherited by subintervals.
\end{proof}

We now turn to the geometric step of the renormalisation. We show that a half-crossing at scale
$k+1$ forces two half-crossings at scale $k$, in boxes taken from a finite list. The two boxes
are separated either in time by at least $\ell_{k+1}/4$ or in space by at least $w_{k+1}/4$;
see \Cref{fig:cascade}. We first isolate the localisation step, which we will use four times.

Let
\begin{equation}
\label{eq:grids}
  \mathfrak Z_k=\tfrac{w_k}{4}\Z^d\cap\big[-2w_{k+1},2w_{k+1}\big]^d,
  \qquad
  \mathfrak J_k=\Big\{\big[\tfrac j2\ell_k,\ \tfrac j2\ell_k+\ell_k\big]:\
  0\le j\le 2a\Big\},
\end{equation}
and let $\mathfrak B_k=\{(z+[-w_k,w_k]^d)\times I:\ z\in\mathfrak Z_k,\ I\in\mathfrak J_k\}$ be the
corresponding finite family of scale-$k$ boxes, of cardinality
\begin{equation}
\label{eq:boxcount}
  |\mathfrak B_k|\ \le\ (16a+1)^d(2a+1)\ =:\ C_{\mathrm{box}}(d) .
\end{equation}
Note that $\mathfrak Z_k$ has mesh $w_k/4=2\vartheta w_k$, so every point of
$[-2w_{k+1},2w_{k+1}]^d$ lies within $\vartheta w_k$ of a point of $\mathfrak Z_k$ in every
coordinate. Note also that every $\alpha\in[0,\ell_{k+1}]$ lies in some $I\in\mathfrak J_k$ with
$\alpha+\ell_k/2\le\sup I$.

\begin{lemma}[Localisation]
\label{lem:localisation}
Let $\gamma$ be the trace of an open path defined on $[\alpha,\beta]\subseteq[0,\ell_{k+1}]$ with
$\gamma(\alpha)\in[-2w_{k+1},2w_{k+1}]^d$. Let $z\in\mathfrak Z_k$ satisfy
$|\gamma_i(\alpha)-z_i|\le\vartheta w_k$ for all $i$, let $I\in\mathfrak J_k$ satisfy $\alpha\in I$
and $\alpha+\ell_k/2\le\sup I$, and put $W=z+[-w_k,w_k]^d$ and $B=W\times I$. Suppose that
\begin{enumerate}[(a)]
\item\label{loc:dur} either $\beta-\alpha\ge\ell_k/2$,
\item\label{loc:esc} or $\gamma(\beta)\notin W$.
\end{enumerate}
Then $H(B)$ occurs.
\end{lemma}

\begin{proof}
Put $\beta'=\min\{\beta,\sup I\}$. We first suppose that $\gamma(t)\in W$ for every
$t\in[\alpha,\beta']$. Under (\ref{loc:dur}) we have $\beta\ge\alpha+\ell_k/2$ and
$\sup I\ge\alpha+\ell_k/2$, so $\beta'\ge\alpha+\ell_k/2$. Under (\ref{loc:esc}) we have
$\beta\notin[\alpha,\beta']$, since $\gamma(\beta)\notin W$, so $\beta'=\sup I\ge\alpha+\ell_k/2$.
In either case the restriction of the path to $[\alpha,\beta']\subseteq I$ is contained in $B$ and
has duration at least $\ell_k/2$, so $T(B)$ occurs.

Otherwise let $t_1\le\beta'$ be the first time at which $|\gamma_i(t)-z_i|\ge w_k$ for some $i$,
say with $\gamma_i(t_1)-z_i\ge w_k$. The infimum is attained because $\gamma$ is right-continuous
and $\{x:\max_i|x_i-z_i|\ge w_k\}$ is closed. Before $t_1$ the trace lies in $W$, and at $t_1$ it
lies in $W^+$, so the restriction of the path to $[\alpha,t_1]$ is contained in $B^+$. Its trace
starts at $\gamma_i(\alpha)\le z_i+\vartheta w_k$ and reaches $\gamma_i(t_1)\ge z_i+w_k$, which is
$S^+_i(B)$.
\end{proof}

Note that under (\ref{loc:esc}) the duration $\ell_k/2$ needed for $T(B)$ comes from the time
window $I$ and not from the duration of the trace. For this reason the spatial case of
\Cref{lem:cascade} below does not require a lower bound on $\beta-\alpha$.

\begin{figure}[!h]
\centering
\begin{tikzpicture}[x=0.75cm,y=0.62cm,>=Stealth,baseline={(0,0)}]
  \draw (0,0) rectangle (8,8); \node[above] at (4,8) {\small$B_{k+1}$};
  \draw[fill=black!8] (2.4,0) rectangle (4.4,1);
  \draw[fill=black!8] (3.2,6.6) rectangle (5.2,7.6);
  \draw[very thick] plot[smooth] coordinates {(3.2,0.2) (4.6,1.6) (3.0,3.2) (4.8,4.8) (3.4,6.2) (4.4,7.6)};
  \node[left] at (2.4,0.5) {\small$B$};
  \node[right] at (5.2,7.1) {\small$B'$};
  \draw[<->] (7.0,1) -- (7.0,6.6) node[midway,fill=white]{\small$\ge\tfrac{\ell_{k+1}}4$};
\end{tikzpicture}
\hfill
\begin{tikzpicture}[x=0.75cm,y=0.62cm,>=Stealth,baseline={(0,0)}]
  \draw (0,0) rectangle (8,8); \node[above] at (4,8) {\small$B_{k+1}$};
  \draw[fill=black!8] (0.4,0.1) rectangle (2.4,1.1);
  \draw[fill=black!8] (6.6,6.1) rectangle (8.0,7.1);
  \draw[very thick] plot[smooth] coordinates {(1.0,0.6) (2.2,2.0) (1.4,3.4) (4.0,4.4) (6.2,5.4) (7.6,6.6)};
  \node[below] at (1.4,0.1) {\small$B$};
  \node[above] at (7.3,7.1) {\small$B'$};
  \draw[<->] (2.4,3.2) -- (6.6,3.2) node[midway,fill=white]{\small$\ge\tfrac{w_{k+1}}4$};
\end{tikzpicture}
\caption{Cascading. Left: on $T(B_{k+1})$ the path has duration at least $\ell_{k+1}/2$, so it
is present during two time intervals of length $\ell_k$ separated by at least $\ell_{k+1}/4$,
and \Cref{lem:localisation} produces a half-crossing of a scale-$k$ box in each. Right: on
$S^+_i(B_{k+1})$ the path travels a distance at least $(1-\vartheta)w_{k+1}$ in the $i$-th
coordinate, so it produces half-crossings of two scale-$k$ boxes separated in space by at least
$w_{k+1}/4$. In both cases the two boxes are drawn from the finite family $\mathfrak B_k$.}
\label{fig:cascade}
\end{figure}
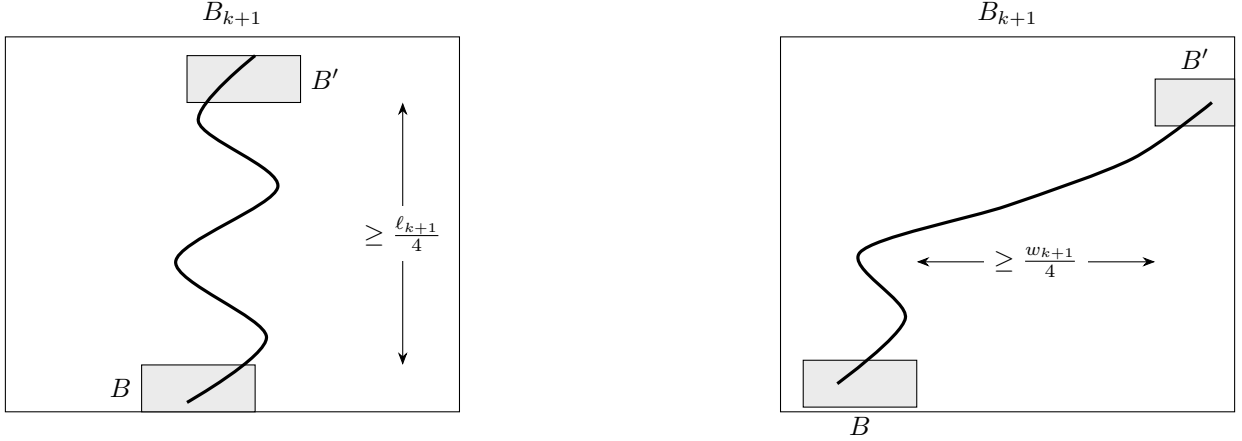

\begin{lemma}[Cascading]
\label{lem:cascade}
Let $C_{\mathrm{ent}}=(2d+1)\,C_{\mathrm{box}}(d)^2$ with $C_{\mathrm{box}}$ as in
\eqref{eq:boxcount}. For every $k\ge0$ there is a family $\mathcal L_k$ of at most
$C_{\mathrm{ent}}$ pairs $(B,B')$ of boxes of $\mathfrak B_k$ such that
\begin{equation}
\label{eq:cascade}
  H(B_{k+1})\ \subseteq\ \bigcup_{(B,B')\in\mathcal L_k}\big(H(B)\cap H(B')\big),
\end{equation}
and such that every pair $(B,B')\in\mathcal L_k$ is separated either in time by at least
$\ell_{k+1}/4$ or in space by at least $w_{k+1}/4$, provided $\ell_0$ is large enough that
$2r\le w_k/8$.
\end{lemma}

\begin{proof}
We treat the $2d+1$ events constituting $H(B_{k+1})$ separately. Each of them will produce a
pair of boxes of $\mathfrak B_k$. Taking for $\mathcal L_k$ the set of all pairs of boxes of
$\mathfrak B_k$ which can arise then gives \eqref{eq:cascade} with the stated cardinality.

We start with the temporal half-crossing. On $T(B_{k+1})$ there is an open path
contained in $B_{k+1}$, with trace $\gamma$ defined on some
$[\alpha,\beta]\subseteq[0,\ell_{k+1}]$ of length at least $\ell_{k+1}/2=4\ell_k$. We apply
\Cref{lem:localisation} twice, first at the time $\alpha_1=\alpha$ and then at the time
$\alpha_2=\beta-\ell_k/2$. In both cases the remaining duration is at least $\ell_k/2$, so
hypothesis (\ref{loc:dur}) holds. Moreover $\gamma(\alpha_j)\in[-w_{k+1},w_{k+1}]^d$, so the
remark following \eqref{eq:grids} gives $z\in\mathfrak Z_k$ and $I\in\mathfrak J_k$ as required.
We obtain boxes $B=W\times I$ and $B'=W'\times I'$ in $\mathfrak B_k$ with $H(B)$ and $H(B')$.
Their time intervals satisfy $\sup I\le\alpha_1+\ell_k$ and $\inf I'\ge\alpha_2-\ell_k/2$, so
they are separated in time by at least
\[
  \alpha_2-\alpha_1-\tfrac32\ell_k\ \ge\ 4\ell_k-\tfrac12\ell_k-\tfrac32\ell_k
  \ =\ 2\ell_k\ =\ \tfrac14\ell_{k+1} .
\]

We now turn to the spatial half-crossings. We consider $S^+_i(B_{k+1})$; the other $2d-1$
cases are identical after a reflection or a permutation of the coordinates. There is an open
path contained in $B_{k+1}^+$ whose trace $\gamma$, defined on
$[\alpha,\beta]\subseteq[0,\ell_{k+1}]$, satisfies $\gamma_i(\alpha)\le\vartheta w_{k+1}$ and
$\gamma_i(\beta)\ge w_{k+1}$. Shrinking $\beta$, we may assume that $\beta$ is the first such
time. Since $2w_k=w_{k+1}/4$ and $\vartheta=1/8$ we have $\gamma_i(\alpha)<w_{k+1}-2w_k$, so
\[
  \rho:=\sup\{t\le\beta:\ \gamma_i(t)\le w_{k+1}-2w_k\}
\]
is well defined and lies in $[\alpha,\beta)$. Right-continuity of $\gamma$ and the bound $2r$
on its jumps give
\begin{equation}
\label{eq:rhoposition}
  w_{k+1}-2w_k\ \le\ \gamma_i(\rho)+2r,\qquad \gamma_i(\rho)\le w_{k+1}-2w_k+2r .
\end{equation}

We apply \Cref{lem:localisation} at the time $\alpha_1=\alpha$, with $z\in\mathfrak Z_k$ within
$\vartheta w_k$ of $\gamma(\alpha)$ coordinatewise, and at the time $\alpha_2=\rho$, with
$z'\in\mathfrak Z_k$ within $\vartheta w_k$ of $\gamma(\rho)$. The remark following
\eqref{eq:grids} gives the corresponding $I,I'\in\mathfrak J_k$. In both cases we have no bound
on the remaining duration, and we use hypothesis (\ref{loc:esc}) instead. That is, the trace ends
outside the window, because
\[
  z_i+w_k\ \le\ \vartheta w_{k+1}+(1+\vartheta)w_k
  \ =\ \Big(\tfrac18+\tfrac{9}{64}\Big)w_{k+1}\ <\ w_{k+1}
\]
and, by \eqref{eq:rhoposition},
\[
  z'_i+w_k\ \le\ w_{k+1}-2w_k+2r+(1+\vartheta)w_k\ =\ w_{k+1}-(1-\vartheta)w_k+2r\ <\ w_{k+1},
\]
where we used $2r\le w_k/8$. Hence in each case $\gamma_i(\beta)\ge w_{k+1}$ exceeds the largest
$i$-th coordinate attained in the window. In other words $\gamma(\beta)\notin W$ and
$\gamma(\beta)\notin W'$, which is (\ref{loc:esc}). \Cref{lem:localisation} therefore gives $H(B)$
and $H(B')$ for $B=W\times I$ and $B'=W'\times I'$ in $\mathfrak B_k$.

It remains to bound the spatial separation of $W$ and $W'$ from below. On the one hand
$W\subseteq\{x_i\le z_i+w_k\}$ and $z_i+w_k\le\vartheta w_{k+1}+(1+\vartheta)w_k$, that is
$W\subseteq\{x_i\le\frac{17}{64}w_{k+1}\}$. On the other hand, by \eqref{eq:rhoposition},
$z'_i\ge\gamma_i(\rho)-\vartheta w_k\ge w_{k+1}-2w_k-2r-\vartheta w_k$, so
$W'\subseteq\{x_i\ge z'_i-w_k\}$ with
\[
  z'_i-w_k\ \ge\ w_{k+1}-(3+\vartheta)w_k-2r\ \ge\ w_{k+1}-\tfrac{25}{64}w_{k+1}-\tfrac18 w_k
  \ \ge\ \tfrac{38}{64}\,w_{k+1} .
\]
The two windows are therefore separated in space by at least
$(\frac{38}{64}-\frac{17}{64})w_{k+1}=\frac{21}{64}w_{k+1}\ge w_{k+1}/4$.
\end{proof}

\subsection{Horizontal decoupling}
\label{sec:horiz}

In this section we show that two half-crossings in boxes which are far apart in space are
almost independent. For a Poisson process on path space this is close to an exact statement.
The paths which enter the two regions during the relevant time window form two independent
Poisson processes, apart from those which enter both regions, and the expected number of the
latter is small.

\begin{lemma}[Horizontal decoupling]
\label{lem:horiz}
There is $C=C(d)<\infty$ such that the following holds. Let
$B,B'\subseteq\R^d\times[0,\Lambda]$ and let $A,A'$ be increasing events determined, in the sense
of \Cref{lem:Hprops}, by $B$ and $B'$ respectively. By \Cref{lem:Hprops} the regions on which
$A$ and $A'$ depend are the $6r$-neighbourhoods of the spatial projections of $B$ and $B'$, and
not those projections themselves. Suppose that these two regions are contained in cubes of side
$W$ and lie at distance $D$ from one another. Then
\begin{equation}
\label{eq:horiz}
  \Prob_\eta(A\cap A')\ \le\ \Prob_\eta(A)\,\Prob_\eta(A')
  +C\eta\big(W^d+\Lambda^{d/2}\big)e^{-D^2/(2d\Lambda)} .
\end{equation}
\end{lemma}

\begin{proof}
Let $S$ be the set of paths which visit the region relevant to $A$ during $[0,\Lambda]$, and
let $S'$ be the set of paths which visit the region relevant to $A'$ during $[0,\Lambda]$. We
split path space into the three disjoint measurable sets $S\setminus S'$, $S'\setminus S$ and
$S\cap S'$. By the restriction theorem the point processes
$\mathcal P_1,\mathcal P_2,\mathcal P_{12}$ carried by these sets are independent Poisson
processes. The recovery marks are indexed by the particles themselves, so the recovery marks
attached to disjoint sets of particles are independent as well. This is all we need here. Note
that at $\lambda=\infty$ the transmission marks are discarded (\Cref{sec:graphical}). At finite
$\lambda$ the transmission marks are indexed by ordered pairs of particles and would not
decompose in this way.

The event $A$ is determined by $\mathcal P_1\cup\mathcal P_{12}$ and $A'$ by
$\mathcal P_2\cup\mathcal P_{12}$. Hence on $\{\mathcal P_{12}=\emptyset\}$ they are determined by
$\mathcal P_1$ and by $\mathcal P_2$ respectively, and are therefore independent. We write
$A(\mathcal P_1)$ for the event evaluated on $\mathcal P_1$ alone. Then
\[
  \Prob(A\cap A')\ \le\ \Prob\big(A(\mathcal P_1)\big)\Prob\big(A'(\mathcal P_2)\big)
  +\Prob(\mathcal P_{12}\ne\emptyset)
  \ \le\ \Prob_\eta(A)\Prob_\eta(A')+\E|\mathcal P_{12}| .
\]
The middle inequality holds because $\mathcal P_1\subseteq\mathcal P$ and $A$ is increasing, and
the last inequality holds because $\Prob(N\ne0)\le\E N$.

It remains to bound $\E|\mathcal P_{12}|$. A path of $S\cap S'$ visits the two regions
relevant to $A$ and to $A'$ during a time interval of length $\Lambda$. By hypothesis these
regions are at distance at least $D$, so the path displaces by at least $D$ during that
interval. By the strong Markov property at the first visit and the reflection bound
$\Prob_y\big(\sup_{s\le\Lambda}|B_s-y|\ge D\big)\le4d\,e^{-D^2/(2d\Lambda)}$, we get
\[
  \E|\mathcal P_{12}|\ \le\ 4d\,e^{-D^2/(2d\Lambda)}\;
  \eta\int_{\R^d}\Prob_x\big(\text{the path visits the cube of side }W\text{ before }\Lambda\big)
  \mathrm dx .
\]
A Brownian motion started at distance $R$ from the cube reaches it before time $\Lambda$
with probability at most $4de^{-R^2/(2d\Lambda)}$, so the last integral is at most
$C_d(W^d+W^{d-1}\Lambda^{1/2}+\Lambda^{d/2})$. Moreover
$W^{d-1}\Lambda^{1/2}\le W^d+\Lambda^{d/2}$ by Young's inequality with exponents $d/(d-1)$ and
$d$. This gives \eqref{eq:horiz}.
\end{proof}

\subsection{Local mixing and vertical decoupling}
\label{sec:vertical}

In this section we consider two half-crossings in boxes which are far apart in time. These
require a different treatment. We first explain why the direct approach fails. Conditionally
on $\mathcal F_{t_0}$ and on the configuration $\Psi$ at time $t_0$, the configuration at time
$t_0+\Delta$ is obtained by displacing the points of $\Psi$ by independent Gaussian vectors of
variance $\Delta$. This conditional law is not stochastically dominated by any Poisson process.
To see this, let $R$ be a region containing $m$ points of $\Psi$, each at distance at least
$\zeta$ from the boundary of $R$. Each of them still lies in $R$ after the displacement with
probability at least $1-e^{-\zeta^2/(2d\Delta)}$. Hence $R$ contains at least $m$ displaced
points with probability at least $1-me^{-\zeta^2/(2d\Delta)}$, which is arbitrarily close to $1$
once $\zeta/\Delta^{1/2}$ is large. On the other hand, a Poisson variable of mean
$(1+\epsilon)m$ is smaller than $m$ with probability at least $e^{-c\epsilon^2m}$. Domination
would require the first of these two failure probabilities to be the larger, and this fails as
soon as $\zeta^2\ge C\Delta(\epsilon^2m+\log m)$. The obstruction therefore disappears at rate
$\exp\{-c\epsilon^2\eta\Delta^{d/2}\}$ per cell of side $\sqrt\Delta$. The correct statement is
a coupling which succeeds with high probability, with that error, and local mixing provides
it. The geometry of the coupling is drawn in \Cref{fig:mixing} and that of the decoupling in
\Cref{fig:vertical}. Below, $\kappa$ denotes the H\"older exponent of the heat kernel of
(H\ref{H2}), which for the Gaussian kernel is $\kappa=1$. We carry $\kappa$ as a parameter only
because the hypothesis $\Delta\ge c_0\ell^2\epsilon^{-4/\kappa}$ of the mixing estimates is
stated for a general $\kappa$. The parameter
$\varrho$ is a density.

The estimate we use is
\Cref{thm:mixingupper}, which we proved in \Cref{sec:mixing}. Recall that it states that the mixed
configuration is dominated by a Poisson process of slightly larger intensity, with error
$\Gamma\exp\{-C\varrho\epsilon^2\Delta^{d/2}\}$ per box. We use it here in the upper form, since
the inclusion needed for extinction runs in the opposite direction to the one required in
\Cref{sec:lipschitz}.

\begin{figure}[!h]
\centering
\begin{tikzpicture}[x=1cm,y=0.7cm,>=Stealth]
  \draw[->,gray] (-5.4,-0.4) -- (-5.4,7.8) node[above,black]{$t$};
  \draw[|-|] (-4,-1.0) -- (4,-1.0) node[right]{\small$Q_K$};
  \draw[|-|] (-3,-1.9) -- (3,-1.9) node[right]{\small$Q_{K'}$};
  \draw[|-|] (-2,-2.8) -- (2,-2.8) node[right]{\small$Q_W$};
  \draw[<->] (-3,-2.4) -- (-2,-2.4) node[midway,below]{\small$\mathsf G$};
  \draw[<->] (-4,-1.5) -- (-3,-1.5) node[midway,below]{\small$\mathsf G$};
  \draw[dotted] (-5.4,3) -- (4,3); \node[left] at (-5.4,3) {\small$t_0$};
  \draw[dotted] (-5.4,4.2) -- (4,4.2); \node[left] at (-5.4,4.2) {\small$t_0+\Delta$};
  \draw[dotted] (-5.4,7) -- (4,7); \node[left] at (-5.4,7) {\small$\Lambda$};
  \draw[fill=black!8] (-2,0.2) rectangle (2,3); \node at (0,1.6) {$B$, event $A$};
  \draw[fill=black!8] (-2,4.2) rectangle (2,7);  \node at (0,5.6) {$B'$, event $A'$};
  \draw[<->] (3.4,3) -- (3.4,4.2) node[midway,right]{\small$\Delta$};
  \draw[gray,->] (-4.6,3) .. controls (-4.3,3.6) .. (-3.6,4.2);
  \node[gray,above left] at (-4.0,3.9) {\scriptsize$\mathcal M$};
\end{tikzpicture}
\caption{The geometry of the vertical decoupling. The event $A$ is carried by $B$, which lies
below time $t_0$, and $A'$ by $B'$, which lies above $t_0+\Delta$. Conditionally on the
configuration at time $t_0$, the local mixing estimate replaces the particles of $Q_K$ at time
$t_0$ by a Poisson process of intensity $\eta(1+\epsilon)$ containing all of them inside
$Q_{K'}$ at time $t_0+\Delta$. The two spatial margins of width $\mathsf G$ absorb the particles which
enter from outside, on the events $\mathcal M$ and $\mathcal N$.}
\label{fig:vertical}
\end{figure}
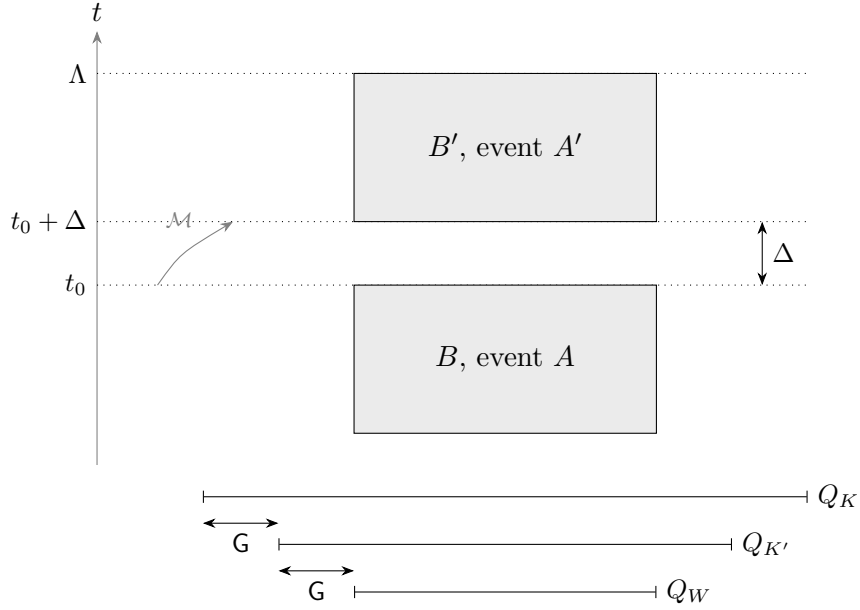

\begin{proposition}[Vertical decoupling]
\label{prop:vertical}
There is $c=c(d)>0$ such that the following holds. Let $\eta>0$, $\epsilon\in(0,1)$,
$0<\Delta\le\Lambda$ and $W>0$, and let $\ell$ satisfy
\begin{equation}
\label{eq:ellcond}
  \Delta\ge c_0\ell^2(\epsilon/3)^{-4/\kappa}\qquad\text{and}\qquad\eta\ell^d\ge1 .
\end{equation}
Let $\mathsf G\ge\max\{c_1\Delta^{1/2}(\epsilon/3)^{-1/d},\,12r\}$, and put $K'=W+2\mathsf G$ and
$K=K'+2\mathsf G$. Let
$B\subseteq Q_W\times[0,t_0]$ and $B'\subseteq Q_W\times[t_0+\Delta,\Lambda]$, and let $A,A'$ be
increasing events determined by $B$ and $B'$ in the sense of \Cref{lem:Hprops}. Then
\begin{equation}
\label{eq:vertical}
  \Prob_\eta(A\cap A')\le\Prob_\eta(A)\Big[\Prob_{\eta(1+\epsilon)}(A')
  +\Gamma\,e^{-c\eta\epsilon^2\Delta^{d/2}}\Big]
  +\Big(\tfrac K\ell\Big)^de^{-c\eta\epsilon^2\ell^d}
  +C_d\,\eta\,(K+\Lambda^{1/2})^de^{-\mathsf G^2/(8d\Lambda)} ,
\end{equation}
where $\Gamma=C_d\,\epsilon^{-d}\big(K\Delta^{-1/2}\big)^{d(d+1)}$ is the factor
\eqref{eq:Gamma} of \Cref{thm:mixingupper} at mixing parameter $\epsilon/3$.
\end{proposition}

\begin{proof}
Let $\Psi$ be the configuration at time $t_0$, which is a Poisson process of intensity $\eta$.
We introduce three auxiliary events.
\begin{itemize}
\item $\mathcal O$ is the event that every subcube of side $\ell$ of $Q_K$ carries at most
  $(1+\epsilon/3)\eta\ell^d$ points of $\Psi$. This event is $\mathcal F_{t_0}$-measurable, and
  by the Poisson Chernoff bound $\Prob(\mathcal O^c)\le(K/\ell)^de^{-\eta\ell^d\epsilon^2/36}$.
\item $\mathcal M$ is the event that no particle lying outside $Q_K$ at time $t_0$ lies in
  $Q_{K'}$ at time $t_0+\Delta$. By stationarity and the reflection bound,
  $\Prob(\mathcal M^c)\le C_d\eta(K+\Delta^{1/2})^de^{-\mathsf G^2/(2d\Delta)}$.
\item $\mathcal N$ is the event that no particle lying outside $Q_{K'}$ at time $t_0+\Delta$
  enters the $6r$-neighbourhood of $Q_W$ during $[t_0+\Delta,\Lambda]$. Similarly
  $\Prob(\mathcal N^c)\le C_d\eta(W+\Lambda^{1/2})^de^{-\mathsf G^2/(8d\Lambda)}$. Here the margin
  is $\mathsf G-6r$ rather than $\mathsf G$, and this is at least $\mathsf G/2$ as soon as
  $\mathsf G\ge12r$.
\end{itemize}

On $\mathcal M\cap\mathcal N$, every particle relevant to $A'$ lies in $Q_{K'}$ at time
$t_0+\Delta$ and lay in $Q_K$ at time $t_0$. We apply \Cref{thm:mixingupper} conditionally on
$\mathcal F_{t_0}$, on the event $\mathcal O$, with $\varrho=(1+\epsilon/3)\eta$, mixing
parameter $\epsilon/3$ and $J$ the set of particles of $Q_K$. Its hypotheses are the two
conditions \eqref{eq:ellcond} and $\mathsf G\ge c_1\Delta^{1/2}(\epsilon/3)^{-1/d}$, and we have
$\varrho\ell^d\ge\eta\ell^d\ge1$. This produces a Poisson process $\psi$ of intensity
$(1+\epsilon/3)^2\eta\le(1+\epsilon)\eta$, independent of $\mathcal F_{t_0}$, which contains all
the relevant positions outside an event of conditional probability at most
$\Gamma e^{-c\eta\epsilon^2\Delta^{d/2}}$. The inclusion given by \Cref{thm:mixingupper} holds
inside $Q_{K'}$ only. Before comparing with a stationary system we therefore superpose on
$\psi$, outside $Q_{K'}$, an independent Poisson process of intensity $\eta(1+\epsilon)$. The
union is Poisson of intensity at most $\eta(1+\epsilon)$ everywhere, and it still contains every
position relevant to $A'$, since by the choice $K'=W+2\mathsf G$ those positions lie in $Q_{K'}$.
We attach to its points independent Brownian motions and independent marks, matched to the
original particles under the inclusion. This yields a stationary marked system of intensity
$\eta(1+\epsilon)$, independent of $\mathcal F_{t_0}$, whose particles contain all those relevant
to $A'$. Since $A'$ is increasing, we obtain
\[
  \Prob\big(A'\cap\mathcal M\cap\mathcal N\mid\mathcal F_{t_0}\big)
  \le\Prob_{\eta(1+\epsilon)}(A')+\Gamma\,e^{-c\eta\epsilon^2\Delta^{d/2}}
  \qquad\text{on }\mathcal O .
\]
We multiply by $\mathbf1_{A\cap\mathcal O}$, which is $\mathcal F_{t_0}$-measurable, take
expectations, and add $\Prob(\mathcal O^c)+\Prob(\mathcal M^c)+\Prob(\mathcal N^c)$. This gives
\eqref{eq:vertical}.
\end{proof}

\subsection{The recursion}
\label{sec:recursion}

We now set up the recursion. The two conditions in \eqref{eq:ellcond} bound $\ell$ from above
and from below. Propagating them through the scales determines the admissible sprinkling
exponent. We set
\begin{equation}
\label{eq:sstar}
  s_\ast:=\frac{d\kappa}{4(\kappa+d)}\qquad\Big(=\tfrac d{4(1+d)}\ \text{ for }\kappa=1\Big),
\end{equation}
fix $s\in(0,s_\ast)$, and put
\begin{equation}
\label{eq:sprink}
  \delta_k=\ell_k^{-s},\qquad \eta_\infty=\ell_0^{-s},\qquad
  \eta_k=\eta_\infty+\sum_{j\ge k}\delta_j,\qquad p_k=\Prob_{\eta_k}\big(H(B_k)\big),
\end{equation}
so that $(\eta_k)$ decreases to $\eta_\infty$ and
$\eta_0=\big(1+(1-a^{-s})^{-1}\big)\ell_0^{-s}$. Note that $s_\ast<\min\{d,\kappa\}/4$, and
that $s_\ast=1/6$ when $\kappa=1$ and $d=2$. Hence every $s\le1/10$ is admissible in every
dimension when $\kappa=1$.

\begin{proposition}[Recursion]
\label{prop:recursion}
Let $s\in(0,s_\ast)$, $\theta_\ast:=\tfrac d2-2s-\tfrac{2sd}\kappa>0$ and $\mathsf N:=(d+2)^2$. There
are $C,c>0$, depending only on $d,s,\kappa$, such that for $\ell_0$ large and every $k\ge0$,
\begin{equation}
\label{eq:recursion}
  p_{k+1}\ \le\ C_{\mathrm{ent}}\,p_k^2+C_{\mathrm{ent}}\,\epsilon_k,
  \qquad
  \epsilon_k=C\,\ell_0^{\mathsf N}\ell_k^{\mathsf N}\Big(
    \exp\big\{-c\,\ell_0^{d/2-s}a^{k\theta_\ast}\big\}
    +\exp\big\{-c\,\ell_0^{2}\ell_{k+1}\big\}\Big) .
\end{equation}
\end{proposition}

The first exponential in $\epsilon_k$ comes from the local mixing. The second comes from the
Gaussian tails of the horizontal decoupling and of the last term of \eqref{eq:vertical}. Both
exponents grow exponentially in $k$, at the respective rates $a^{\theta_\ast}$ and $a$, so
$\epsilon_k$ decays superexponentially. This is all that \Cref{lem:contraction} requires. Note
that neither term dominates the other in general, since $\theta_\ast>1$ once $d$ is large.

\begin{proof}
By \Cref{lem:cascade} and a union bound,
$p_{k+1}\le C_{\mathrm{ent}}\max_{(B,B')}\Prob_{\eta_{k+1}}\big(H(B)\cap H(B')\big)$, where the
maximum is over the pairs produced by the cascading. Throughout the proof we use that
$\Prob_{\eta_{k+1}}(H(B))\le\Prob_{\eta_k}(H(B))=p_k$ for every scale-$k$ box $B$. This follows
from \Cref{lem:Hprops} together with the superposition coupling of \Cref{prop:mono}(ii)--(iii).
All the boxes involved lie in $[-2w_{k+1},2w_{k+1}]^d\times[0,2\ell_{k+1}]$, so we may take
$\Lambda=2\ell_{k+1}$ throughout.

We start with the horizontally separated pairs. We apply \Cref{lem:horiz} with $W=4w_k$ and
$\Lambda=2\ell_{k+1}$. By \Cref{lem:cascade} the two boxes are separated by at least
$\frac{21}{64}w_{k+1}$. The regions on which the two events depend are the $6r$-neighbourhoods
of their projections, so these regions are separated by at least $\frac{21}{64}w_{k+1}-12r$.
Since $2r\le w_0/8$ and $w_{k+1}\ge8w_0$ we have $2r\le w_{k+1}/64$, hence
$12r\le\frac{6}{64}w_{k+1}$, and we may take
\[
  D\ =\ \tfrac{15}{64}w_{k+1} .
\]
The same estimate shows that the regions are contained in cubes of side
$2w_k+4r+12r\le3w_k\le W$. Then
$D^2/\Lambda\ge\frac{225}{4096}w_{k+1}^2/(2\ell_{k+1})\ge\ell_0^2\ell_{k+1}/37$, and the prefactor
is at most $C\eta_0(w_k^d+\ell_{k+1}^{d/2})\le C\ell_0^{d}\ell_k^{d}$. The resulting error is at
most $C\ell_0^d\ell_k^d\exp\{-\ell_0^2\ell_{k+1}/(128d)\}$, which is accounted for by the second
term of $\epsilon_k$. Note that the box separation alone would give the constant $64$ in place
of $128$; the difference is the cost of the $6r$ collar.

We now consider the vertically separated pairs. We apply \Cref{prop:vertical} with
$\Delta=\ell_{k+1}/4=2\ell_k$, $\Lambda=2\ell_{k+1}$, $W=6w_{k+1}$, $\mathsf G=w_{k+1}$ and mixing
parameter $\epsilon^{(k)}=\delta_k/\eta_{k+1}$, so that $\eta_{k+1}(1+\epsilon^{(k)})=\eta_k$
and $\epsilon^{(k)}\asymp a^{-ks}\le1$. We take
\[
  \ell^{(k)}=\big(c_0^{-1}\Delta(\epsilon^{(k)}/3)^{4/\kappa}\big)^{1/2}
  \ \asymp\ \ell_0^{1/2}a^{k(1/2-2s/\kappa)},
\]
which is increasing in $k$ because $s<\kappa/4$. Then
$\eta_{k+1}(\ell^{(k)})^d\ge c\,\ell_0^{d/2-s}a^{k(d/2-2sd/\kappa)}\ge1$ for $\ell_0$ large,
which is the second condition in \eqref{eq:ellcond}. Moreover $\mathsf G=w_{k+1}=a\ell_0\ell_k$
exceeds $c_1\Delta^{1/2}(\epsilon^{(k)}/3)^{-1/d}\asymp\ell_k^{1/2}a^{ks/d}$ because $s<d/2$,
and it exceeds $12r$ because $2r\le w_0/8\le w_{k+1}/8$. The choice $W=6w_{k+1}$ contains every
scale-$k$ box involved. Indeed, their centres lie in $[-2w_{k+1},2w_{k+1}]^d$ and each box,
with its collar, extends a further $w_k+2r$, so all of them lie in
$Q_W=[-3w_{k+1},3w_{k+1}]^d$. Using
$\eta_{k+1}(\epsilon^{(k)})^2=\delta_k^2/\eta_{k+1}\ge c\,\ell_0^{s}\ell_k^{-2s}$ we get
\[
  \eta_{k+1}(\epsilon^{(k)})^2\Delta^{d/2}\ \ge\ c\,\ell_0^{d/2-s}a^{k(d/2-2s)},
  \qquad
  \eta_{k+1}(\epsilon^{(k)})^2(\ell^{(k)})^d\ \ge\ c\,\ell_0^{d/2-s}a^{k\theta_\ast},
\]
and $\theta_\ast>0$ holds if and only if $s<s_\ast$. Since $\theta_\ast\le d/2-2s$, the second
of these exponents is the smaller, and it therefore dominates the two mixing errors. This is
the first term of $\epsilon_k$. The remaining term of \eqref{eq:vertical} has the Gaussian
exponent $\mathsf G^2/(8d\Lambda)=w_{k+1}^2/(16d\ell_{k+1})=\ell_0^2\ell_{k+1}/(16d)$, and so it
is again accounted for by the second term of $\epsilon_k$.

It remains to bound the prefactors. We show that all the prefactors appearing are at most
$C\ell_0^{\mathsf N}\ell_k^{\mathsf N}$ with $\mathsf N=(d+2)^2$. First, $K\asymp w_{k+1}=a\ell_0\ell_k$ and
$\ell^{(k)}\ge1$ bound $(K/\ell^{(k)})^d$, $\eta_{k+1}(K+\Lambda^{1/2})^d$ and the prefactor of
\Cref{lem:horiz} by $C\ell_0^{2d}\ell_k^{2d}$. Second, since
$K\Delta^{-1/2}\asymp\ell_0\ell_k^{1/2}$ and $(\epsilon^{(k)})^{-1}\asymp a^{ks}\le\ell_k^s$,
the factor $\Gamma$ of \Cref{prop:vertical} is at most
$C\ell_0^{d(d+1)}\ell_k^{d(d+1)/2+sd}$. Both pairs of exponents are at most $\mathsf N$, because
$s<1$. This gives \eqref{eq:recursion}.
\end{proof}

\begin{lemma}[Contraction]
\label{lem:contraction}
If $p_0\le1/(4C_{\mathrm{ent}})$ and $C_{\mathrm{ent}}^2\epsilon_k\le2^{-(k+4)}$ for all
$k\ge0$, then $p_k\le2^{-(k+2)}/C_{\mathrm{ent}}$ for all $k\ge0$.
\end{lemma}

\begin{proof}
Put $q_k=C_{\mathrm{ent}}p_k$. Multiplying \eqref{eq:recursion} by $C_{\mathrm{ent}}$ gives
$q_{k+1}\le q_k^2+C_{\mathrm{ent}}^2\epsilon_k$, and $q_0\le2^{-2}$ by hypothesis. If
$q_k\le2^{-(k+2)}$, then $q_k^2\le2^{-(2k+4)}\le2^{-(k+4)}$, so
$q_{k+1}\le2^{-(k+4)}+2^{-(k+4)}=2^{-(k+3)}$. The claim follows by induction.
\end{proof}

\subsection{The triggering estimate}
\label{sec:trigger}

It remains to show that $p_0\to0$ as $\ell_0\to\infty$, with $\eta_0\asymp\ell_0^{-s}$. On the
lattice the temporal half-crossing is excluded by a cluster-and-healing event. The spatial
half-crossings are excluded separately, by counting self-avoiding lattice paths and using that
the infection advances only at jump times. Neither argument has a continuum counterpart. We
therefore run the cluster argument on short space-time slabs. This excludes both kinds of
half-crossing at once and requires no bound on the speed of the infection front.

\begin{lemma}[Local extinction in an isolated cluster]
\label{lem:localext}
Fix $m\in\N$, $r>0$, $\mu\in(0,\infty]$ and $d\ge1$. There is
$p=p(m,r,\mu,d)>0$ such that, for any $k\le m$ particles performing independent Brownian
motions with no other particle present and all infected at time $0$,
\[
  \inf_{x\in(\R^d)^k}\Prob_x\big(\text{every particle is susceptible at time }1\big)\ \ge\ p ,
\]
and consequently
$\Prob_x(\text{some particle is infected at time }t)\le(1-p)^{\lfloor t\rfloor}$ for $t\ge1$.
\end{lemma}

\begin{proof}
Put $M=32rm$. We choose $y_i\in B(x_i,M/4)$ with $|y_i-y_j|>4r$ for $i\ne j$. This is possible
greedily, since the volume forbidden to $y_i$ by the previously chosen points is at most
$m\omega_d(4r)^d$, which is smaller than $|B(x_i,M/4)|=\omega_d(8rm)^d$ because
$8^dm^d>m4^d$.

Let $\mathcal A$ be the event that each particle follows a tube of radius $r/8$ from $x_i$ to
$y_i$ during $[0,1/2]$ and remains in $B(y_i,r/8)$ during $[1/2,1]$. The tubes have length at
most $M/4$, which depends only on $m$ and $r$, so $\Prob_x(\mathcal A)\ge p_1(m,r,d)>0$
uniformly in $x$. On $\mathcal A$, any two particles are at distance more than $4r-r/4>2r$
throughout $[1/2,1]$, so all $k$ particles are isolated in that interval. Each of them
therefore carries an effective recovery mark in $[1/2,1]$ with probability
$1-e^{-\mu/2}$, independently of the others; this probability is $1$ when $\mu=\infty$. Once a
particle has recovered it cannot be reinfected while it is isolated. Taking
$p=p_1(1-e^{-\mu/2})^m$ gives the first assertion. The second assertion follows from the first
by the Markov property together with \Cref{prop:mono}(i), which allows us to restart at each
integer time from the all-infected configuration.
\end{proof}

We fix $b\in(0,1)$ with $b\le s/d$, and we set
\begin{equation}
\label{eq:slabparams}
  \mathsf T=\ell_0^b,\qquad h_0=C_\star\log\ell_0,\qquad
  \mathsf d_0=\ell_0^{b/2}\log\ell_0,\qquad \mathsf R=2r+2\mathsf d_0 ,
\end{equation}
where $C_\star=C_\star(p,d)$ is chosen in \Cref{lem:trigger} below and $\ell_0$ is large enough
that $h_0+1\le \mathsf T$. We put
\[
  Q_0=[-w_0-2r,\ w_0+2r]^d,\qquad \Qin=[-2w_0,2w_0]^d,\qquad \Qext=[-4w_0,4w_0]^d .
\]
Note that throughout \Cref{sec:extinction} a subscripted $Q$ denotes the cube of the indicated
half-width, and not, as in \Cref{sec:particles,sec:mixing,sec:lipschitz}, the cube of the
indicated side. This convention is local to \Cref{sec:extinction}. We adopt it because every
region here is described by its distance from the origin. The cubes $Q_K,Q_{K'},Q_W$ of
\Cref{prop:vertical} are the exception; they follow the side convention of \Cref{sec:mixing},
from which they are quoted.
We call \emph{$\mathsf R$-clusters} the connected components of the graph on the particles which
joins two of them when their distance is at most $\mathsf R$. For $0\le j<\ell_0^{1-b}$ let
$\mathcal G(j)$ be the intersection of the following four events, all referring to the slab
$[j\mathsf T,(j+1)\mathsf T+h_0+1]$.
\begin{enumerate}[(G1)]
\item Every $\mathsf R$-cluster of the time-$j\mathsf T$ configuration which meets $\Qext$ has fewer than $m$
  particles.
\item Every particle lying in $\Qext$ at time $j\mathsf T$ displaces by at most $\mathsf d_0$ during
  $[j\mathsf T,(j+1)\mathsf T+h_0+1]$.
\item No particle lying outside $\Qext$ at time $j\mathsf T$ enters $\Qin$ during
  $[j\mathsf T,(j+1)\mathsf T+h_0+1]$.
\item For every $\mathsf R$-cluster $\mathcal K$ of the time-$j\mathsf T$ configuration which meets
  $\Qin$, and every integer $i\in[j\mathsf T,(j+1)\mathsf T+1]$, the isolated system consisting of the
  particles of $\mathcal K$ alone, all infected at time $i$, has no infected particle at time
  $i+h_0$.
\end{enumerate}

The geometry of one slab is drawn in \Cref{fig:slab}.

\begin{figure}[!h]
\centering
\begin{tikzpicture}[x=0.85cm,y=0.8cm,>=Stealth]
  \draw[dotted] (-5,0) -- (5,0); \draw[dotted] (-5,4.4) -- (5,4.4);
  \node[left] at (-5,0) {\small$j\mathsf T$};
  \node[left] at (-5,3.3) {\small$(j{+}1)\mathsf T$};
  \node[left] at (-5,4.4) {\small$(j{+}1)\mathsf T{+}h_0{+}1$};
  \draw[dotted] (-5,3.3) -- (5,3.3);
  \draw[|-|] (-4.6,-0.7) -- (4.6,-0.7) node[right]{\small$\Qext$};
  \draw[|-|] (-2.3,-1.4) -- (2.3,-1.4) node[right]{\small$\Qin$};
  \draw[|-|] (-1.1,-2.1) -- (1.1,-2.1) node[right]{\small$Q_0$};
  \fill[black!8] (-0.95,0) rectangle (0.95,4.4);
  \draw[gray] plot[smooth] coordinates {(-0.6,0) (-0.3,1.4) (-0.7,2.8) (-0.4,4.4)};
  \draw[gray] plot[smooth] coordinates {(0.2,0) (0.6,1.4) (0.1,2.8) (0.5,4.4)};
  \draw[gray] plot[smooth] coordinates {(0.8,0) (0.4,1.4) (0.85,2.8) (0.6,4.4)};
  \draw[very thick] plot[smooth] coordinates {(-0.6,0.8) (-0.3,1.4)};
  \draw[->,very thick] (-0.3,1.4) -- (0.6,1.4);
  \draw[very thick] plot[smooth] coordinates {(0.6,1.4) (0.35,2.1)};
  \fill (0.35,2.1) circle (1.6pt);
  \node[right] at (1.15,2.1) {\scriptsize dies by (G4)};
  \draw[<->] (-0.95,5.0) -- (0.95,5.0) node[midway,above]{\small$2m\mathsf R$};
  \draw[gray,->] (-4.9,0.4) .. controls (-4.0,1.6) .. (-3.2,3.0);
  \node[gray,right] at (-3.2,3.0) {\scriptsize (G3)};
\end{tikzpicture}
\caption{One slab of the triggering estimate, with time running upwards. On $\mathcal G(j)$ every
$\mathsf R$-cluster meeting $\Qext$ has fewer than $m$ particles by (G1) and stays inside a ball of
radius $m\mathsf R+\mathsf d_0$ by (G2), so distinct clusters never come into contact. By (G3) no
particle from outside $\Qext$ enters $\Qin$. The cluster therefore evolves in isolation, and
(G4) ends every open path started in it within time $h_0$. Since $2m\mathsf R=o(w_0)$ and
$h_0+1\ll\ell_0/2$, neither a spatial nor a temporal half-crossing of $B_0$ can occur.}
\label{fig:slab}
\end{figure}
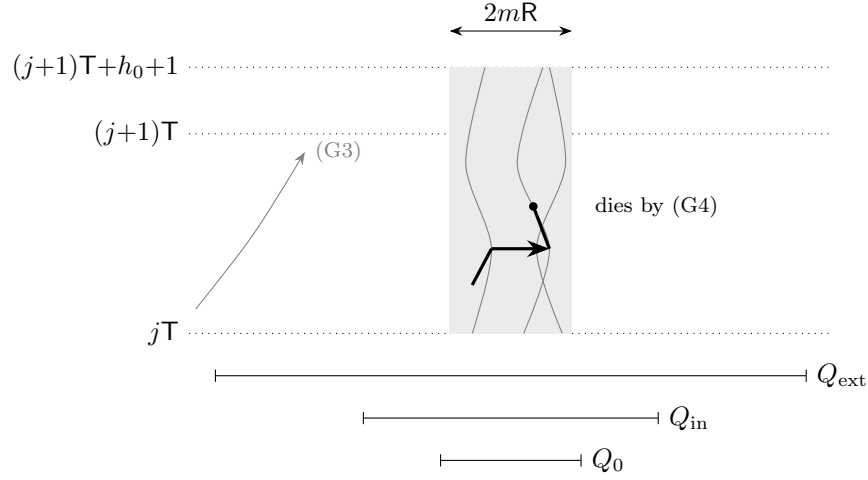

Note that we impose clause (G4) at every integer time of the slab and not only at its
beginning. The reason is the following. By attractiveness, an open path started at a time
$\sigma$ is dominated by the all-infected process started at $\sigma$, and not by the one
started at $j\mathsf T$, since the starting particle may already have recovered and been reinfected
in between. The cost is a union bound over $\mathsf T+2$ times. The gain is that the bound on
the duration of an open path improves from order $\mathsf T$ to order $\log\ell_0$.

\begin{lemma}[Confinement and local death]
\label{lem:slab}
Let $\ell_0$ be large enough that $m\mathsf R+2\mathsf d_0+4r<w_0$. On $\mathcal G(j)$, every open path whose trace
starts inside $Q_0$ at a time $\sigma\in[j\mathsf T,(j+1)\mathsf T]$ uses only particles of a single
$\mathsf R$-cluster of the time-$j\mathsf T$ configuration, has trace contained in a ball of radius $2m\mathsf R$,
and has duration less than $h_0+1$.
\end{lemma}

\begin{proof}
We first show that the starting particle lies in $\Qext$ at time $j\mathsf T$, so that (G2)
applies to it below. The starting particle lies in $Q_0\subseteq\Qin$ at the time $\sigma$ of
the slab, and $\Qin$ and $\Qext$ are separated by $2w_0$. A particle outside $\Qext$ at time
$j\mathsf T$ would have to displace by at least $2w_0$ to enter $\Qin$ during the slab, which (G3)
forbids.

By (G2), two particles of $\Qext$ lying in distinct $\mathsf R$-clusters at time $j\mathsf T$ start at
distance more than $\mathsf R=2r+2\mathsf d_0$ from each other and each moves by at most $\mathsf d_0$
during the slab, so they are never in contact during the slab. By (G1) the cluster $\mathcal K$
containing the starting particle has fewer than $m$ particles, hence diameter at most
$(m-1)\mathsf R$ at time $j\mathsf T$. By (G2) the starting particle lies within $\mathsf d_0$ of $Q_0$ at
time $j\mathsf T$, so every particle of $\mathcal K$ lies within
$w_0+2r+\mathsf d_0+(m-1)\mathsf R<2w_0$ of the origin in every coordinate. Hence
$\mathcal K\subseteq\Qin$ at time $j\mathsf T$, and (G2) applies to every particle of $\mathcal K$. It
follows that $\mathcal K$ remains throughout the slab inside the ball of radius
$m\mathsf R+\mathsf d_0$ centred at the position of the starting particle at time $j\mathsf T$.
Since $m\mathsf R+2\mathsf d_0+4r<w_0$ and $Q_0$ has half-width $w_0+2r$,
that ball together with its $2r$-neighbourhood has half-width at most
$w_0+2r+\mathsf d_0+m\mathsf R+\mathsf d_0+2r<2w_0$, and so it is contained in $\Qin$. Hence by (G3) no
particle from outside $\Qext$ ever comes into contact with $\mathcal K$. This gives that
$\mathcal K$ evolves in isolation during the slab, which proves the first two assertions.

It remains to bound the duration. Suppose that the path has duration at least $h_0+1$, and put
$i=\lceil\sigma\rceil$, an integer with $j\mathsf T\le i\le\sigma+1\le(j+1)\mathsf T+1$. Note that
$\mathsf T$ need not be an integer, which is why we imposed (G4) up to $(j+1)\mathsf T+1$. We now check
that clause (G4) applies to $\mathcal K$, that is, that $\mathcal K$ meets $\Qin$ at time
$j\mathsf T$. The starting particle lies in $Q_0$ at time $\sigma$ and displaces by at most
$\mathsf d_0$ during the slab by (G2), so at time $j\mathsf T$ it lies within $\mathsf d_0$ of $Q_0$.
Moreover $Q_0$ has half-width $w_0+2r$ while $\Qin$ has half-width $2w_0$. This boundary layer
is the reason why we impose (G4) on the clusters meeting $\Qin$ rather than on those meeting
$Q_0$. The path passes through a particle at time $i$, and its restriction to the time interval
starting at $i$ is again an open path by \Cref{def:openpath}(ii). Hence, in the isolated system
on $\mathcal K$ started all infected at time $i$, the endpoint of the path is infected at a time
at least $i+h_0$. This contradicts (G4).
\end{proof}

\begin{lemma}[Base-scale estimate]
\label{lem:trigger}
Let $s\in(0,s_\ast)$, let $b\le s/d$ and $m>(6d+6)/s$, and let $C_\star$ be large enough that
$(1-p)^{\lfloor C_\star\log\ell_0\rfloor}\le\ell_0^{-(2d+3)}$, with $p$ as in \Cref{lem:localext}.
Then $\Prob_{\eta_0}\big(H(B_0)\big)\to0$ as $\ell_0\to\infty$.
\end{lemma}

\begin{proof}
We first show that on $\bigcap_j\mathcal G(j)$ the event $H(B_0)$ fails. A witness for $T(B_0)$
is an open path contained in $B_0$ of duration at least $\ell_0/2$. Its trace starts inside
$[-w_0,w_0]^d\subseteq Q_0$ at some time $\sigma\in[0,\ell_0]$, so \Cref{lem:slab}, applied in
the slab containing $\sigma$, bounds its duration by $h_0+1=O(\log\ell_0)$, which is smaller
than $\ell_0/2$. A witness for $S^\pm_i(B_0)$ is an open path contained in $B_0^+$ whose trace
starts inside $Q_0$ and travels a distance at least $(1-\vartheta)w_0=\tfrac78\ell_0^2$ in the
$i$-th coordinate. By \Cref{lem:slab}, however, its trace stays in a ball of radius
$2m\mathsf R=O(\ell_0^{b/2}\log\ell_0)=o(w_0)$.

We now bound the four failure probabilities. The configuration at time $j\mathsf T$ is again Poisson
of intensity $\eta_0$, so none of the four probabilities depends on $j$.

We start with (G1). We cover $\Qext$ by balls of radius $2\sqrt d\,m\mathsf R$ centred on the points
of the lattice $m\mathsf R\Z^d$ which lie within distance $m\mathsf R$ of $\Qext$. Their number is at most
$C_d(w_0/(m\mathsf R))^d\le C_d\ell_0^{2d}$. Any set $S$ of diameter at most $(m-1)\mathsf R$ meeting
$\Qext$ lies inside one of them. Indeed, choosing $y\in S\cap\Qext$ and a lattice point $z$ with
$|y-z|_\infty\le m\mathsf R/2$, every point of $S$ is within
$\sqrt d\,m\mathsf R/2+(m-1)\mathsf R\le2\sqrt d\,m\mathsf R$ of $z$. The number of particles in such a ball is
Poisson with mean at most $C_dm^d\eta_0\mathsf R^d$, and since $bd/2\le s/2$,
\[
  \eta_0\mathsf R^d\ \le\ C\ell_0^{-s}\big(3\ell_0^{b/2}\log\ell_0\big)^d
  \ =\ C'\ell_0^{bd/2-s}(\log\ell_0)^d\ \le\ \ell_0^{-s/3}
\]
for $\ell_0$ large. Hence
$\Prob\big((\mathrm{G1})^c\big)\le C(m,d)\ell_0^{2d-ms/3}\le C(m,d)\ell_0^{-2}$ by the choice of
$m$.

Next we consider (G2). The particles lying in $\Qext$ at time $j\mathsf T$ form a Poisson process of
intensity $\eta_0$ on $\Qext$, and conditionally on their positions their subsequent
displacements are independent Brownian motions. We write $\tau=\mathsf T+h_0+1$ for the length of
the slab. By Markov's inequality and the reflection bound,
\[
  \Prob\big((\mathrm{G2})^c\big)\ \le\ \eta_0|\Qext|\;
  \Prob\Big(\sup_{s\le\tau}|B_s|>\mathsf d_0\Big)
  \ \le\ 8^d\ell_0^{2d-s}\cdot4d\,\exp\Big\{-\frac{\mathsf d_0^2}{2d\tau}\Big\} .
\]
The exponent has $\tau$ in its denominator, so the right-hand side is increasing in $\tau$.
Since $h_0+1\le \mathsf T$ we have $\tau\le2\mathsf T$, and replacing $\tau$ by $2\mathsf T$ therefore only
increases the right-hand side. As $\mathsf d_0^2=\ell_0^b(\log\ell_0)^2$ and $\mathsf T=\ell_0^b$,
\[
  \frac{\mathsf d_0^2}{2d\cdot2\mathsf T}\ =\ \frac{\ell_0^b(\log\ell_0)^2}{4d\,\ell_0^b}
  \ =\ \frac{(\log\ell_0)^2}{4d} ,
\]
and using $\ell_0^{-s}\le1$ we obtain
$\Prob((\mathrm{G2})^c)\le C_d\,\ell_0^{2d}e^{-(\log\ell_0)^2/(4d)}
=C_d\,\ell_0^{2d-\log\ell_0/(4d)}$, which decays faster than any power of $\ell_0$. Note that
both replacements, that of $\tau$ by $2\mathsf T$ and that of $\ell_0^{-s}$ by $1$, weaken the
bound.

We now consider (G3). A particle outside $\Qext$ is at distance at least $2w_0$ from $\Qin$, so
\[
  \Prob\big((\mathrm{G3})^c\big)\ \le\ C\,\eta_0\,w_0^{d-1}\sqrt{\mathsf T}\,e^{-w_0^2/(d\mathsf T)},
  \qquad\text{and}\qquad w_0^2/\mathsf T=\ell_0^{4-b} .
\]

Finally we consider (G4). The number of $\mathsf R$-clusters meeting $\Qin$ is at most the number
of particles in $\Qin$, whose expectation is at most $C\ell_0^{2d}$, and there are at most
$\mathsf T+2\le2\ell_0^b$ integer times $i\in[j\mathsf T,(j+1)\mathsf T+1]$. On (G1) every such cluster has
fewer than $m$ particles, so by \Cref{lem:localext} and the choice of $C_\star$,
\[
  \Prob\big((\mathrm{G4})^c\cap(\mathrm{G1})\big)
  \ \le\ C\ell_0^{2d}\cdot2\ell_0^b\cdot\ell_0^{-(2d+3)}\ \le\ C\ell_0^{-2} .
\]

We sum the four bounds and multiply by the number $\ell_0^{1-b}\le\ell_0$ of slabs. This gives
\[
  \Prob_{\eta_0}(H(B_0))\ \le\ C(m,d)\ell_0^{-1}+o(\ell_0^{-1})\ \longrightarrow\ 0 . \qedhere
\]
\end{proof}

\subsection{Proof of \texorpdfstring{\Cref{thm:extinction}}{Theorem 7.1}}
\label{sec:extproof}

\begin{proof}[Proof of \Cref{thm:extinction}]
We fix $s\in(0,s_\ast)$, and then $b$, $m$ and $C_\star$ as in \Cref{lem:trigger}. By that
lemma, we can choose $\ell_0$ so large that $p_0\le1/(4C_{\mathrm{ent}})$ and that
$2r\le w_0/8$, as required in \Cref{lem:cascade}. By \Cref{prop:recursion} the errors
$\epsilon_k$ decay superexponentially in $k$. Hence after enlarging $\ell_0$ further we have
$C_{\mathrm{ent}}^2\epsilon_k\le2^{-(k+4)}$ for every $k$, and \Cref{lem:contraction} gives
$p_k\to0$. We put $\bar\eta=\eta_\infty=\ell_0^{-s}$ and let $\eta<\bar\eta$.

Recall that $\eta\le\eta_k$ for every $k$ and that $H(B_k)$ is increasing by \Cref{lem:Hprops}.
Realising $\mathcal P^{\eta}\subseteq\mathcal P^{\eta_k}$ on one space by superposition, as in
the proof of \Cref{prop:mono}(iii), we obtain
\begin{equation}
\label{eq:statHk}
  \Prob_\eta\big(H(B_k)\big)\ \le\ \Prob_{\eta_k}\big(H(B_k)\big)\ =\ p_k\ \longrightarrow\ 0 .
\end{equation}
This is a statement about the stationary system, in which no particle is distinguished. The
theorem concerns the infection started from a single particle at the origin, that is the system
under the Palm measure. The two are related by the Mecke equation and not by an inequality.
Adding the Palm particle can only increase the probability of the increasing event $H(B_k)$, so
we cannot apply \eqref{eq:statHk} to it directly. We argue as follows.

We fix a bounded $Q\subseteq\R^d$ with $|Q|>0$, and we let $N_Q=\#\{u\in\mathcal P:X_u(0)\in Q\}$.
This is a Poisson variable of mean $\eta|Q|$, so that $\E[N_Q^2]<\infty$. We take $k$ large
enough that $Q\subseteq[-\vartheta w_k,\vartheta w_k]^d$, which is possible since
$w_k\to\infty$. If some $u\in\mathcal P$ with $X_u(0)\in Q$ has an infection which survives,
then in particular it survives past time $\ell_k/2$, so $H(B_k)$ occurs by \Cref{lem:surv2H}.
Hence, by Cauchy--Schwarz and \eqref{eq:statHk},
\begin{align*}
  \E\Big[\#\big\{u\in\mathcal P:\ X_u(0)\in Q,\ \text{the infection started from }u
  \text{ survives}\big\}\Big]
  &\ \le\ \E\big[N_Q\mathbf 1_{H(B_k)}\big]\\
  &\ \le\ \E\big[N_Q^2\big]^{1/2}p_k^{1/2} .
\end{align*}
The left-hand side does not depend on $k$ and the right-hand side tends to $0$, so the left-hand
side vanishes.

It remains to identify the left-hand side. We apply the Mecke equation to $\mathcal P$ with
$f(w,\mathcal P)=\mathbf 1\{w(0)\in Q\}\mathbf 1\{\text{the infection started from }w\text{ in }
\mathcal P\cup\{w\}\text{ survives}\}$, and we integrate out the marks. This gives that the
left-hand side equals
\[
  \int_{\{w(0)\in Q\}}\Prob\big(\text{the infection started from }w\text{ in }
  \mathcal P\cup\{w\}\text{ survives}\big)\,\nu_\eta(\mathrm dw)
  \ =\ \eta\int_Q\Prob^{x}_\eta(\text{survival})\,\mathrm dx ,
\]
where $\Prob^x_\eta$ is the law of the system with one extra particle started at $x$, that is, by
\eqref{eq:pathPPP}, the Palm measure of $\mathcal P$ at $x$. By translation invariance
$\Prob^x_\eta(\text{survival})$ does not depend on $x$ and equals the survival probability of the
theorem. Since $\eta|Q|>0$, that probability is $0$.
\end{proof}
\section{Proofs of the main theorems}
\label{sec:mainproof}

We prove \Cref{thm:motion} first, since it is assembled from the other two theorems. We then prove
the two inputs.

\begin{proof}[Proof of \Cref{thm:motion}]
Let $d\ge2$, $r>0$ and $\eta<\etaB$.

We begin with the medium. Since $\eta<\etaB$, the Gilbert graph $G_t$ has almost surely no infinite
component for each fixed $t$ \cite[Ch.~3]{MeesterRoy96}. We work under the Palm measure at the
origin. By Slivnyak's theorem this is the law of $\mathcal P$ with one further particle placed at
the origin. The probability that the cluster of that particle is infinite is $\theta(\eta)$, which
vanishes for $\eta<\etaB$. Hence the cluster of the seed is almost surely finite, and
\Cref{prop:soft} is not available at these intensities. In other words, the survival we prove below
is not of the soft kind.

We next prove survival and positive speed. Fix $\lambda\in(0,\infty)$ and let
$\mu_*=\mu_*(\lambda;\eta,r,d)>0$ be the threshold of \Cref{thm:smallmu} at this intensity. Let
$\mu\le\mu_*$. That theorem gives, for the finite-rate model,
\[
  \Prob^{\lambda,\mu}_\eta(\text{survival})>0
  \qquad\text{and}\qquad
  \Prob^{\lambda,\mu}_\eta\Big(\liminf_{t\to\infty}\frac{\rad(I_t)}{t}>0\Big)>0 .
\]
For the instantaneous model, we realise the two processes on one marked configuration, with the same
$\mathcal P$, the same recovery marks and the same initial infected particle. This is possible
because $\mu\le\mu_*<\infty$, which is the hypothesis of \Cref{prop:domination}. That proposition gives
$I^{\lambda,\mu,\unc}_t\subseteq I^{\infty,\mu,\iso}_t$ for every $t\ge0$ almost surely. Hence
$\rad(I^{\lambda,\mu,\unc}_t)\le\rad(I^{\infty,\mu,\iso}_t)$ for every $t$ by \eqref{eq:rad}. The
event that the finite-rate radius grows linearly is therefore contained in the event that the
instantaneous one does, and the same holds for survival. This gives that both displays hold with
$\Prob^{\infty,\mu}_\eta$ in place of $\Prob^{\lambda,\mu}_\eta$.

Finally, we consider the critical density. Put $J=\{\mu\in(0,\infty):\eta_c(\infty,\mu)<\etaB\}$.
Since $\eta_c(\infty,\cdot)$ is non-decreasing by \Cref{prop:mono}(iii), $J$ is a down-set. Since
$\eta_c(\infty,\mu)\le\etaB$ for every $\mu$ by \Cref{thm:density}(iii), the complement of $J$ in
$(0,\infty)$ is the set on which equality holds. Hence $\mu_\dagger:=\sup J$ has the two stated
properties, and it remains to show that $J\ne\emptyset$. We apply the previous paragraph at
$\eta=\etaB/2$ with any $\lambda\in(0,\infty)$. Note that $\etaB/2$ is positive in every dimension
and finite because $d\ge2$ (\Cref{sec:particles}). For every $\mu\le\mu_*(\lambda;\etaB/2)$ the
finite-rate infection at that intensity survives with positive probability, so
$\eta_c(\lambda,\mu)\le\etaB/2$ by \eqref{eq:etac}. Hence
$\eta_c(\infty,\mu)\le\eta_c(\lambda,\mu)\le\etaB/2<\etaB$ by \eqref{eq:etacinf}. Thus
$(0,\mu_*(\lambda;\etaB/2)]\subseteq J$ and $\mu_\dagger>0$.
\end{proof}

\begin{proof}[Proof of \Cref{thm:density}]
By \Cref{prop:mono}(iii) the survival probability is non-decreasing in $\eta$ at every fixed
$(\lambda,\mu)$. Hence $\{\eta>0:\Prob^{\lambda,\mu}_\eta(\text{survival})>0\}$ is an up-set, and
the number $\eta_c(\lambda,\mu)$ of \eqref{eq:etac} separates almost sure extinction from survival
with positive probability.

We first prove (i). Let $\bar\eta=\bar\eta(\mu,r,d)>0$ be the threshold of \Cref{thm:extinction},
and let $\eta<\bar\eta$. At $\lambda=\infty$, \Cref{thm:extinction} states directly that the
infection started from a single particle dies out almost surely, so $\eta_c(\infty,\mu)\ge\bar\eta$.
Let now $\lambda\in(0,\infty)$. Then $\mu<\infty$, since the combination $(\lambda<\infty,\unc)$
admits no $\mu=\infty$. As in the proof of \Cref{thm:motion} we realise the finite-rate and the
instantaneous model on one marked configuration, which \Cref{prop:domination} permits since
$\mu<\infty$. That proposition gives
\[
  I^{\lambda,\mu,\unc}_t\ \subseteq\ I^{\infty,\mu,\iso}_t\qquad\text{for every }t\ge0 ,
\]
almost surely. By \Cref{thm:extinction} the right-hand side is empty from some finite time
onwards, almost surely. Hence so is the left-hand side, and
$\Prob^{\lambda,\mu}_\eta(\text{survival})=0$. As $\eta<\bar\eta$ was arbitrary, we obtain
$\eta_c(\lambda,\mu)\ge\bar\eta$ for every $\lambda\in(0,\infty]$, with $\bar\eta$ depending only
on $\mu,r,d$.

Part (ii) is \Cref{prop:crowding}, which is valid for every $d\ge1$ and every $\mu\in(0,\infty]$.

We now prove (iii). \Cref{prop:supercrit} gives $\eta_c(\infty,\mu)\le\etaB$ for $d\ge2$ and every
$\mu\in(0,\infty]$, in particular at $\mu=\infty$. The inequality
$\eta_c(\infty,\mu)\le\eta_c(\infty,\infty)$ is the monotonicity of $\eta_c$ in $\mu$ from
\Cref{prop:mono}(iii).

The final assertion combines (i) with (ii).
\end{proof}

\begin{proof}[Proof of \Cref{thm:smallmu}]
The survival and positive-speed statement is proved in \Cref{sec:smallmu}, and the argument is
assembled in \Cref{sec:goodcell}. It gives $\Prob^{\lambda,\mu}_\eta(\text{survival})>0$ for every
$\mu\le\mu_*(\lambda;\eta,r,d)$. By \eqref{eq:muc} this is $\mu_c(\lambda,\eta)\ge\mu_*(\lambda)>0$.

For the last assertion, fix $\lambda\in(0,\infty)$ and let $\eta'>0$. We apply the above at
intensity $\eta'$. This produces $\mu_*(\lambda;\eta',r,d)>0$ such that the infection at
intensity $\eta'$ survives with positive probability whenever
$\mu\le\mu_*(\lambda;\eta',r,d)$. For such $\mu$ we therefore have
$\eta_c(\lambda,\mu)\le\eta'$ by \eqref{eq:etac}. Hence
$\limsup_{\mu\downarrow0}\eta_c(\lambda,\mu)\le\eta'$, and $\eta'>0$ was arbitrary. The
case $\lambda=\infty$ follows from $\eta_c(\infty,\mu)\le\eta_c(\lambda,\mu)$, which is
\eqref{eq:etacinf}.
\end{proof}
\section{Remarks and open problems}
\label{sec:remarks}

\subsection{The absence of a reproduction number}
\label{sec:noR0}

Consider a contact process in which recovery is unconditional and the medium has bounded degree.
There the infectious period of a particle is an $\Exp(\mu)$ variable independent of everything
else. One bounds the expected number of transmissions it makes by a reproduction number $R_0$, and
$R_0<1$ gives almost sure extinction. No such argument is available in the instantaneous model, for
two independent reasons.

First, an infected particle lives until its accumulated isolated time reaches an
$\Exp(\mu)$ threshold. Recall that $\Vol=|B(0,2r)|$. By ergodicity the isolated fraction of time
is $e^{-N}$ with $N=\eta\Vol$, so the mean lifetime is of order $e^{N}/\mu$. This is not a rate but a quantity already governed
by the geometry.

Second, the offspring number and the lifetime are positively correlated. By
the Mecke equation the expected offspring number of a tagged infected particle at $x$ is
\[
  \eta\int_{\R^d}\Prob^{x,y}\big(\text{the }y\text{-particle meets the }x\text{-particle before
  its death in }\omega\cup\delta_y\big)\,\mathrm dy ,
\]
in which the death time is computed in the environment that includes the candidate offspring.
Since adding a particle can only lengthen a lifetime, the product bound goes the wrong way.
\Cref{sec:extinction} avoids this difficulty altogether, since \Cref{lem:localext} bounds the
extinction time of a bounded isolated cluster with no reproduction estimate at all.

There is also a ceiling which no criterion of this type can pass. A first-moment criterion counts
the particles that a single infected particle can reach, and that count is at least the mean
degree $\eta\Vol$. Hence any criterion of that shape is confined to $\eta\Vol<1$. Subcriticality of
the static Boolean model is strictly weaker. The critical mean degree $\etaB\Vol$ equals
$4\log\big(1/(1-\phi_c)\big)\approx4.512$ in $d=2$, from the critical area fraction
$\phi_c=0.676339$ for fully penetrable discs \cite{QTZ00}. It is about $2.7$ in $d=3$, and it
exceeds $1$ in every finite dimension, tending to $1$ only as $d\to\infty$ \cite{Penrose96}. Since
\Cref{thm:density} locates $\eta_c(\infty,\mu)$ in $[\bar\eta,\etaB]$, the interval in which the
transition is known to lie extends well above mean degree one. In other words, an argument of this
type can bound $\eta_c$ from below at very small intensity, but it cannot determine $\eta_c$. Both
this ceiling and the correlation above come from the same unboundedness of the contact degree.

\subsection{Reduction to \texorpdfstring{$\mu=\infty$}{mu = infinity}}
\label{sec:muinfty}

By \Cref{prop:mono}(iii) we have $\eta_c(\infty,\mu)\le\eta_c(\infty,\infty)$ for every $\mu$. Hence
a proof that $\eta_c(\infty,\infty)<\etaB$ would give, in one step, survival for every
$\mu\in(0,\infty]$ at every intensity in $(\eta_c(\infty,\infty),\etaB)$, that is, strictly below
the static percolation threshold. By \Cref{sec:corners} the instantaneous model at $\mu=\infty$
carries no rates at all. The local events driving a Lipschitz surface would therefore be functions
of $\mathcal P$ alone and hence automatically increasing. \Cref{thm:GS} would then apply with no
treatment of marks, and with no need for the auxiliary paths of \Cref{sec:auxpath}.
\Cref{prop:crowding} is a worked instance. Its event $\widehat E$ is purely geometric, and that
section uses none of the devices of \Cref{sec:smallmu}. The missing input is a cell event that
carries the infection below $\etaB$, not a treatment of marks. This is the open question. The
conditioned kernel of \Cref{sec:condkernel} would still be needed, with \Cref{thm:mixinglower}
rather than \Cref{thm:mixingupper} as the relevant input. The reason is that one would need the
mixed configuration to contain a Poisson process rather than to be contained in one.

\subsection{Open problems}
\label{sec:open}

\begin{enumerate}[(1)]
\item How far does the strict inequality $\eta_c(\infty,\mu)<\etaB$ extend? By
  \Cref{thm:motion} it holds for every $\mu$ below a threshold $\mu_\dagger\in(0,\infty]$, and
  $\mu_\dagger>0$. Is $\mu_\dagger=\infty$, that is, does motion strictly help at every finite
  recovery rate? And is $\eta_c(\infty,\infty)<\etaB$, or is it equal to $\etaB$? The second does
  not follow from the first. A supremum of quantities each strictly below $\etaB$ may equal
  $\etaB$, and even $\sup_{\mu<\infty}\eta_c(\infty,\mu)<\etaB$ would still have to be transferred
  to $\mu=\infty$, for which monotonicity gives only the inequality $\le$. This is the question of
  whether motion still helps when healing is instantaneous. As explained above,
  a positive
  answer would immediately give survival strictly below $\etaB$ for every $\mu\in(0,\infty]$ in
  one step.
\item Quantify the two degenerations. How does $\bar\eta(\mu)$, equivalently
  $\eta_c(\lambda,\mu)$, behave as $\mu\downarrow0$? And how does $\mu_*(\lambda)$ behave as
  $\lambda\downarrow0$? \Cref{rem:muStar} gives
  $\mu_*(\lambda)\asymp(\lambda\eta)^{6}(\log(1/\lambda))^{-6}$, and we make no claim that this is
  of the right order. The natural guess is $\mu_c(\lambda,\eta)\asymp\lambda$ as $\lambda\to0$.
  Similarly, is $\sup_{\mu<\infty}\eta_c(\infty,\mu)=\eta_c(\infty,\infty)$, and is $\eta_c$
  continuous in its arguments?
\item Above $\etaB$, does local survival hold, and is $\liminf_{t\to\infty}\iota(t)>0$, where
  $\iota(t)$ is the intensity of $I_t$? \Cref{prop:soft} shows that global survival as defined here
  carries no such information, and \Cref{prop:supercrit} produces no structure along which the
  infection is transported. \Cref{thm:motion} shows that this degeneracy is confined to the region
  above $\etaB$ and that the transition detected by $\eta_c$ is not an artefact of it. The natural
  version of the question is therefore whether \Cref{thm:motion} itself can be strengthened to
  local survival and non-degeneracy at the intensities below $\etaB$ where it applies.
\item Is $\eta_c(\lambda,\mu)<\infty$ at finite $\lambda$ and fixed $\mu$? Both routes to
  \Cref{thm:density}(ii)--(iii) are $\lambda=\infty$ arguments. \Cref{prop:supercrit} needs the whole
  Gilbert cluster to be infected at arbitrarily small times, and \Cref{prop:crowding} needs the
  spread within a crowded cell to be a deterministic consequence of the geometry. Neither
  argument works at finite $\lambda$, and \Cref{prop:domination} runs the wrong way to transport
  them. A natural route to a bound at fixed $\mu$ is a metastability argument. Once the number of
  infected particles in a cell exceeds a positive fraction of the local population, the SIS
  dynamics should keep it there for a time exponential in that population, which is far longer
  than a block. Making this rigorous requires a comparison for the infection restricted to a cell
  which we have not found.
\item Extinction at small $\lambda$ and fixed intensity. \Cref{thm:density}(i) is a
  small-intensity statement uniform in $\lambda$, and says nothing at a fixed intensity above
  $\bar\eta$. One expects extinction there for $\lambda$ small, that is, $\eta_c(\lambda,\mu)>\eta$
  for $\lambda$ small enough. As explained in \Cref{sec:noR0}, a first-moment argument cannot
  give this
  above mean degree one. The question is the finiteness of an annealed Lyapunov exponent for
  branching in a medium of moving catalysts. This is expected to hold for $d\ge3$ and small
  $\lambda$ and to fail for $d\le2$. A proof would presumably have to be a renormalisation, as in
  \cite{KestenSidoravicius06}.
\item Is the exponent $s_\ast$ of \eqref{eq:sstar} an artefact of the method? It arises from the
  two conditions \eqref{eq:ellcond}, which squeeze the mixing scale $\ell$ from both sides. Any
  improvement of \Cref{thm:mixingupper} in the regime of small $\varrho\Delta^{d/2}$ would relax
  it.
\item Extend the results to non-Gaussian mobility, for instance $\alpha$-stable flights, for which
  \Cref{thm:mixinglower} is available in \cite{GJLV26} at general $\alpha$. Gaussianity enters
  quantitatively in three places, and all three would have to be replaced. These are the kernel
  estimates (H\ref{H1})--(H\ref{H4}) of \Cref{sec:HK}; \Cref{lem:condkernel}, whose proof uses the
  product structure of the coordinates and the method of images, neither of which is available for
  a jump process; and \Cref{lem:localext}, which uses Brownian tube estimates.
\end{enumerate}


\begin{thebibliography}{99}

\bibitem{BHO26} R.\ Baldasso, M.\ Hil\'ario, I.\ Ornelas.
\emph{Epidemic phase transitions in the zero-range process.}
Preprint, 2026. \href{https://arxiv.org/abs/2607.22908}{arXiv:2607.22908}.

\bibitem{BaldassoStauffer23} R.\ Baldasso, A.\ Stauffer.
\emph{Local and global survival for infections with recovery.}
Stochastic Process.\ Appl.\ \textbf{160} (2023), 161--173.
\href{https://doi.org/10.1016/j.spa.2023.03.008}{doi:10.1016/j.spa.2023.03.008}.

\bibitem{BaldassoStauffer20} R.\ Baldasso, A.\ Stauffer.
\emph{Local survival of spread of infection among biased random walks.}
Electron.\ J.\ Probab.\ \textbf{27} (2022), paper no.~135, 28 pp.
\href{https://doi.org/10.1214/22-EJP861}{doi:10.1214/22-EJP861}.

\bibitem{BaldassoTeixeira20} R.\ Baldasso, A.\ Teixeira.
\emph{Spread of an infection on the zero range process.}
Ann.\ Inst.\ Henri Poincar\'e Probab.\ Stat.\ \textbf{56} (2020), 1898--1928.
\href{https://doi.org/10.1214/19-AIHP1021}{doi:10.1214/19-AIHP1021}.

\bibitem{DauvergneSly22} D.\ Dauvergne, A.\ Sly.
\emph{The SIR model in a moving population: propagation of infection and herd immunity.}
Comm.\ Pure Appl.\ Math., published online 8 July 2026.
\href{https://doi.org/10.1002/cpa.70054}{doi:10.1002/cpa.70054}.

\bibitem{DauvergneSly23} D.\ Dauvergne, A.\ Sly.
\emph{Spread of infections in a heterogeneous moving population.}
Probab.\ Theory Related Fields \textbf{187} (2023), 73--131.
\href{https://doi.org/10.1007/s00440-023-01216-6}{doi:10.1007/s00440-023-01216-6}.

\bibitem{DGG26} A.\ Drewitz, G.\ Gallo, P.\ Gracar.
\emph{Lipschitz cutset for fractal graphs and applications to the spread of infections.}
Ann.\ Inst.\ Henri Poincar\'e Probab.\ Stat.\ \textbf{62} (2026), 830--878.
\href{https://doi.org/10.1214/24-AIHP1539}{doi:10.1214/24-AIHP1539}.

\bibitem{GrimmettLi22} G.\ R.\ Grimmett, Z.\ Li.
\emph{Brownian snails with removal: epidemics in diffusing populations.}
Electron.\ J.\ Probab.\ \textbf{27} (2022), paper no.~78, 31 pp.
\href{https://doi.org/10.1214/22-EJP804}{doi:10.1214/22-EJP804}.

\bibitem{Hilario2015} M.\ Hil\'ario, F.\ den Hollander, R.\ dos Santos, V.\ Sidoravicius,
A.\ Teixeira.
\emph{Random walk on random walks.}
Electron.\ J.\ Probab.\ \textbf{20} (2015), no.~95, 35 pp.
\href{https://doi.org/10.1214/EJP.v20-4437}{doi:10.1214/EJP.v20-4437}.

\bibitem{FMUV23} L.\ R.\ Fontes, T.\ S.\ Mountford, D.\ Ungaretti, M.\ E.\ Vares.
\emph{Renewal contact processes: phase transition and survival.}
Stochastic Process.\ Appl.\ \textbf{161} (2023), 102--136.
\href{https://doi.org/10.1016/j.spa.2023.03.005}{doi:10.1016/j.spa.2023.03.005}.

\bibitem{GJLV26} P.\ Gracar, B.\ Jahnel, L.\ L\"uchtrath, A.\ D.\ Vu.
\emph{Detection, coverage and percolation in dynamic Boolean models with random radii based on
$\alpha$-stable processes.}
Preprint, 2026. \href{https://arxiv.org/abs/2602.22109}{arXiv:2602.22109}.

\bibitem{GracarStauffer19a} P.\ Gracar, A.\ Stauffer.
\emph{Random walks in random conductances: decoupling and spread of infection.}
Stochastic Process.\ Appl.\ \textbf{129} (2019), 3547--3569.
\href{https://doi.org/10.1016/j.spa.2018.09.016}{doi:10.1016/j.spa.2018.09.016}.

\bibitem{GracarStauffer19b} P.\ Gracar, A.\ Stauffer.
\emph{Multi-scale Lipschitz percolation of increasing events for Poisson random walks.}
Ann.\ Appl.\ Probab.\ \textbf{29} (2019), 376--433.
\href{https://doi.org/10.1214/18-AAP1420}{doi:10.1214/18-AAP1420}.

\bibitem{HUVV22} M.\ Hil\'ario, D.\ Ungaretti, D.\ Valesin, M.\ E.\ Vares.
\emph{Results on the contact process with dynamic edges or under renewals.}
Electron.\ J.\ Probab.\ \textbf{27} (2022), paper no.~91, 31 pp.
\href{https://doi.org/10.1214/22-EJP811}{doi:10.1214/22-EJP811}.

\bibitem{KestenSidoravicius05} H.\ Kesten, V.\ Sidoravicius.
\emph{The spread of a rumor or infection in a moving population.}
Ann.\ Probab.\ \textbf{33} (2005), 2402--2462.
\href{https://doi.org/10.1214/009117905000000413}{doi:10.1214/009117905000000413}.

\bibitem{KestenSidoravicius06} H.\ Kesten, V.\ Sidoravicius.
\emph{A phase transition in a model for the spread of an infection.}
Illinois J.\ Math.\ \textbf{50} (2006), 547--634.
\href{https://doi.org/10.1215/ijm/1258059486}{doi:10.1215/ijm/1258059486}.

\bibitem{KestenSidoravicius08} H.\ Kesten, V.\ Sidoravicius.
\emph{A shape theorem for the spread of an infection.}
Ann.\ of Math.\ (2) \textbf{167} (2008), 701--766.
\href{https://doi.org/10.4007/annals.2008.167.701}{doi:10.4007/annals.2008.167.701}.

\bibitem{LastPenrose18} G.\ Last, M.\ Penrose.
\emph{Lectures on the Poisson Process.}
Institute of Mathematical Statistics Textbooks \textbf{7}, Cambridge University Press, 2018.

\bibitem{MeesterRoy96} R.\ Meester, R.\ Roy.
\emph{Continuum Percolation.}
Cambridge Tracts in Mathematics \textbf{119}, Cambridge University Press, 1996.

\bibitem{PSSS13} Y.\ Peres, A.\ Sinclair, P.\ Sousi, A.\ Stauffer.
\emph{Mobile geometric graphs: detection, coverage and percolation.}
Probab.\ Theory Related Fields \textbf{156} (2013), 273--305.
\href{https://doi.org/10.1007/s00440-012-0428-1}{doi:10.1007/s00440-012-0428-1}.

\bibitem{Penrose96} M.\ D.\ Penrose.
\emph{Continuum percolation and Euclidean minimal spanning trees in high dimensions.}
Ann.\ Appl.\ Probab.\ \textbf{6} (1996), 528--544.
\href{https://doi.org/10.1214/aoap/1034968142}{doi:10.1214/aoap/1034968142}.

\bibitem{PopovTeixeira15} S.\ Popov, A.\ Teixeira.
\emph{Soft local times and decoupling of random interlacements.}
J.\ Eur.\ Math.\ Soc.\ \textbf{17} (2015), 2545--2593.
\href{https://doi.org/10.4171/JEMS/565}{doi:10.4171/JEMS/565}.

\bibitem{QTZ00} J.\ Quintanilla, S.\ Torquato, R.\ M.\ Ziff.
\emph{Efficient measurement of the percolation threshold for fully penetrable discs.}
J.\ Phys.\ A: Math.\ Gen.\ \textbf{33} (2000), L399--L407.
\href{https://doi.org/10.1088/0305-4470/33/42/104}{doi:10.1088/0305-4470/33/42/104}.

\bibitem{vdBMW97} J.\ van den Berg, R.\ Meester, D.\ G.\ White.
\emph{Dynamic Boolean models.}
Stochastic Process.\ Appl.\ \textbf{69} (1997), 247--257.
\href{https://doi.org/10.1016/S0304-4149(97)00044-6}{doi:10.1016/S0304-4149(97)00044-6}.

\end{thebibliography}
\end{document}